\documentclass[10pt,a4paper,final]{article}
\usepackage[utf8x]{inputenc}
\usepackage{ucs}
\usepackage{amsmath}
\usepackage{amsfonts}
\usepackage{amsthm}
\usepackage{amssymb}
\usepackage{graphicx}
\usepackage{hyperref}
\usepackage{mathtools}
\usepackage[margin=1in]{geometry}

\usepackage{xcolor}
\usepackage{bbm}
\usepackage{enumitem}
\usepackage{mathabx}
\usepackage{cases}
\usepackage{stmaryrd}
\makeatletter
\DeclareFontFamily{OMX}{MnSymbolE}{}
\DeclareSymbolFont{MnLargeSymbols}{OMX}{MnSymbolE}{m}{n}
\SetSymbolFont{MnLargeSymbols}{bold}{OMX}{MnSymbolE}{b}{n}
\DeclareFontShape{OMX}{MnSymbolE}{m}{n}{
	<-6>  MnSymbolE5
	<6-7>  MnSymbolE6
	<7-8>  MnSymbolE7
	<8-9>  MnSymbolE8
	<9-10> MnSymbolE9
	<10-12> MnSymbolE10
	<12->   MnSymbolE12
}{}
\DeclareFontShape{OMX}{MnSymbolE}{b}{n}{
	<-6>  MnSymbolE-Bold5
	<6-7>  MnSymbolE-Bold6
	<7-8>  MnSymbolE-Bold7
	<8-9>  MnSymbolE-Bold8
	<9-10> MnSymbolE-Bold9
	<10-12> MnSymbolE-Bold10
	<12->   MnSymbolE-Bold12
}{}

\let\llangle\@undefined
\let\rrangle\@undefined
\DeclareMathDelimiter{\llangle}{\mathopen}%
{MnLargeSymbols}{'164}{MnLargeSymbols}{'164}
\DeclareMathDelimiter{\rrangle}{\mathclose}%
{MnLargeSymbols}{'171}{MnLargeSymbols}{'171}
\makeatother

\makeatletter
\newcommand{\opnorm}{\@ifstar\@opnorms\@opnorm}
\newcommand{\@opnorms}[1]{%
	\left|\mkern-1.5mu\left|\mkern-1.5mu\left|
	#1
	\right|\mkern-1.5mu\right|\mkern-1.5mu\right|
}
\newcommand{\@opnorm}[2][]{%
	\mathopen{#1|\mkern-1.5mu#1|\mkern-1.5mu#1|}
	#2
	\mathclose{#1|\mkern-1.5mu#1|\mkern-1.5mu#1|}
}
\makeatother

\usepackage{imakeidx}
\usepackage[nottoc]{tocbibind}
\makeindex[intoc]

\newtheorem{thm}{Theorem}
\newtheorem{lem}[thm]{Lemma}
\newtheorem{prop}[thm]{Proposition}
\newtheorem{coro}[thm]{Corollary}

\newtheorem{assume}{Assumption}

\newtheorem{algo}{Algorithm}

\theoremstyle{definition}
\newtheorem{defi}{Definition}
\newtheorem{rmk}{Remark}

\newcommand{\sC}{\mathcal{C}}
\newcommand{\sS}{\mathbb{S}}
\newcommand{\sP}{\mathcal{P}}
\newcommand{\sR}{\mathbb{R}}

\newcommand{\dom}{\mathrm{dom}}

\newcommand{\Ball}{\mathrm{Ball}}
\newcommand{\prox}{\mathrm{prox}}

\newcommand{\crit}{\mathrm{crit}}

\DeclareMathOperator*\dist{dist}
\newcommand{\KL}{K\L }
\newcommand{\Loja}{\L ojasiewicz }

\newcommand{\Lr}{\mathcal{L}_{r}}
\newcommand{\Fr}{\mathcal{F}_{r}}
\newcommand{\Fk}{\mathcal{F}_{k}}
\newcommand{\F}{\mathcal{F}}

\newcommand{\M}{\mathbf{M}}
\newcommand{\E}{\mathcal{E}}
\newcommand{\Id}{\mathbf{Id}}
\newcommand{\0}{\mathbf{0}}
\newcommand{\mysum}{\displaystyle \sum\limits}

\title{The proximal alternating direction method of multipliers in the nonconvex setting: convergence analysis and rates}
\author{Radu Ioan Bo\c{t}\thanks{Faculty of Mathematics, University of Vienna, Oskar-Morgenstern-Platz 1, 1090 Vienna, Austria, e-mail: \url{radu.bot@univie.ac.at}. Research partially supported by FWF (Austrian Science Fund), project I 2419-N32.} 
\thanks{Invited Associate Professor, Babe\c s-Bolyai University, Faculty of Mathematics and Computer Sciences, str. Mihail Kog\u alniceanu 1, 400084 Cluj-Napoca, Romania.} \and Dang-Khoa Nguyen \thanks{Faculty of Mathematics, University of Vienna, Oskar-Morgenstern-Platz 1, 1090 Vienna, Austria, e-mail: \url{dang-khoa.nguyen@univie.ac.at}. The author gratefully acknowledges the financial support of the Doctoral Programme {\it Vienna Graduate School on Computational Optimization (VGSCO)} which is funded by Austrian Science Fund  (FWF, project W1260-N35).}}

\begin{document}
	
\maketitle	

%%%%%%%%%%%%%%%%%%%%%%%%%%%%%%%%%%%%%%%%%%%%%%%%%%%%%%%%%%%%%%%%%%%%%%%%%%%%%%%%%%%%%%%%%%%%%%%%%%%%%%%
%
%
%		ABSTRACT
%
%
%%%%%%%%%%%%%%%%%%%%%%%%%%%%%%%%%%%%%%%%%%%%%%%%%%%%%%%%%%%%%%%%%%%%%%%%%%%%%%%%%%%%%%%%%%%%%%%%%%%%%%%

\textbf{Abstract.} We propose two numerical algorithms in the fully nonconvex setting for the minimization of the sum of a smooth function and the composition of a nonsmooth function with a linear operator. The iterative schemes are formulated in the spirit of the proximal alternating direction method of multipliers and its linearized variant, respectively. The proximal terms are introduced via variable metrics, a fact which allows us to derive new proximal splitting algorithms for nonconvex structured optimization problems, as particular 
instances of the general schemes. Under mild conditions on the sequence of variable metrics and by assuming that a regularization of the associated augmented Lagrangian has the Kurdyka-\L ojasiewicz property, we prove that the iterates converge to a KKT point of the objective function. By assuming that the augmented 
Lagrangian has the \L ojasiewicz property, we also derive convergence rates for both the augmented Lagrangian and the iterates. 



\section{Introduction}
\label{sec:intro}

\subsection{Problem formulation and motivation}
\label{subsec:moti}

In this paper we are interested in solving optimization problems of the form
\begin{equation}
\label{intro:problem}
\min\limits_{x \in \sR^{n}} \left\lbrace g \left( A x \right) + h \left( x \right) \right\rbrace,
\end{equation}
where $g \colon \sR^{m} \to \sR \cup \left\lbrace + \infty \right\rbrace$ is a proper and lower semicontinuous function, $h \colon \sR^{n} \to \sR$ is a Fr\'{e}chet differentiable function with $L$-Lipschitz continuous gradient and $A \colon \sR^{n} \to \sR^{m}$ is a linear operator. The spaces $\sR^{n}$ and $\sR^{m}$ are equipped with Euclidean inner products $\left\langle \cdot , \cdot \right\rangle$ and associated norms $\left\lVert \cdot \right\rVert = \sqrt{\left\langle \cdot , \cdot \right\rangle}$, which are both denoted in the same way, as there is no risk of confusion.

We start by briefly describing the Alternating Direction Method of Multipliers (ADMM) designed to solve optimization problems of the form
\begin{equation}
\label{intro:problem-convex}
\min\limits_{x \in \sR^{n}} \left\lbrace f \left( x \right) + g \left( A x \right) + h \left( x \right) \right\rbrace,
\end{equation}
where $g$ and $h$ are assumed to be also {\it convex} and $f \colon \sR^{n} \to \sR \cup \left\lbrace + \infty \right\rbrace$ is another proper, convex and lower semicontinuous function. By introducing an auxiliary variable, one can rewrite problem \eqref{intro:problem-convex} as
\begin{equation}
\label{intro:problem:ADMM}
\min\limits_{\substack{\left( x , z \right) \in \sR^{n} \times \sR^{m} \\ A x - z = \0}} \left\lbrace f \left( x \right) + g \left( z \right) + h \left( x \right) \right\rbrace.
\end{equation}
For a fixed real number $r > 0$, the \emph{augmented Lagrangian} associated with problem \eqref{intro:problem:ADMM} reads
\begin{equation*}
\Lr \colon \sR^{n} \times \sR^{m} \times \sR^{m} \to \sR \cup \left\lbrace + \infty \right\rbrace, \ \Lr \left( x , z , y \right) = f(x) + g \left( z \right) +  h \left( x \right) +  
\left\langle y , A x - z \right\rangle + \dfrac{r}{2} \left\lVert A x - z \right\rVert^{2}.
\end{equation*}
Given a starting vector $\left( x^{0} , z^{0} , y^{0} \right) \in \sR^{n} \times \sR^{m} \times \sR^{m}$ and $\{\M_1^{k} \}_{k \geq 0} \subseteq \sR^{n \times n}, \left\lbrace \M_{2}^{k} \right\rbrace_{k \geq 0} \subseteq \sR^{m \times m}$, two sequences of symmetric and positive semidefinite matrices, 
the following \emph{proximal ADMM algorithm formulated in the presence of a smooth function and involving variable metrics} has been proposed and investigated in \cite{Banert-Bot-Csetnek}:  generate the sequence $\{(x^k, z^k, y^k)\}_{k \geq 0}$ for every $k \geq 0$ as
\begin{subequations}
	\label{intro:algo}
	\begin{align}
	\label{intro:algo:x}
	x^{k+1} & \in \arg\min\limits_{x \in \sR^{n}} \left\lbrace f \left( x \right) + \langle x - x^k, \nabla h(x^k) \rangle  + \dfrac{r}{2} \left\lVert A x - z^{k} + \dfrac{1}{r} y^{k} \right\rVert ^{2} + \frac{1}{2} \left \|x - x^k \right \|_{\M^k_1}^2\right\rbrace ,	\\
	\label{intro:algo:z}
	z^{k+1} & = \arg\min\limits_{z \in \sR^{m}} \left\lbrace g \left( z \right) + \dfrac{r}{2} \left\lVert A x^{k+1} - z + \dfrac{1}{r} y^{k} \right\rVert ^{2} + \frac{1}{2} \left \|z - z^k \right \|_{\M^k_2}^2 \right\rbrace , \\
	\label{intro:algo:y}
	y^{k+1} & = y^{k} + \rho r \left( A x^{k+1} - z^{k+1} \right) .
	\end{align}	
\end{subequations}
It has been proved in \cite{Banert-Bot-Csetnek} that, if $\rho=1$ and the set of the saddle points of the Lagrangian associated with \eqref{intro:problem:ADMM}  (which is nothing else than $\Lr$ when $r=0$) is nonempty, and the two matrix sequences and the operator $A$ fulfill mild additional assumptions, 
then the sequence $\{(x^k, z^k, y^k)\}_{k \geq 0}$ converges to a saddle point of the Lagrangian associated with problem \eqref{intro:problem:ADMM} and provides in this way both an optimal solution of \eqref{intro:problem}  and an optimal solution of its Fenchel dual problem. 
Furthermore, an ergodic primal-dual gap convergence rate result has been proved.

In case $h=0$, the above iterative scheme encompasses as special cases different numerical algorithms considered in the literature. If $\M_{1}^{k} = \M_{2}^{k} = 0$ for all $k \geq 0$, then \eqref{intro:algo:x}-\eqref{intro:algo:y} becomes the \emph{classical ADMM algorithm} 
(\cite{Boyd-et.al, Fortin-Glowinski, Gabay, Gabay-Meicer}), which lately gained a huge popularity in the optimization community, despite its poor implementation properties caused by the fact that, in general, the calculation of the sequence of primal variables $\left\lbrace x^{k} \right\rbrace_{k \geq 0}$ 
does not correspond to a proximal step. For an \emph{inertial version} of the classical ADMM algorithm we refer the reader to \cite{Bot-Csetnek-iADMM}. 
On the other hand, if $\M_{1}^{k} = \M_{1}$ and  $\M_{2}^{k} = \M_{2}$ for all $k \geq 0$, then \eqref{intro:algo:x}-\eqref{intro:algo:y} recovers the \emph{proximal ADMM algorithm} investigated by Shefi and Teboulle in \cite{Shefi-Teboulle} (see also \cite{Cui-Li-Sun-Toh, Fazel-Pong-Sun-Tseng}). 
It has been pointed out in \cite{Shefi-Teboulle} that, for suitable choices of the matrices $\M_1$ and $\M_2$, the proximal ADMM algorithm becomes a primal-dual splitting algorithm in the sense of those considered in \cite{Bot-Csetnek-Heinrich, Chambolle-Pock, Condat, Vu}, and which, due to its full splitting character, 
overcomes the drawbacks of the classical ADMM algorithm. Recently, in \cite{Bot-Csetnek-ADMM} it has been shown that, if $f$ is strongly convex, then suitable choices of the non-constant sequences $\left\lbrace \M_{1}^{k} \right\rbrace_{k \geq 0}$ and $\left\lbrace \M_{2}^{k} \right\rbrace_{k \geq 0}$ lead to a rate of convergence of $\mathcal{O} \left( 1/k \right)$  for the sequence of 
primal iterates.

In this paper, we propose a proximal ADMM (P-ADMM) algorithm  and a proximal linearized ADMM (PL-ADMM) algorithm for solving the optimization problem \eqref{intro:problem} and carry out a convergence analysis for both algorithms. 
We first prove, under not very restrictive assumptions on the problem data, that the sequence of generated iterates $\{(x^k, z^k, y^k)\}_{k \geq 0}$ is bounded. Given these premises we show that the cluster points of $\{(x^k, z^k, y^k)\}_{k \geq 0}$ are {\it KKT points} of the problem \eqref{intro:problem}. Provided that a regularization of the augmented Lagrangian satisfies the Kurdyka-\L ojasiewicz property, we show global convergence of the generated sequence of iterates.
Provided this regularization of the augmented Lagrangian has the \Loja property, we derive rates of convergence for the sequence of iterates. To the best of our knowledge, these are the first results in the literature that deal with convergence rates for the nonconvex ADMM.

In the following we will comment on previous works addressing the ADMM algorithm in the nonconvex setting. None of the papers which have addressed nonconvex optimization problems involving compositions with linear operators propose and investigate iterative schemes designed in the 
spirit of full splitting algorithms. In \cite{Li-Pong}, the convergence of the ADMM algorithm for solving the problem \eqref{intro:problem} is studied under the assumption that $h$ is twice continuously differentiable with bounded Hessian. 
In \cite{Hong-Lou-Razaviyayn}, the ADMM algorithm is used to minimize the sum of finitely many smooth nonconvex functions and a nonsmooth convex function, by rewriting it as an general consensus problem. No linear operator occurs in the formulation of the optimization problem under investigation. In \cite{Ames-Hong}, 
the ADMM algorithm is used to solve a DC optimization problem over the unit ball which occurs in the penalized zero-variance linear discriminant analysis.  In \cite{Wang-Xu-Xu}, a nonconvex ADMM algorithm involving proximal terms induced via Bregman distances is introduced and investigated, however, without addressing the question
of the boundedness of the generated iterates. On the other hand, in \cite{Guo-Han-Wu}, in order to guarantee boundedness of the iterates a strong assumption on $g$ is made, which is proved to hold for the normed-squared function. In \cite{Wang-Yin-Zeng}, a lot of efforts are made to guarantee boundedness for the generated iterates of the nonconvex ADMM algorithm, 
which is an essential component  of the convergence analysis, however, this is done by assuming that the objective function is continuous and \emph{coercive over the feasible set}, while its nonsmooth part is either \emph{restricted prox-regular} or \emph{piecewise linear}. Similar ingredients are used in \cite{Liu-Shen-Gu} in the convergence analysis of a nonconvex linearized ADMM algorithm.

Recently, Bolte, Sabach and Teboulle have proposed in \cite{Bolte-Sabach-Teboulle:MOOR} a generic iterative scheme for solving a general optimization problem of the form \eqref{intro:problem}, but by replacing the linear operator $A$ with a general continuously differentiable operator. A global convergence analysis relying on the use of the Kurdyka-\L ojasiewicz property is carried out. It is also shown that the generic iterative scheme encompasses several Lagrangian based algorithms, including the proximal alternating  direction method of multipliers and the proximal alternating linearized minimization method. The latter is analysed into detail in the particular case when $g$ is composed with a linear operator, which coincides with the one in this paper.  The two algorithms we propose in this paper are formulated in the same spirit, however,  they lead for some particular choices of the variable metrics to full splitting algorithms. In addition, we carefully address the issue of the boundedness of the sequence of generated iterates and complement the convergence analysis with the derivation of convergence rates.

The major strengths of our paper are:
\begin{enumerate}[label=\arabic*., leftmargin=\parindent]
\item We prove under quite general assumptions that the sequence $\{(x^k, z^k, y^k)\}_{k \geq 0}$ is bounded. In the nonconvex setting, the boundedness of the sequence of generated iterates plays a central role in the convergence analysis. In fact, the reason, why we assume  in this paper that the function $g$ is smooth, is exclusively given by the fact that only in this setting we can prove boundedness of this sequence under general assumptions.

\item We prove convergence for relaxed variants of the nonconvex ADMM algorithms, which allow to chose in the update of the dual sequence $\rho \in \left(0, 2 \right)$. We notice that $\rho = 1$ is the standard choice in the literature (\cite{Ames-Hong,Banert-Bot-Csetnek,Bot-Csetnek-ADMM,Li-Pong,Shefi-Teboulle,Wang-Yin-Zeng}).
Gabay and Mercier proved in \cite{Gabay-Meicer} in the convex setting that $\rho$ may be chosen in $\left(0, 2 \right)$, however, the majority of the extensions of the convex relaxed ADMM algorithm assume that $\rho \in \left( 0 , \frac{1 + \sqrt{5}}{2} \right)$ (see \cite{Cui-Li-Sun-Toh,Fazel-Pong-Sun-Tseng,Gabay,Sun-Toh-Yang,Xu-Wu,Yang-Pong-Chen}) or ask for a particular choice of $\rho$, which is interpreted as a step size (see \cite{Hong-Luo}). In \cite{Yang-Pong-Chen}, an alternating minimization algorithm for the minimization of the sum of a simple nonsmooth function and a smooth function  in the nonconvex setting, which allows for a parameter $\rho$ different from $1$, has been proposed. 
	
\item By appropriate choices of the matrix sequences, we derive from the proposed iterative schemes full splitting algorithms for solving the nonconvex complexly structured optimization problem \eqref{intro:problem}. More precisely, (P-ADMM) gives rise to an iterative scheme formulated only in terms of proximal steps for the functions $g$ and $h$ and of forward evaluations of the matrix $A$, while (PL-ADMM) gives rise to an iterative scheme in which the function $h$ is performed via a gradient step. Exact formulas for proximal operators are available not only for large classes of convex functions (\cite{Beck,Combettes-Wajs}), but also of nonconvex functions (\cite{Attouch-Bolte-Redont-Soubeyran, Hare-Sagastizabal, Lewis-Malick}). The fruitful idea to linearize the step involving the smooth term has been used in the past in the context of ADMM algorithms mostly in the convex setting (see \cite{Lin-Liu-Li,Ouyang-Chen-Lan-Pasiliao,Ren-Lin,Xu-Wu,Yang-Yuan}), but also in the nonconvex setting (see \cite{Bolte-Sabach-Teboulle:MOOR, Liu-Shen-Gu}).
\end{enumerate}

\subsection{Notations and preliminaries}\label{subsec:nota}

Let $N$ be a strictly positive integer. We denote by $\mathbbm{1}:= \left( 1, \ldots , 1 \right) \in \sR^{N}$ and write for $x := \left( x_{1}, \ldots , x_{N} \right)$, $y := \left( y_{1}, \ldots , y_{N} \right) \in \sR^{N}$
\begin{equation*}
x < y \ \textrm{ if and only if } \ x_{i} < y_{i} \ \forall i = 1 , \ldots , N.
\end{equation*}

We endow the Cartesian product $\sR^{N_{1}} \times \sR^{N_{2}} \times \ldots \times \sR^{N_{p}}$, where $p$ is a strictly positive integer, with \emph{inner product} and associated \emph{norm} defined for $u := \left( u_{1} , \ldots , u_{p} \right), u' := \left( u_{1}' , \ldots , u_{p}' \right) 
\in \sR^{N_{1}} \times \sR^{N_{2}} \times \ldots \times \sR^{N_{p}}$ by
\begin{equation*}
\left\llangle u, u' \right\rrangle = \mysum_{i = 1}^{p} \left\langle u_{i} , u_{i}' \right\rangle \quad \textrm{ and } \quad \opnorm{u} = \sqrt{\mysum_{i = 1}^{p} \left\lVert u_{i} \right\rVert ^{2}},
\end{equation*}
respectively. For every $u := \left( u_{1} , \ldots , u_{p} \right), u' := \left( u_{1}' , \ldots , u_{p}' \right) \in \sR^{N_{1}} \times \sR^{N_{2}} \times \ldots \times \sR^{N_{p}}$ we have
\begin{equation}
\label{intro:norm-inequality}
\dfrac{1}{\sqrt{p}} \mysum_{i = 1}^{p} \left\lVert u_{i} \right\rVert \leq \opnorm{u} = \sqrt{\mysum_{i = 1}^{p} \left\lVert u_{i} \right\rVert ^{2}} \leq \mysum_{i = 1}^{p} \left\lVert u_{i} \right\rVert .
\end{equation}

We denote by $\sS^{N}_{+}$ the family of symmetric and positive semidefinite matrices $M \in \sR^{N \times N}$. Every $M \in \sS^{N}_{+}$ induces a \emph{semi-norm} defined by
\begin{equation*}
\left\lVert x \right\rVert _{M}^{2} := \left\langle M x , x \right\rangle \ \forall x \in \sR^{N} .
\end{equation*}

The \emph{Loewner partial ordering} on $\sS^{N}_{+}$ is defined for $M , M ' \in \sS^{N}_{+}$ as
\begin{equation*}
M \succcurlyeq M ' \Leftrightarrow \left\lVert x \right\rVert _{M}^{2} \geq \left\lVert x \right\rVert _{M '}^{2}  \ \forall x \in \sR^{N} .
\end{equation*}
Thus $M \in \sS^{N}_{+}$ is nothing else than $M \succcurlyeq \0$.
For $\alpha > 0$ we set
\begin{equation*}
\sP^{N}_{\alpha} := \left\lbrace M \in \sS^{N}_{+} \colon M \succcurlyeq \alpha \Id \right\rbrace,
\end{equation*}
where $\Id$ denotes the identity matrix in $\sR^{N \times N}$. 
If $M \in \sP^{N}_{\alpha}$, then the semi-norm $\left\lVert \cdot \right\rVert _{M}$ becomes a norm.

The linear operator $A$ is \emph{surjective} if and only if its associated matrix has full row rank. This assumption is further equivalent to the fact that the matrix associated to $A A^{*}$, where $A^*$ denotes the \emph{adjoint operator} of $A$, is positively definite. Since
\begin{equation*}
\lambda_{\min} \left( A A^{*} \right) \left\lVert y \right\rVert ^{2} \leq \left\lVert y \right\rVert ^{2}_{A A^{*}} = \langle A A^{*} y, y \rangle = \left\lVert A^{*} y \right\rVert ^{2} \ \forall y \in \sR^{m},
\end{equation*}
this is further equivalent to $\lambda_{\min} \left( A A^* \right) > 0$ (and  $A A^{*} \in \sP^{n}_{\lambda_{\min} \left( A A^{*} \right)}$), where $\lambda_{\min} (\cdot)$ denotes the smallest eigenvalue of a matrix.
Similarly, $A$ is injective if and only if $\lambda_{\min} \left( A^* A \right) > 0$ (and  $A^* A \in \sP^{m}_{\lambda_{\min} \left( A^* A\right)}$).

\begin{prop}
	\label{prop:lips-semiconvex}
	Let $\Psi \colon \sR^{N} \to \sR$ be Fr\'echet differentiable such that its gradient is Lipschitz continuous with constant $L > 0$. Then the following statements are true: 
	\begin{enumerate}
		\item For every $x,y \in \sR^{N}$ and every $z \in [x,y] = \{(1-t)x + ty : t \in [0,1] \}$ it holds
		\begin{equation}
		\label{eq:semi-convex}
		\Psi \left( y \right) \leq \Psi \left( x \right) + \left\langle \nabla \Psi \left( z \right) , y - x \right\rangle + \dfrac{L}{2} \left\lVert y - x \right\rVert ^{2};
		\end{equation}

		\item If $\Psi$ is bounded from below, then for every $\sigma > 0$ it holds
		\begin{equation*}
		\inf\limits_{x \in \sR^{N}} \left\lbrace \Psi \left( x \right) - \left( \dfrac{1}{\sigma} - \dfrac{L}{2 \sigma^{2}} \right) \left\lVert \nabla \Psi \left( x \right) \right\rVert ^{2} \right\rbrace > - \infty .
		\end{equation*}
	\end{enumerate}	
\end{prop}

\begin{proof}
	\begin{enumerate}
		\item Let be $x , y \in \sR^{N}$ and $z:= (1-t)x + ty$ for $t \in [0,1]$. By the fundamental theorem for line integrals we have
		\begin{align}
		\label{eq:Fund-line}
		\Psi \left( y \right) - \Psi \left( x \right) & = \int_{0}^{1} \left\langle \nabla \Psi \left((1-s)x + sy \right) , y - x \right\rangle ds \nonumber \\
		& = \int_{0}^{1} \left\langle \nabla \Psi \left((1-s)x + sy \right) - \nabla \Psi \left( z \right) , y - x \right\rangle ds + 
		\left\langle \nabla \Psi \left( z \right) , y - x \right\rangle .
		\end{align}
		Since
		\begin{align}
		\label{eq:CS-integ}
		& \left\lvert \int_{0}^{1} \left\langle \nabla \Psi \left((1-s)x + sy  \right) - \nabla \Psi \left( z \right) , y - x \right\rangle ds \right\rvert \nonumber \\
		\leq & \ \int_{0}^{1} \left\lVert \nabla \Psi \left((1-s)x + sy  \right) - \nabla \Psi \left( z \right) \right\rVert \cdot \left\lVert y - x \right\rVert ds \leq L \left\lVert x - y \right\rVert ^{2} \int_{0}^{1} \left\lvert s - t \right\rvert ds \nonumber \\
		= & \ L \left\lVert x - y \right\rVert ^{2} \left( \int_{0}^{t} \left( - s + t\right) ds + \int_{t}^{1} \left( s - t \right) ds \right) = L \left( \dfrac{1}{2} - t \left( 1 - t\right) \right) \left\lVert x - y \right\rVert ^{2} .
		\end{align}
		The inequality in \eqref{eq:semi-convex} follows by combining \eqref{eq:Fund-line} and \eqref{eq:CS-integ} and by using that $0 \leq t \leq 1$.
		
		\item The inequality in \eqref{eq:semi-convex} gives for every $x \in \sR^{N}$
		\begin{align*}
		- \infty < \inf\limits_{y \in \sR^{N}} \Psi \left( y \right) & \leq \Psi \left( x - \dfrac{1}{\sigma} \nabla \Psi \left( x \right) \right) \\	
		& \leq \Psi \left( x \right) + \left\langle \left( x - \dfrac{1}{\sigma} \nabla \Psi \left( x \right) \right) - x , \nabla \Psi \left( x \right) \right\rangle + \dfrac{L}{2} \left\lVert \left( x - \dfrac{1}{\sigma} \nabla \Psi \left( x \right) \right) - x \right\rVert ^{2} \\
		& = \Psi \left( x \right) - \left( \dfrac{1}{\sigma} - \dfrac{L}{2 \sigma^{2}} \right) \left\lVert \nabla \Psi \left( x \right) \right\rVert ^{2},
		\end{align*}
		which leads to the desired conclusion.
		\qedhere
	\end{enumerate}
\end{proof}

\begin{rmk}\label{descent}
The so-called Descent Lemma, which says that for a Fr\'echet differentiable function $\Psi \colon \sR^{N} \to \sR$ having Lipschitz continuous gradient with constant $L > 0$ it holds
\begin{equation*}
	\Psi \left( y \right) \leq \Psi \left( x \right) + \left\langle \nabla \Psi \left( x \right) , y - x \right\rangle + \dfrac{L}{2} \left\lVert y - x \right\rVert ^{2} \quad \forall x,y \in \sR^N,
\end{equation*}
follows from statement (i) of the above proposition for $z:=x$.
	
Moreover, for $z:=y$ we have that
\begin{equation*}
\Psi \left( x \right) \geq \Psi \left( y \right) + \left\langle \nabla \Psi \left( y \right) , x - y \right\rangle - \dfrac{L}{2} \left\lVert x - y \right\rVert ^{2} \quad \forall x,y \in \sR^N,
\end{equation*}
which is equivalent to the fact that $\Psi + \dfrac{L}{2} \left\lVert \cdot \right\rVert ^{2}$ is a convex function, in other words, $\Psi$ is a $L$-semiconvex function (\cite{Bolte-Daniilidis-Ley-Mazet}). It follows from the previous result that a  Fr\'echet differentiable function with $L$-Lipschitz continuous gradient is $L$-semiconvex.
\end{rmk}

The \emph{limiting subdifferential} will play an important role in the convergence analysis we are going to carry out for the nonconvex ADMM algorithm. Let $\Psi \colon \sR^{N} \to \sR \cup \left\lbrace + \infty \right\rbrace$ be 
a proper and lower semicontinuous function. For $x \in \dom \Psi := \left\lbrace x \in \sR^{N} \colon \Psi \left( x \right) < + \infty \right\rbrace$, the \emph{Fr\'{e}chet (viscosity) subdifferential} of $\Psi$ at $x$ is 
\begin{equation*}
\widehat{\partial} \Psi \left( x \right) := \left\lbrace d \in \sR^{N} \colon \liminf\limits_{y \to x} \dfrac{\Psi \left( y \right) - \Psi \left( x \right) - \left\langle d , y - x \right\rangle}{\left\lVert y - x \right\rVert} \geq 0 \right\rbrace
\end{equation*}
and the \emph{limiting (Mordukhovich) subdifferential} of $\Psi$ at $x$ is
\begin{align*} 
\partial \Psi \left( x \right) := \{ d \in \sR^{N} \colon  & \mbox {exist sequences} \ x^{k} \to x \ \mbox{and} \  d^{k} \to d \textrm{ as } k \to +\infty\\
& \mbox{such that} \ \Psi \left( x^{k} \right) \to \Psi \left( x \right) \textrm{ as } k \to +\infty \textrm{ and } d^{k} \in \widehat{\partial} \Psi \left( x^{k} \right) \ \mbox{for all} \ k \geq 0 \}.
\end{align*}
For $x \notin \dom \left( \Psi \right)$, we set $\widehat{\partial} \Psi \left( x \right) = \partial \Psi \left( x \right) := \emptyset$.

The inclusion $\widehat{\partial} \Psi \left( x \right) \subseteq \Psi \left( x \right)$ holds for every $x \in \sR^{N}$. If $\Psi$ is convex,then these two subdifferentials coincide with the \emph{convex subdifferential}, in other words
\begin{equation*}
\widehat{\partial} \Psi \left( x \right) = \partial \Psi \left( x \right) = \left\lbrace d \in \sR^{N} \colon \Psi \left( y \right) \geq \Psi\left( x \right) + \left\langle d , y - x \right\rangle \ \forall y \in \sR^{N} \right\rbrace \ \mbox{for all} \ x \in \dom \Psi.
\end{equation*}
If $x \in \sR^{N}$ is a local minimum of $\Psi$, then $0 \in \partial \Psi \left( x \right)$. We denote by $\crit (\Psi) = \{x \in \sR^N : 0 \in \partial \Psi \left( x \right) \}$ the set of \emph{critical points} of $\Psi$. The limiting subdifferential fulfills the \emph{closedness criterion}: if $\left\lbrace x^{k} \right\rbrace _{k \geq 0}$ and $\{d^k\}_{k \geq 0}$ are sequence in $\sR^{N}$ such that $d^{k} \in \partial \Psi \left( x^{k} \right)$ for all $k \geq 0$, and $\left( x^{k} , d^{k}  \right) \to \left( x , d \right)$ and $\Psi \left( x^{k} \right) \to \Psi \left( x \right)$ as $k \to +\infty$, then $d \in \partial \Psi \left( x \right)$. We have the following \emph{subdifferential sum rule} holds (\cite[Proposition 1.107]{Mordukhovich}, \cite[Exercise 8.8]{Rockafellar-Wets}): if $\Phi \colon \sR^{N} \to \sR$ is a continuously differentiable function, then $\partial \left( \Psi + \Phi \right) \left( x \right) = 
\partial \Psi \left( x \right) + \nabla \Phi \left( x \right)$ for all $x \in \sR^{N}$; and the following \emph{subdifferential rule for the composition with a linear operator $A \colon \sR^{N'} \to \sR^{N}$} (\cite[Proposition 1.112]{Mordukhovich}, \cite[Exercise 10.7]{Rockafellar-Wets}): 
if $x \in \dom \Psi$ and $A$ is injective, then $\partial \left( \Psi \circ A \right) \left( x \right) = A^{*} \partial \Psi \left( A x \right)$. 

We close this section by presenting two convergence results for real sequences that will be used in the sequel in the convergence analysis. The next lemma is often used in the literature when proving convergence of numerical algorithms relying on Fej\'er monotonicity techniques (see, for instance, 
\cite[Lemma 2.2]{Bot-Csetnek}, \cite[Lemma 2]{Bot-Csetnek-Laszlo}).
\begin{lem}
	\label{lem:conv-pre}
	Let $\left\lbrace b_{k} \right\rbrace _{k \geq 0}$ be a sequence in $\sR$ and $\left\lbrace \xi_{k} \right\rbrace _{k \geq 0}$ a sequence in $\sR_{+}$. Assume that $\left\lbrace b_{k} \right\rbrace _{k \geq 0}$ is bounded from below and that for every $k \geq 0$
	\begin{equation*}
	b_{k+1} + \xi_{k} \leq b_{k}.
	\end{equation*}
	Then the following statements hold:
	\begin{enumerate}
		\item the sequence $\left\lbrace \xi_{k} \right\rbrace _{k \geq 0}$ is summable, namely $\mysum_{k \geq 0} \xi_{k} < + \infty$;
		\item the sequence $\left\lbrace b_{k} \right\rbrace _{k \geq 0}$ is monotonically decreasing and convergent.
	\end{enumerate}
\end{lem}

The following lemma, which is an extension of \cite[Lemma 2.3]{Bot-Csetnek} (see, also \cite[Lemma 3]{Bot-Csetnek-Laszlo}), is of interest by its own. 

\begin{lem}
	\label{lem:conv-ext}
	Let $\left\lbrace a^{k} := \left( a_{1}^{k} , a_{2}^{k} , \ldots , a_{N}^{k} \right) \right\rbrace _{k \geq 0}$ be a sequence in $\sR^N_{+}$ and $\left\lbrace \delta_{k} \right\rbrace _{k \geq 0}$ a sequence in $\sR$ such that
	\begin{equation}
	\label{sum:hypo}
	\left\langle \mathbbm{1} , a^{k+1} \right\rangle \leq \left\langle c_{0} , a^{k} \right\rangle + \left\langle c_{1} , a^{k-1} \right\rangle + \left\langle c_{2} , a^{k-2} \right\rangle + \delta_{k} \  \forall k \geq 2,
\end{equation}
	where $c_{0} := \left( c_{0,1} , c_{0,2} , \ldots , c_{0,N} \right) \in \sR^{N}$, $c_{1} := \left( c_{1,1} , c_{1,2} , \ldots , c_{1,N} \right) \in \sR_{+}^{N}$ and $c_{2} := \left( c_{2,1} , c_{2,2} , \ldots , c_{2,N} \right) \in \sR_{+}^{N}$ fulfill $c_{0} + c_{1} + c_{2} < \mathbbm{1}$. 
	Assume further that there exists $\widebar{\delta} \geq 0$ such that for every $\overline{K} \geq \underline{K} \geq 2$
	\begin{equation*}
	\mysum_{k = \underline{K}}^{\overline{K}} \delta_{k} \leq \widebar{\delta} .
	\end{equation*}
	Then, for every $i = 1, \ldots, N$, it holds
	\begin{equation*}
	\mysum_{k \geq 0} a_{i}^{k} < + \infty.
	\end{equation*}
	In particular, for every $i = 1, \ldots, N$  and every $\overline{K} \geq \underline{K} \geq 2$, it holds
	\begin{equation}
	\label{sum:upper}
	\mysum_{k = \underline{K}}^{\overline{K}} a_{i}^{k} 
	\leq \dfrac{\mysum_{j=1}^{N} \left[ \left( 1 - c_{0,j} - c_{1,j} \right) a_{j}^{\underline{K}} + \left( 1 - c_{0,j} \right) a_{j}^{\underline{K}+1} + a_{j}^{\underline{K}+2} \right] +  \widebar{\delta}}{1 - c_{0,i} - c_{1,i} - c_{2,i}} .
	\end{equation}
\end{lem}

\begin{proof}
	Fix $\overline{K} \geq \underline{K} \geq 2$. If $\overline{K} = \underline{K}$ or $\overline{K} = \underline{K} + 1$, then \eqref{sum:upper} holds automatically. Assume now that $\overline{K} \geq \underline{K} + 2$. Summing up the inequality in \eqref{sum:hypo} for $k =  \underline{K} + 2 , \cdots , \overline{K}$, we obtain
	\begin{equation}
	\label{sum:-up}
	\left\langle \mathbbm{1}, \mysum_{k=\underline{K}+2}^{\overline{K}} a^{k+1} \right\rangle 
	\leq \left\langle c_{0} , \mysum_{k=\underline{K}+2}^{\overline{K}} a^{k} \right\rangle 
	+ \left\langle c_{1} , \mysum_{k=\underline{K}+2}^{\overline{K}} a^{k-1} \right\rangle 
	+ \left\langle c_{2} , \mysum_{k=\underline{K}+2}^{\overline{K}} a^{k-2} \right\rangle 
	+ \mysum_{k=\underline{K}+2}^{\overline{K}} \delta_{k} .	
	\end{equation}	
	Since
	\begin{align*}
	\mysum_{k = \underline{K}+2}^{\overline{K}} a^{k+1} 	& 
	= \mysum_{k = \underline{K}+3}^{\overline{K}+1} a^{k} 
	= \mysum_{k = \underline{K}}^{\overline{K}} a^{k} + a^{\overline{K}+1} - a^{\underline{K}} - a^{\underline{K} + 1} - a^{\underline{K}+2}\\
	\mysum_{k = \underline{K}+2}^{\overline{K}} a^{k} 		& 
	= \mysum_{k = \underline{K}}^{\overline{K}} a^{k} - \left( a^{\underline{K}} + a^{\underline{K}+1} \right) \\
	\mysum_{k = \underline{K}+2}^{\overline{K}} a^{k-1} 	& 
	= \mysum_{k = \underline{K} + 1}^{\overline{K}-1} a^{k} 
	= \mysum_{k = \underline{K}}^{\overline{K}} a^{k} - \left( a^{\underline{K}} + a^{\overline{K}} \right) \\
	\mysum_{k = \underline{K}+2}^{\overline{K}} a^{k-2} 	& 
	= \mysum_{k = \underline{K}}^{\overline{K} - 2} a^{k} 
	= \mysum_{k = \underline{K}}^{\overline{K}} a^{k} - \left( a^{\overline{K}-1} + a^{\overline{K}} \right) , 
	\end{align*}
	the inequality in \eqref{sum:-up} can be rewritten as
	\begin{align*}
	& \left\langle \mathbbm{1} , \mysum_{k=\underline{K}}^{\overline{K}} a^{k} \right\rangle 
	+ \left\langle \mathbbm{1} , a^{\overline{K}+1} - a^{\underline{K}} - a^{\underline{K}+1} - a^{\underline{K}+2} \right\rangle 
	\leq \left\langle c_{0} , \mysum_{k=\underline{K}}^{\overline{K}} a^{k} \right\rangle 
	- \left\langle c_{0} , a^{\underline{K}} + a^{\underline{K}+1} \right\rangle \\
	& \ + \left\langle c_{1} , \mysum_{k=\underline{K}}^{\overline{K}} a^{k} \right\rangle 
	- \left\langle c_{1} , a^{\underline{K}} + a^{\overline{K}} \right\rangle
	+ \left\langle c_{2} , \mysum_{k=\underline{K}}^{\overline{K}} a^{k} \right\rangle
	- \left\langle c_{2} , a^{\overline{K}-1} + a^{\overline{K}} \right\rangle
	+ \mysum_{k=\underline{K}+2}^{\overline{K}} \delta_{k},
	\end{align*}
	which further implies
	\begin{align*}
	\begin{split}
	\mysum_{j=1}^{N} \left[ \left( 1 - c_{0,j} - c_{1,j} - c_{2,j} \right) \mysum_{k=\underline{K}}^{\overline{K}} a_{j}^{k} \right] & = \left\langle \mathbbm{1} - c_{0} - c_{1} - c_{2} , \mysum_{k = \underline{K}}^{\overline{K}} a^{k} \right\rangle \\
	& \leq \left\langle \mathbbm{1} - c_{0} - c_{1} , a^{\underline{K}} \right\rangle + \left\langle \mathbbm{1} - c_{0} , a^{\underline{K}+1} \right\rangle + \left\langle \mathbbm{1}, a^{\underline{K}+2} \right\rangle + \mysum_{k=\underline{K}+2}^{\overline{K}} \delta_{k} \\
	& = \mysum_{j=1}^{N} \left[ \left( 1 - c_{0,j} - c_{1,j} \right) a_{j}^{\underline{K}} + \left( 1 - c_{0,j} \right) a_{j}^{\underline{K}+1} + a_{j}^{\underline{K}+2} \right] + \mysum_{k=\underline{K}+2}^{\overline{K}} \delta_{k}.
	\end{split}
	\end{align*}
Hence, for every $i = 1 , \ldots , N$, it holds
	\begin{equation*}
	\left( 1 - c_{0,i} - c_{1,i} - c_{2,i} \right) \mysum_{k=\underline{K}}^{\overline{K}} a_{i}^{k} 
	\leq \mysum_{j=1}^{N} \left[ \left( 1 - c_{0,j} - c_{1,j} \right) a_{j}^{\underline{K}} + \left( 1 - c_{0,j} \right) a_{j}^{\underline{K}+1} + a_{j}^{\underline{K}+2} \right] +  \widebar{\delta}
	\end{equation*}
	and the conclusion follows by taking into consideration that $c_{0} + c_{1} + c_{2} < \mathbbm{1}$.
\end{proof}

\section{A proximal ADMM algorithm and a proximal linearized ADMM algorithm in the nonconvex setting}\label{sec:main}

%%%%	CONSTANTS

%%	Active

\newcommand{\Cact}{C_{2}}

\newcommand{\Cmu}{C_{\mu}}
\newcommand{\Tdecy}{T_{1}}
\newcommand{\Tdecx}{T_{0}}

%% 	C + thm/lem + var + a/b
% 	a : proximal
% 	b : linearized

%%	Decreasing

\newcommand{\Cdecxx}{C_{1}}
\newcommand{\Cdecxa}{C_{1,1}}
\newcommand{\Cdecxb}{C_{1,2}}

\newcommand{\Cdecx}{C_{0}}
\newcommand{\Cdecxxa}{C_{0,1}}
\newcommand{\Cdecxxb}{C_{0,2}}

%%	Estimate

\newcommand{\Cesty}{T_{2}}

\newcommand{\Cestx}{C_{3}}

\newcommand{\Cestxx}{C_{4}}

%%	Subdifferential of augmented Lagrangian

\newcommand{\CaLrx}{C_{5}}
\newcommand{\CaLrz}{C_{6}}
\newcommand{\CaLry}{C_{7}}

%%	Subdifferential of auxillary function

\newcommand{\CaHkx}{C_{8}}
\newcommand{\CaHkz}{C_{9}}
\newcommand{\CaHky}{C_{10}}

%%	Subdifferential of auxillary function - corollary

\newcommand{\Csubx}{C_{11}}
\newcommand{\Csuby}{C_{12}}
\newcommand{\Csubyy}{C_{13}}

%%	Error estimates

\newcommand{\Crec}{C_{23}}
\newcommand{\Citex}{C_{24}}
\newcommand{\Citey}{C_{25}}
\newcommand{\Citez}{C_{26}}

%%	Subdifferential of auxillary function - better

\newcommand{\CaHkbx}{C_{14}}
\newcommand{\CaHkby}{C_{15}}
\newcommand{\CaHkbyy}{C_{16}}

%%	Error estimates - better

\newcommand{\Clower}{C_{17}}
\newcommand{\Cupper}{C_{18}}
\newcommand{\Crecb}{C_{19}}
\newcommand{\Citebx}{C_{20}}
\newcommand{\Citeby}{C_{21}}
\newcommand{\Citebz}{C_{22}}

In this section we propose two proximal ADMM algorithms for solving the optimization problem \eqref{intro:problem} and study their convergence behaviour. A central role will be played by the augmented Lagrangian associated with
problem \eqref{intro:problem}, which is defined for every $r > 0$ as
\begin{equation*}
\Lr \colon \sR^{n} \times \sR^{m} \times \sR^{m} \to \sR \cup \left\lbrace + \infty \right\rbrace, \ \Lr \left( x , z , y \right) = g \left( z \right) +  h \left( x \right) +  
\left\langle y , A x - z \right\rangle + \dfrac{r}{2} \left\lVert A x - z \right\rVert^{2}.
\end{equation*}

\subsection{General formulations and full proximal splitting algorithms as particular instances}
\label{subsec:formula}

\begin{algo}
	\label{algo:al.1}	
	Let be the matrix sequences $\left\lbrace \M_{1}^{k} \right\rbrace_{k \geq 0} \in \sS^{n}_{+}$ , $\left\lbrace \M_{2}^{k} \right\rbrace_{k \geq 0} \in \sS^{m}_{+}$, $r > 0$ and $0 < \rho < 2$. For a given starting vector $\left( x^{0} , z^{0} , y^{0} \right) \in \sR^{n} \times \sR^{m} \times \sR^{m}$, generate the sequence $\left\lbrace \left( x^{k} , z^{k} , y^{k} \right) \right\rbrace _{k \geq 0}$ for every $k \geq 0$ as:
	\begin{subequations} 
		\begin{align}		
		\begin{split}
		\label{al.1:algo:z}
		z^{k+1} & \in \arg\min\limits_{z \in \sR^{m}} \left\lbrace \Lr \left( x^{k} , z , y^{k} \right) + \dfrac{1}{2} \left\lVert z - z^{k} \right\rVert _{\M_{2}^{k}} ^{2} \right\rbrace \\ 
		& = \arg\min\limits_{z \in \sR^{m}} \left\lbrace g \left( z \right) + \left\langle y^{k} , A x^{k} - z \right\rangle + \dfrac{r}{2} \left\lVert A x^{k} - z \right\rVert ^{2} + \dfrac{1}{2} \left\lVert z - z^{k} \right\rVert _{\M_{2}^{k}} ^{2} \right\rbrace , 
		\end{split}
		\\		
		\begin{split}
		\label{al.1:algo:x}
		x^{k+1} & \in \arg\min\limits_{x \in \sR^{n}} \left\lbrace \Lr \left( x , z^{k+1} , y^{k} \right) + \dfrac{1}{2} \left\lVert x - x^{k} \right\rVert _{\M_{1}^{k}} ^{2} \right\rbrace \\
		& = \arg\min\limits_{x \in \sR^{n}} \left\lbrace h \left( x \right) + \left\langle y^{k} , A x - z^{k+1} \right\rangle + \dfrac{r}{2} \left\lVert A x - z^{k+1} \right\rVert ^{2} + \dfrac{1}{2} \left\lVert x - x^{k} \right\rVert _{\M_{1}^{k}} ^{2} \right\rbrace ,
		\end{split}
		\\
		\label{al.1:algo:y}
		y^{k+1} & := y^{k} + \rho r \left( A x^{k+1} - z^{k+1} \right) . 
		\end{align}	
	\end{subequations}
\end{algo}

Let $\left\lbrace t_{k} \right\rbrace _{k \geq 0}$ be a sequence of positive real numbers such that $t_{k} \geq r \left\lVert A \right\rVert ^{2}$, and $\M_{1}^{k} := t_{k} \Id - r A^{*} A$  and $\M_{2}^{k} := \0$ for every $k \geq 0$. In this particular case
Algorithm \ref{algo:al.1} becomes an iterative scheme which generates a sequence $\left\lbrace \left( x^{k} , z^{k} , y^{k} \right) \right\rbrace _{k \geq 0}$ for every $k \geq 0$ as: 
	\begin{align*}	
	z^{k+1} & \in \arg\min\limits_{z \in \sR^{m}} \left\lbrace g \left( z \right) + \dfrac{r}{2} \left\lVert z - A x^{k} - \dfrac{1}{r} y^{k} \right\rVert ^{2} \right\rbrace , \\
	x^{k+1} & \in \arg\min\limits_{x \in \sR^{n}} \left\lbrace h \left( x \right) + \dfrac{t_{k}}{2} \left\lVert x - x^{k} + \frac{1}{t_{{k}}} A^{*} \left[ y^{k} + r \left( A x^{k} - z^{k+1} \right) \right] \right\rVert ^{2} \right\rbrace , \\
	y^{k+1} & := y^{k} + \rho r \left( A x^{k+1} - z^{k+1} \right) .
	\end{align*}

Recall that the \emph{proximal point operator with parameter $\gamma >0$} of a proper and lower semicontinuous function $\Psi \colon \sR^{N} \to \sR \cup \left\lbrace + \infty \right\rbrace$ is the set-valued operator defined as (\cite{Moreau})
\begin{equation*}
 \prox_{\gamma \Psi} : \sR^N \mapsto 2^{\sR^N}, \quad  \prox_{\gamma \Psi} \left( x \right) = \arg\min\limits_{y \in \sR^{N}} \left\lbrace \Psi \left( y \right) + \dfrac{1}{2 \gamma} \left\lVert x - y \right\rVert ^{2} \right\rbrace.
\end{equation*}

The above particular instance of Algorithm  \ref{algo:al.1} is an iterative scheme formulated in the spirit of full splitting numerical methods; in other words, the functions $g$ and $h$ are evaluated by their proximal operators, while the linear operator $A$ and its adjoint operator are evaluated by simple forward steps. Exact formulas for the proximal operator are available not only for large classes of convex functions (\cite{Beck,Combettes-Wajs}), but also for many nonconvex functions occurring in applications (\cite{Attouch-Bolte-Redont-Soubeyran, Hare-Sagastizabal, Lewis-Malick}).

The second algorithm that we propose in this paper replaces for every $k \geq 0$ the function $h$ in the definition of $x^{k+1}$ by its linearization at $x^k$.

\begin{algo}
	\label{algo:al.2}	
	Let be the matrix sequences $\left\lbrace \M_{1}^{k} \right\rbrace_{k \geq 0} \in \sS^{n}_{+}$ , $\left\lbrace \M_{2}^{k} \right\rbrace_{k \geq 0} \in \sS^{m}_{+}$, $r > 0$ and $0 < \rho < 2$. For a given starting vector $\left( x^{0} , z^{0} , y^{0} \right) \in \sR^{n} \times \sR^{m} \times \sR^{m}$, generate the sequence $\left\lbrace \left( x^{k} , z^{k} , y^{k} \right) \right\rbrace _{k \geq 0}$ for every $k \geq 0$ as:
	\begin{subequations} 
		\begin{align}
		\label{al.2:algo:z}
		z^{k+1} & \in \arg\min\limits_{z \in \sR^{m}} \left\lbrace g \left( z \right) + \left\langle y^{k} , A x^{k} - z \right\rangle + \dfrac{r}{2} \left\lVert A x^{k} - z \right\rVert ^{2} + \dfrac{1}{2} \left\lVert z - z^{k} \right\rVert _{\M_{2}^{k}} ^{2} \right\rbrace , \\
		\label{al.2:algo:x}
		x^{k+1} & \in \arg\min\limits_{x \in \sR^{n}} \left\lbrace \left\langle x - x^{k} , \nabla h \left( x^{k} \right) \right\rangle + \left\langle y^{k} , A x - z^{k+1} \right\rangle + \dfrac{r}{2} \left\lVert A x - z^{k+1} \right\rVert ^{2} + \dfrac{1}{2} \left\lVert x - x^{k} \right\rVert _{\M_{1}^{k}} ^{2} \right\rbrace , \\	
		\label{al.2:algo:y}
		y^{k+1} & := y^{k} + \rho r \left( A x^{k+1} - z^{k+1} \right). 
		\end{align}	
	\end{subequations}
\end{algo}
Due to the presence of the variable metric inducing matrix sequences, Algorithm \ref{algo:al.2} represents a unifying scheme for several linearized ADMM algorithms from the literature (see \cite{Lin-Liu-Li,Liu-Shen-Gu,Ouyang-Chen-Lan-Pasiliao,Ren-Lin,Xu-Wu,Yang-Yuan}). By choosing as above $\M_{1}^{k} := t_{k} \Id - r A^{*} A$, where $t_k$ is positive such that $t_{k} \geq r \left\lVert A \right\rVert ^{2}$, and $\M_{2}^{k} := \0$, for every $k \geq 0$, 
Algorithm \ref{algo:al.2} translates for every $k \geq 0$ into:
	\begin{align*}
	z^{k+1} & \in \arg\min\limits_{z \in \sR^{m}} \left\lbrace g \left( z \right) + \dfrac{r}{2} \left\lVert z - A x^{k} - \dfrac{1}{r} y^{k} \right\rVert ^{2} \right\rbrace , \\
	x^{k+1} & := x^{k} - \frac{1}{t_k} \left( \nabla h \left( x^{k} \right) + A^{*} \left[ y^{k} + r \left( A x^{k} - z^{k+1} \right) \right] \right) , \\
	y^{k+1} & := y^{k} + \rho r \left( A x^{k+1} - z^{k+1} \right).
	\end{align*}	
In this iterative scheme  the smooth term is evaluated via a gradient step, which is an improvement with respect to other nonconvex ADMM algorithms, such as \cite{Wang-Yin-Zeng, Yang-Pong-Chen}, where the smooth function is involved in a subproblem, which may be difficult to solve, unless it can be reformulated as a proximal step (see \cite{Li-Pong}).

We will carry out a parallel convergence analysis for Algorithm \ref{algo:al.1} and Algorithm \ref{algo:al.2} in the following setting.

\begin{assume}\label{ass}
We assume that 
\begin{enumerate}
	\item $g$ and $h$ are bounded from below;
	
	\item $A$ is surjective and thus the constant 
	\begin{equation*}
	\Tdecx := \begin{cases}
	\dfrac{1}{\lambda_{\min}(AA^*) \rho}, 	& \textrm{ if } 0 < \rho \leq 1 , \\
	\dfrac{\rho}{\lambda_{\min}(AA^*) \left( 2 - \rho \right) ^{2}}, 			& \textrm{ if } 1 < \rho < 2, 	\end{cases}
	\end{equation*}
	is well-defined;
	
\item $ \mu_{1} := \sup\limits_{k \geq 0} \left\lVert \M_{1}^{k} \right\rVert < + \infty$ and $\mu_{2} := \sup\limits_{k \geq 0} \left\lVert \M_{2}^{k} \right\rVert < + \infty$;

	\item $r >0, \rho \in (0,2)$ and $\mu_{1} \geq 0$  are such that
\end{enumerate}
\begin{equation}
	\label{assume:r:positive}
	r \geq 4 \Tdecx L > 0
	\end{equation}
and
	\begin{equation}
\label{assume:vm:positive}
2 \M_{1}^{k} + r A^{*} A \succcurlyeq \left( L + \dfrac{C_{\M}}{r} \right) \Id \quad \forall k \geq 0,
\end{equation}
where
\begin{equation*}
C_{\M} := \left \{\begin{array}{rl} \left( 6\mu_{1} ^{2} + 4\left( L + \mu_{1} \right) ^{2} \right) \Tdecx, & \! \mbox{for Algorithm} \ \ref{algo:al.1}, \\
\left( 4\mu_{1} ^{2} + 6\left( L + \mu_{1} \right) ^{2} \right) \Tdecx, & \! \mbox{for Algorithm} \ \ref{algo:al.2}.
  \end{array} \right.
\end{equation*}
\end{assume}

\begin{rmk}
\begin{enumerate}
		\item It has been noticed also by other authors (see, for instance, \cite{Bolte-Sabach-Teboulle:MOOR}) that the surjectivity of the linear operator is an assumption which at this moment cannot be omitted when aiming to prove convergence for nonconvex Lagrangian based algorithms. 
		
		\item \label{rmk:para-choice}
		In the following we discuss possible choices of the matrix sequence $\left\lbrace \M_{1}^{k} \right\rbrace _{k \geq 0}$ which fulfil Assumption \ref{ass}:
		
		\begin{enumerate}[leftmargin=\parindent]
			\item 
			\label{rmk:para-choice:i}
			If $\sup\limits_{k \geq 0} \left\lVert \M_{1}^{k} \right\rVert = \mu_{1} > \dfrac{L}{2}$, then, for every
			\begin{equation*}
			r \geq \max \left\lbrace 4 \Tdecx L , \dfrac{C_{\M}}{2 \mu_{1} - L} \right\rbrace > 0 ,
			\end{equation*}
			there exists $\alpha_{1} > 0$ such that 
			\begin{equation*}
			\mu_{1} \geq \alpha_{1} \geq \dfrac{1}{2} \left( L + \frac{C_{\M}}{r} \right) > 0.
			\end{equation*}
The inequality in \eqref{assume:vm:positive} is ensured for $\M_{1}^{k}$ chosen such that
			\begin{equation*}
			\mu_1 \Id \succcurlyeq \M_{1}^{k} \succcurlyeq \alpha_1 \Id \quad \forall k \geq 0 .
			\end{equation*}
			
			\item 
			\label{rmk:para-choice:iii}
			If $A$ is assumed to be also injective, then $\lambda_{\min} \left( A^{*} A \right) >0.$  By choosing
			\begin{equation*}
			r \geq \max \left\lbrace 4 \Tdecx L , \dfrac{L + \sqrt{L^{2} + 4 \lambda_{\min} \left( A^{*} A \right) C_{\M}}}{2 \lambda_{\min} \left( A^{*} A \right)} \right\rbrace > 0 ,
			\end{equation*}
			it follows that $r^2 \lambda_{\min} \left( A^{*} A \right) - rL - C_{\M} \geq 0$. Thus,
			\begin{equation*}
			r A^{*} A - \left( L + r^{-1} C_{\M} \right) \Id \succcurlyeq 0,
			\end{equation*}
			and  \eqref{assume:vm:positive}  holds for an arbitrary sequence of symmetric and positive semidefinite matrices $\left\lbrace \M_{1}^{k} \right\rbrace _{k \geq 0}$. A possible choice is $\M_{1}^{k} = 0$, which, together $\M_{2}^{k} = 0$, for every $k \geq 0$, allows us to recover the classical ADMM algorithm and the linearized ADMM algorithm as particular instances of our iterative schemes.
			
			\item
			\label{rmk:para-choice:ii}
For $t >0$, we take  $\M_{1}^{k} := t \Id - r A^{*} A$ for every $k \geq 0$ (see also Section 6.3 in \cite{Bolte-Sabach-Teboulle:MOOR}). Then
			\begin{equation*}
			\mu_{1} = \left\lVert t \Id - r A^{*} A \right\rVert = \lambda_{\max} \left( t \Id - r A^{*} A \right) = t - r \lambda_{\min} \left( A^{*} A \right) .
			\end{equation*}
Condition \eqref{assume:vm:positive} is equivalent to
\begin{align*}
2 t - r \left\lVert A \right\rVert ^{2} - \left( L + \dfrac{C_{\M}}{r} \right) \geq  0
\end{align*}
and is guaranteed for both algorithms when
\begin{align*}
2 t - r \left\lVert A \right\rVert ^{2} - \left( L + \dfrac{\left( 4 \mu_{1} ^{2} + 6 \left( L + \mu_{1} \right) ^{2} \right) \Tdecx}{r} \right) \geq 0
\end{align*}
or, equivalently,
			\begin{align*}
10 \Tdecx \mu_{1} ^{2} - 2 \left( r - 6 \Tdecx L \right) \mu_{1} + 6 \Tdecx L^{2} + r^{2} \left( \left\lVert A \right\rVert ^{2} - 2 \lambda_{\min} \left( A^{*} A \right) \right) - Lr \leq 0.
			\end{align*}
This quadratic inequality in $\mu_1 \geq 0$ has nonnegative solutions if, for instance, $r \geq 6T_0L$ (thus \eqref{assume:r:positive} holds) and the reduced discriminant
			\begin{align*}
			\Delta := & \left( r - 6 \Tdecx L \right) ^{2} - 60 \Tdecx ^{2} L^{2} - 10 \Tdecx r^{2} \left( \left\lVert A \right\rVert ^{2} - 2 \lambda_{\min} \left( A^{*} A \right) \right) + 10 \Tdecx Lr \\
			= & \left[ 1 + 10 \Tdecx \left( 2 \lambda_{\min} \left( A^{*} A \right) - \left\lVert A \right\rVert ^{2} \right) \right] r^{2} - 2 \Tdecx Lr - 24 \Tdecx ^{2} L^{2}
			\end{align*}
is nonnegative. This holds true if the condition number of the matrix $A^*A$ fulfils $$\kappa \left( A^{*} A \right) := \dfrac{\lambda_{\max} \left( A^{*} A \right)}{\lambda_{\min} \left( A^{*} A \right)} = \dfrac{\left\lVert A \right\rVert ^{2}}{\lambda_{\min} \left( A^{*} A \right)} \leq 2.$$
In conclusions, if the latter is given, then we can chose an arbitrary
$$r \geq 6T_0L$$
and $t$ such that
			\begin{equation*}
 r \lambda_{\min} \left( A^{*} A \right) \leq t \leq r \lambda_{\min} \left( A^{*} A \right) + \dfrac{1}{10 \Tdecx} \left( r - 6 \Tdecx L + \sqrt{\Delta} \right) .
			\end{equation*}
For a similar choice for the 
		\end{enumerate}
	
	\item When proving convergence and deriving convergence rates for variable metric algorithms designed for convex optimization problems one usually assumes monotonicity for the matrix sequences inducing the variable metrics (see, for instance, \cite{Combettes-Vu-quasi-Fejer, Banert-Bot-Csetnek,Bot-Csetnek-ADMM}). It is worth to mention that  the convergence analysis for both Algorithm \ref{algo:al.1} and Algorithm \ref{algo:al.2} does not require monotonicity assumptions on $\left\lbrace \M_{1}^{k} \right\rbrace _{k \geq 0}$ or $\left\lbrace \M_{2}^{k} \right\rbrace _{k \geq 0}$.
	\end{enumerate}
\end{rmk}

\subsection{Preliminaries of the convergence analysis}
\label{subsec:pre-conv}
Within the setting of Assumption \ref{ass} we will make use of the following constants:
\begin{align*}
\Cdecx := \begin{cases}
L + \dfrac{4 \Tdecx  \left( L + \mu_{1} \right) ^{2}}{r}, & \mbox{for Algorithm} \ \ref{algo:al.1}, \\
L + \dfrac{4 \Tdecx \mu_{1} ^{2}}{r}, & \mbox{for Algorithm} \ \ref{algo:al.2}, \\
\end{cases} \quad \quad \Cdecxx := \begin{cases}
\dfrac{4 \Tdecx \mu_{1} ^{2}}{r}, & \mbox{for Algorithm} \ \ref{algo:al.1},\\
\dfrac{4 \Tdecx \left( L + \mu_{1} \right) ^{2}}{r}, & \mbox{for Algorithm} \ \ref{algo:al.2},\\
\end{cases}
\end{align*}
and 
\begin{equation*}
\Tdecy := \begin{cases}
\dfrac{1 - \rho}{\lambda_{\min}(AA^*)  \rho^{2} r} , 						& \textrm{ if } 0 < \rho \leq 1 , \\
\dfrac{\rho - 1}{\lambda_{\min}(AA^*) \left( 2 - \rho \right) \rho r} , 	& \textrm{ if } 1 < \rho < 2 ,
\end{cases}
\end{equation*}
and we will denote for every $k \geq 0$
\begin{equation*}
\M_{3}^{k}:= 2 \M_{1}^{k} + r A^{*} A - \Cdecx \Id .
\end{equation*}

The following monotonicity result will play a fundamental role in our convergence analysis.
\begin{lem}
	\label{lem:decrease}
	Suppose that Assumption \ref{ass} holds true and let $\left\lbrace \left( x^{k} , z^{k} , y^{k} \right) \right\rbrace _{k \geq 0}$ be a sequence generated by Algorithm \ref{algo:al.1} or Algorithm \ref{algo:al.2}. Then for every $k \geq 1$ it holds:
	\begin{align}
	\label{dec:inq}
	& \Lr \left( x^{k+1} , z^{k+1} , y^{k+1} \right) + \Tdecy \left\lVert A^{*} \left( y^{k+1} - y^{k} \right) \right\rVert ^{2} + \dfrac{1}{2} \left\lVert x^{k+1} - x^{k} \right\rVert _{\M_{3}^{k}}^{2} + \dfrac{1}{2} \left\lVert z^{k+1} - z^{k} \right\rVert _{\M_{2}^{k}}^{2} \nonumber \\
	\leq & \ \Lr \left( x^{k} , z^{k} , y^{k} \right) + \Tdecy \left\lVert A^{*} \left( y^{k} - y^{k-1} \right) \right\rVert ^{2} + \dfrac{\Cdecxx}{2} \left\lVert x^{k} - x^{k-1} \right\rVert^2.
	\end{align}
\end{lem}
\begin{proof}
Let $k \geq 1$ be fixed. In both cases the proof builds on showing that the following inequality
	\begin{align}
	\label{dec:pre}
	& \Lr \left( x^{k+1} , z^{k+1} , y^{k+1} \right) + \dfrac{1}{2} \left\lVert x^{k+1} - x^{k} \right\rVert _{2 \M_{1}^{k} + r A^{*} A}^{2} - \dfrac{L}{2} \left\lVert x^{k+1} - x^{k} \right\rVert ^{2} + \dfrac{1}{2} \left\lVert z^{k+1} - z^{k} \right\rVert _{\M_{2}^{k}}^{2} \nonumber \\
	\leq & \ \Lr \left( x^{k} , z^{k} , y^{k} \right) + \dfrac{1}{\rho r} \left\lVert y^{k+1} - y^{k} \right\rVert ^{2}
	\end{align}
is true and on providing afterwards an upper bound for $\dfrac{1}{\rho r} \left\lVert y^{k+1} - y^{k} \right\rVert ^{2}$.
	\begin{enumerate}[leftmargin=\parindent]		
		\item For \emph{Algorithm \ref{algo:al.1}}:
		From \eqref{al.1:algo:z} we have
		\begin{align}
		\label{dec:al.1:arg-g}
		& g \left( z^{k+1} \right) + \left\langle y^{k} , A x^{k} - z^{k+1} \right\rangle + \dfrac{r}{2} \left\lVert A x^{k} - z^{k+1} \right\rVert ^{2} + \dfrac{1}{2} \left\lVert z^{k+1} - z^{k} \right\rVert _{\M_{2}^{k}}^{2} \nonumber \\
		\leq & \ g \left( z^{k} \right) + \left\langle y^{k} , A x^{k} - z^{k} \right\rangle + \dfrac{r}{2} \left\lVert A x^{k} - z^{k} \right\rVert ^{2} .
		\end{align}
		The optimality criterion of \eqref{al.1:algo:x} is
		\begin{equation}
		\label{dec:al.1:opt.con-f}
		\nabla h \left( x^{k+1} \right) = - A^{*} y^{k} - r A^{*} \left( A x^{k+1} - z^{k+1} \right) + \M_{1}^{k} \left( x^{k} - x^{k+1} \right) .
		\end{equation}
		From \eqref{eq:semi-convex} (applied for $z:=x^{k+1}$) we get
		\begin{align}
		\label{dec:al.1:semi-f}
		h \left( x^{k+1} \right) \leq & \ h \left( x^{k} \right) + \left\langle y^{k} , A x^{k} - A x^{k+1} \right\rangle + r \left\langle A x^{k+1} - z^{k+1} , A x^{k} - A x^{k+1} \right\rangle \nonumber \\
		& \ - \left\lVert x^{k+1} - x^{k} \right\rVert _{\M_{1}^{k}}^{2} + \dfrac{L}{2} \left\lVert x^{k+1} - x^{k} \right\rVert ^{2} .
		\end{align}		
		By combining \eqref{al.1:algo:y},  \eqref{dec:al.1:arg-g} and \eqref{dec:al.1:semi-f}, after some rearrangements, we obtain \eqref{dec:pre}.

By using the notation
		\begin{equation}
		\label{defi:u1}
		u_{1}^{l} := - \nabla h \left( x^{l} \right) + \M_{1}^{l-1} \left( x^{l-1} - x^{l} \right) \  \forall l \geq 1
		\end{equation}
		and by taking into consideration \eqref{al.1:algo:y}, we can rewrite \eqref{dec:al.1:opt.con-f} as
		\begin{equation}
		\label{defi:new-y1}
		A^{*} y^{l+1} = \rho u_{1}^{l+1} + \left( 1 - \rho \right) A^{*} y^{l} \ \forall l \geq 0.
		\end{equation}

		\begin{itemize}[leftmargin=\parindent]
			\item \emph{The case $0 < \rho \leq 1$}.
We have
\begin{equation*}
				A^{*} \left( y^{k+1} - y^{k} \right) = \rho \left( u_{1}^{k+1} - u_{1}^{k} \right) + \left( 1 - \rho \right) A^{*} \left( y^{k} - y^{k-1} \right) .
				\end{equation*}
			Since $0 < \rho \leq 1$, the convexity of $\left\lVert \cdot \right\rVert ^{2}$ gives
			\begin{equation*}
			\left\lVert A^{*} \left( y^{k+1} - y^{k} \right) \right\rVert ^{2} \leq \rho \left\lVert u_{1}^{k+1} - u_{1}^{k} \right\rVert ^{2} + \left( 1 - \rho \right) \left\lVert A^{*} \left( y^{k} - y^{k-1} \right) \right\rVert ^{2} 
			\end{equation*}
and from here we get
			\begin{align}
			\label{dec:al.1:mul.p1.pre-bound}
			& \lambda_{\min}(AA^*) \rho \left\lVert y^{k+1} - y^{k} \right\rVert ^{2} \leq \rho \left\lVert A^{*} \left( y^{k+1} - y^{k} \right) \right\rVert ^{2} \nonumber \\
			\leq & \ \rho \left\lVert u_{1}^{k+1} - u_{1}^{k} \right\rVert ^{2} + \left( 1 - \rho \right) \left\lVert A^{*} \left( y^{k} - y^{k-1} \right) \right\rVert ^{2} - \left( 1 - \rho \right) \left\lVert A^{*} \left( y^{k+1} - y^{k} \right) \right\rVert ^{2}.
			\end{align}
By using the Lipschitz continuity of $\nabla h$ we have
			\begin{equation}
			\label{dec:al.1:u-pre-bound}
			\left\lVert u_{1}^{k+1} - u_{1}^{k} \right\rVert \leq \left( L + \mu_{1} \right) \left\lVert x^{k+1} - x^{k} \right\rVert + \mu_{1} \left\lVert x^{k} - x^{k-1} \right\rVert ,
			\end{equation}
thus
			\begin{equation}
			\label{dec:al.1:u-bound}
			\left\lVert u_{1}^{k+1} - u_{1}^{k} \right\rVert ^{2} \leq 2 \left( L + \mu_{1} \right) ^{2} \left\lVert x^{k+1} - x^{k} \right\rVert ^{2} + 2 \mu_{1}^{2} \left\lVert x^{k} - x^{k-1} \right\rVert ^{2} .
			\end{equation}
After plugging \eqref{dec:al.1:u-bound} into \eqref{dec:al.1:mul.p1.pre-bound} we get
			\begin{align}
			\label{dec:al.1:p1}
			\dfrac{1}{\rho r} \left\lVert y^{k+1} - y^{k} \right\rVert ^{2} & \leq \frac{2 \left( L + \mu_{1} \right) ^{2}}{ \lambda_{\min}(AA^*) \rho r} \left\lVert x^{k+1} - x^{k} \right\rVert ^{2} + \frac{2 \mu_{1}^{2}}{\lambda_{\min}(AA^*) \rho r} \left\lVert x^{k} - x^{k-1} \right\rVert ^{2} \nonumber \\
			& \quad + \dfrac{\left( 1 - \rho \right)}{\lambda_{\min}(AA^*) \rho^{2} r} \left\lVert A^{*} \left( y^{k} - y^{k-1} \right) \right\rVert ^{2} - \dfrac{\left( 1 - \rho \right)}{\lambda_{\min}(AA^*) \rho^{2} r} \left\lVert A^{*} \left( y^{k+1} - y^{k} \right) \right\rVert ^{2},
			\end{align}
which, combined with  \eqref{dec:pre}, provides \eqref{dec:inq}.		
			\item \emph{The case $1 < \rho < 2$}.
This time we have from \eqref{defi:new-y1} that
			\begin{equation*}
			A^{*} \left( y^{k+1} - y^{k} \right) = \left( 2 - \rho \right) \dfrac{\rho}{2 - \rho} \left( u_{1}^{k+1} - u_{1}^{k} \right) + \left( \rho - 1 \right) A^{*} \left( y^{k-1} - y^{k} \right) . 
			\end{equation*} 
As $1 < \rho < 2$, the convexity of $\left\lVert \cdot \right\rVert ^{2}$ gives
			\begin{equation*}
			\left\lVert A^{*} \left( y^{k+1} - y^{k} \right) \right\rVert ^{2} \leq \dfrac{\rho^{2}}{2 - \rho} \left\lVert u_{1}^{k+1} - u_{1}^{k} \right\rVert ^{2} + \left( \rho - 1 \right) \left\lVert A^{*} \left( y^{k} - y^{k-1} \right) \right\rVert ^{2}
			\end{equation*}
and from here it follows
			\begin{align}
			\label{dec:al.1:mul.p2.pre-bound}
			&  \lambda_{\min}(AA^*) \left( 2 - \rho \right) \left\lVert y^{k+1} - y^{k} \right\rVert ^{2} \leq \left( 2 - \rho \right) \left\lVert A^{*} \left( y^{k+1} - y^{k} \right) \right\rVert ^{2} \nonumber \\
			\leq & \ \dfrac{\rho^{2}}{2 - \rho} \left\lVert u_{1}^{k+1} - u_{1}^{k} \right\rVert ^{2} + \left( \rho - 1 \right) \left\lVert A^{*} \left( y^{k} - y^{k-1} \right) \right\rVert ^{2} - \left( \rho - 1 \right) \left\lVert A^{*} \left( y^{k+1} - y^{k} \right) \right\rVert ^{2}.
			\end{align}
After plugging \eqref{dec:al.1:u-bound} into \eqref{dec:al.1:mul.p2.pre-bound} we get
			\begin{align}
			\label{dec:al.1:p2}
			\dfrac{1}{\rho r} \left\lVert y^{k+1} - y^{k} \right\rVert ^{2} & \leq \frac{2 \rho \left( L + \mu_{1} \right) ^{2}}{\lambda_{\min}(AA^*) \left( 2 - \rho \right) ^{2} r} \left\lVert x^{k+1} - x^{k} \right\rVert ^{2} + \frac{2 \rho \mu_{1} ^{2}}{\lambda_{\min}(AA^*) \left( 2 - \rho \right) ^{2} r} \left\lVert x^{k} - x^{k-1} \right\rVert ^{2} \nonumber \\
			& \quad + \dfrac{\left( \rho - 1 \right)}{\lambda_{\min}(AA^*) \left( 2 - \rho \right) \rho r} \left\lVert A^{*} \left( y^{k} - y^{k-1} \right) \right\rVert ^{2} \nonumber\\
& \quad - \dfrac{\left( \rho - 1 \right)}{\lambda_{\min}(AA^*) \left( 2 - \rho \right) \rho r} \left\lVert A^{*} \left( y^{k+1} - y^{k} \right) \right\rVert ^{2},
			\end{align}
		\end{itemize}
which, combined with  \eqref{dec:pre}, provides \eqref{dec:inq}.
				
		\item For \emph{Algorithm \ref{algo:al.2}}:
		The optimality criterion of \eqref{al.2:algo:x} is
		\begin{equation}
		\label{dec:al.2:opt.con-f}
		\nabla h \left( x^{k} \right) = - A^{*} y^{k} - r A^{*} \left( A x^{k+1} - z^{k+1} \right) + \M_{1}^{k} \left( x^{k} - x^{k+1} \right) .
		\end{equation}
From \eqref{eq:semi-convex} (applied for $z:=x^k$) we get
		\begin{align}
		\label{dec:al.2:semi-f}
		h \left( x^{k+1} \right) \leq & \ h \left( x^{k} \right) + \left\langle y^{k} , A x^{k} - A x^{k+1} \right\rangle + r \left\langle A x^{k+1} - z^{k+1} , A x^{k} - A x^{k+1} \right\rangle \nonumber \\
		& \ - \left\lVert x^{k+1} - x^{k} \right\rVert _{\M_{1}^{k}}^{2} + \dfrac{L}{2} \left\lVert x^{k+1} - x^{k} \right\rVert ^{2} .
		\end{align}		
Since the definition of $z^{k+1}$ in \eqref{al.2:algo:z} leads also to \eqref{dec:al.1:arg-g}, by combining this inequality with \eqref{dec:al.2:semi-f} and \eqref{al.2:algo:y}, after some rearrangments, \eqref{dec:pre} follows.
By using this time the notation
		\begin{equation}
		\label{defi:u2}
		u_{2}^{l} := - \nabla h \left( x^{l-1} \right) + \M_{1}^{l-1} \left( x^{l-1} - x^{l} \right) \  \forall l \geq 1
		\end{equation}
		and by taking into consideration \eqref{al.2:algo:y}, we can rewrite \eqref{dec:al.2:opt.con-f} as
		\begin{equation}
		\label{defi:new-y2}
		A^{*} y^{l+1} = \rho u_{2}^{l+1} + \left( 1 - \rho \right) A^{*} y^{l} \  \forall l \geq 0.
		\end{equation} 

		\begin{itemize}[leftmargin=\parindent]
			\item \emph{The case $0 < \rho \leq 1$}.
As in \eqref{dec:al.1:mul.p1.pre-bound} we obtain
			\begin{align}
			\label{dec:al.2:mul.p1.pre-bound}
			&  \lambda_{\min}(AA^*) \rho \left\lVert y^{k+1} - y^{k} \right\rVert ^{2} \leq \rho \left\lVert A^{*} \left( y^{k+1} - y^{k} \right) \right\rVert ^{2} \nonumber \\
			\leq & \ \rho \left\lVert u_{2}^{k+1} - u_{2}^{k} \right\rVert ^{2} + \left( 1 - \rho \right) \left\lVert A^{*} \left( y^{k} - y^{k-1} \right) \right\rVert ^{2} - \left( 1 - \rho \right) \left\lVert A^{*} \left( y^{k+1} - y^{k} \right) \right\rVert ^{2} .			
			\end{align}
By using the Lipschitz continuity of $\nabla h$ we have
			\begin{equation}
			\label{dec:al.2:u-pre-bound}
			\left\lVert u_{2}^{k+1} - u_{2}^{k} \right\rVert \leq \mu_{1} \left\lVert x^{k+1} - x^{k} \right\rVert + \left( L + \mu_{1} \right) \left\lVert x^{k} - x^{k-1} \right\rVert,
			\end{equation}
thus
			\begin{equation}
			\label{dec:al.2:u-bound}
			\left\lVert u_{2}^{k+1} - u_{2}^{k} \right\rVert ^{2} \leq 2 \mu_{1}^{2} \left\lVert x^{k+1} - x^{k} \right\rVert ^{2} + 2 \left( L + \mu_{1} \right) ^{2} \left\lVert x^{k} - x^{k-1} \right\rVert ^{2} .
			\end{equation}
After plugging \eqref{dec:al.2:u-bound} into \eqref{dec:al.2:mul.p1.pre-bound} it follows
			\begin{align}
			\label{dec:al.2:p1}
			\dfrac{1}{\rho r} \left\lVert y^{k+1} - y^{k} \right\rVert ^{2} & \leq \frac{2 \mu_{1}^{2}}{\lambda_{\min}(AA^*) \rho r} \left\lVert x^{k+1} - x^{k} \right\rVert ^{2} + \frac{2 \left( L + \mu_{1} \right) ^{2}}{\lambda_{\min}(AA^*) \rho r} \left\lVert x^{k} - x^{k-1} \right\rVert ^{2} \nonumber \\
			& \quad + \dfrac{\left( 1 - \rho \right)}{\lambda_{\min}(AA^*) \rho^{2} r} \left\lVert A^{*} \left( y^{k} - y^{k-1} \right) \right\rVert ^{2} - \dfrac{\left( 1 - \rho \right)}{\lambda_{\min}(AA^*) \rho^{2} r} \left\lVert A^{*} \left( y^{k+1} - y^{k} \right) \right\rVert ^{2},
			\end{align}
which, combined with  \eqref{dec:pre}, provides \eqref{dec:inq}.			
\item \emph{The case $1 < \rho < 2$}.
			As in \eqref{dec:al.1:mul.p2.pre-bound} we obtain
			\begin{align}
			\label{dec:al.2:mul.p2.pre-bound}
			& \lambda_{\min}(AA^*)  \left( 2 - \rho \right) \left\lVert y^{k+1} - y^{k} \right\rVert ^{2} \leq \left( 2 - \rho \right) \left\lVert A^{*} \left( y^{k+1} - y^{k} \right) \right\rVert ^{2} \nonumber \\
			\leq & \ \frac{\rho^2}{2-\rho} \left\lVert u_{2}^{k+1} - u_{2}^{k} \right\rVert ^{2} + \left( \rho - 1 \right) \left\lVert A^{*} \left( y^{k} - y^{k-1} \right) \right\rVert ^{2} - \left( \rho - 1 \right) \left\lVert A^{*} \left( y^{k+1} - y^{k} \right) \right\rVert ^{2}.
			\end{align}
After plugging \eqref{dec:al.2:u-bound} into \eqref{dec:al.2:mul.p2.pre-bound} it follows
			\begin{align}
			\label{dec:al.2:p2}
			\dfrac{1}{\rho r} \left\lVert y^{k+1} - y^{k} \right\rVert ^{2} & \leq \frac{2 \rho \mu_{1} ^{2}}{\lambda_{\min}(AA^*) \left( 2 - \rho \right) ^{2} r} \left\lVert x^{k+1} - x^{k} \right\rVert ^{2} + \frac{2 \rho \left( L + \mu_{1} \right) ^{2}}{\lambda_{\min}(AA^*) \left( 2 - \rho \right) ^{2} r} \left\lVert x^{k} - x^{k-1} \right\rVert ^{2} \nonumber \\
			& \quad + \dfrac{\left( \rho - 1 \right)}{\lambda_{\min}(AA^*) \left( 2 - \rho \right) \rho r} \left\lVert A^{*} \left( y^{k} - y^{k-1} \right) \right\rVert ^{2} \nonumber \\
& \quad - \dfrac{\left( \rho - 1 \right)}{\lambda_{\min}(AA^*) \left( 2 - \rho \right) \rho r} \left\lVert A^{*} \left( y^{k+1} - y^{k} \right) \right\rVert ^{2} .
			\end{align}
which, combined with  \eqref{dec:pre}, provides \eqref{dec:inq}.	
		\end{itemize}
	\end{enumerate}
This concludes the proof.
\end{proof}

The following three estimates will be useful in the sequel.
\begin{lem}
	\label{lem:estimate}
Suppose that Assumption \ref{ass} holds true and let $\left\lbrace \left( x^{k} , z^{k} , y^{k} \right) \right\rbrace _{k \geq 0}$ be a sequence generated by Algorithm \ref{algo:al.1} or Algorithm \ref{algo:al.2}. Then the following statements are true:
	\begin{enumerate}[leftmargin=\parindent]
		\item[(i)] for every $k \geq 1$
		\begin{align}
		\label{est:z-bounded-by-x}
		\left\lVert z^{k+1} - z^{k} \right\rVert & \leq \left\lVert A \right\rVert \cdot \left\lVert x^{k+1} - x^{k} \right\rVert + \left\lVert A x^{k+1} - z^{k+1} \right\rVert + \left\lVert A x^{k} - z^{k} \right\rVert \nonumber \\
		& = \left\lVert A \right\rVert \cdot \left\lVert x^{k+1} - x^{k} \right\rVert + \dfrac{1}{\rho r} \left\lVert y^{k+1} - y^{k} \right\rVert + \dfrac{1}{\rho r} \left\lVert y^{k} - y^{k-1} \right\rVert;
		\end{align}
		
		\item[(ii)] for every $k \geq 0$
		\begin{equation}
		\label{est:mult}
		\dfrac{1}{2r} \left\lVert y^{k+1} \right\rVert ^{2} \leq \dfrac{\Tdecy}{2} \left\lVert A^{*} \left( y^{k+1} - y^{k} \right) \right\rVert ^{2} + \dfrac{\Tdecx}{r} \left\lVert \nabla h \left( x^{k+1} \right) \right\rVert ^{2} + \dfrac{\Cdecxx}{4} \left\lVert x^{k+1} - x^{k} \right\rVert ^{2};
		\end{equation}
		
                  \item[(iii)] for every $k \geq 1$
		\begin{align}
		\begin{split}
		\label{est:y-bounded-by-x}
		\!\left\lVert y^{k+1} - y^{k} \right\rVert \leq \Cestx \left\lVert x^{k+1} - x^{k} \right\rVert + \Cestxx \left\lVert x^{k} - x^{k-1} \right\rVert + \Cesty \left( \left\lVert A^{*} \left( y^{k} - y^{k-1} \right) \right\rVert - \left\lVert A^{*} \left( y^{k+1} - y^{k} \right) \right\rVert \right),	
		\end{split}
		\end{align}
		where
		\begin{align*}
		\label{est:const}
		& \Cestx := \begin{cases}
		\dfrac{\rho \left( L + \mu_{1} \right)}{\sqrt{\lambda_{\min}(AA^*)} \left( 1 - \left\lvert 1 - \rho \right\rvert \right)}, & \mbox{for Algorithm} \ \ref{algo:al.1},\\\\
		\dfrac{\rho \mu_{1}}{\sqrt{\lambda_{\min}(AA^*)} \left( 1 - \left\lvert 1 - \rho \right\rvert \right)}, & \mbox{for Algorithm} \ \ref{algo:al.2},\\ 
		\end{cases}\nonumber\\ 
		& \Cestxx := \begin{cases}
		 \dfrac{\rho \mu_{1}}{\sqrt{\lambda_{\min}(AA^*)} \left( 1 - \left\lvert 1 - \rho \right\rvert \right)}, & \mbox{for Algorithm} \ \ref{algo:al.1},\\ \\
		\dfrac{\rho \left( L + \mu_{1} \right)}{\sqrt{\lambda_{\min}(AA^*)} \left( 1 - \left\lvert 1 - \rho \right\rvert \right)}, & \mbox{for Algorithm} \ \ref{algo:al.2},\\
		\end{cases} \nonumber \\ 
		& \Cesty := \dfrac{\left\lvert 1 - \rho \right\rvert}{\sqrt{\lambda_{\min}(AA^*)} \left( 1 - \left\lvert 1 - \rho \right\rvert \right)} .
		\end{align*}
	\end{enumerate}	
\end{lem}

\begin{proof}
	The statement in \eqref{est:z-bounded-by-x} is straightforward. 
	
	From \eqref{defi:new-y1} and \eqref{defi:new-y2} we have for every $k \geq 0$
\begin{equation*}
		A^{*} y^{k+1} = \rho u^{k+1} + \left( 1 - \rho \right) A^{*} y^{k} 
		\end{equation*}
	or, equivalently,
	\begin{equation*}
	\rho A^{*} y^{k+1} = \rho u^{k+1} + \left( 1 - \rho \right) A^{*} \left( y^{k} - y^{k+1} \right),
	\end{equation*}
	where $u^{k+1}$ is defined as being equal to $u_{1}^{k+1}$ in \eqref{defi:u1}, for Algorithm \ref{algo:al.1}, and, respectively, to $u_{2}^{k+1}$ in \eqref{defi:u2}, for Algorithm \ref{algo:al.2}.
	
	For $0 < \rho \leq 1$ we have
	\begin{equation}
	\label{est:pre-p1}
	\lambda_{\min}(AA^*) \rho^{2} \left\lVert y^{k+1} \right\rVert ^{2} \leq \rho^{2} \left\lVert A^{*} y^{k+1} \right\rVert ^{2} \leq \rho \left\lVert u^{k+1} \right\rVert ^{2} + \left( 1 - \rho \right) \left\lVert A^{*} \left( y^{k+1} - y^{k} \right) \right\rVert ^{2},
	\end{equation}
	while, for $1 < \rho < 2$, we have
	\begin{equation}
	\label{est:pre-p2}
	\lambda_{\min}(AA^*) \rho^{2} \left\lVert y^{k+1} \right\rVert ^{2} \leq \rho^{2} \left\lVert A^{*} y^{k+1} \right\rVert ^{2} \leq \dfrac{\rho ^{2}}{2 - \rho} \left\lVert u^{k+1} \right\rVert ^{2} + \left( \rho - 1 \right) \left\lVert A^{*} \left( y^{k+1} - y^{k} \right) \right\rVert ^{2}.
	\end{equation}
  Notice further that for $1 < \rho < 2$ we have $\dfrac{1}{\rho} < 1$ and $1 < \dfrac{\rho}{2 - \rho}$.
	
In case $u^{k+1}$ is defined as in \eqref{defi:u1} it holds
	\begin{equation}
	\label{est:u1-bound}
	\left\lVert u^{k+1} \right\rVert ^{2} = \left\lVert u_{1}^{k+1} \right\rVert ^{2} \leq 2 \left\lVert \nabla h \left( x^{k+1} \right) \right\rVert ^{2} + 2 \mu_{1}^{2} \left\lVert x^{k+1} - x^{k} \right\rVert ^{2} \ \forall k \geq 0,
	\end{equation}
while, in case $u_{2}^{k+1}$ is defined as in \eqref{defi:u2}, it holds
	\begin{equation}
	\label{est:u2-bound}
	\left\lVert u^{k+1} \right\rVert ^{2} = \left\lVert u_{2}^{k+1} \right\rVert ^{2} \leq 2 \left\lVert \nabla h \left( x^{k+1} \right) \right\rVert ^{2} + 2 \left( L + \mu_{1} \right) ^{2} \left\lVert x^{k+1} - x^{k} \right\rVert ^{2} \  \forall k \geq 0.
	\end{equation}
We divide \eqref{est:pre-p1} and \eqref{est:pre-p2} by $2 \lambda_{\min}(AA^*) \rho^{2} r > 0$ and  plug \eqref{est:u1-bound} and, respectively, \eqref{est:u2-bound} into the resulting inequalities. This gives us \eqref{est:mult}.	
		
Finally, in order to prove \eqref{est:y-bounded-by-x}, we notice that for every $k \geq 1$ it holds
	\begin{equation*}
	\left\lVert A^{*} \left( y^{k+1} - y^{k} \right) \right\rVert \leq \rho \left\lVert u^{k+1} - u^{k} \right\rVert + \left\lvert 1 - \rho \right\rvert \left\lVert A^{*} \left( y^{k} - y^{k-1} \right) \right\rVert,
	\end{equation*}
so,
	\begin{align}
	\label{est:pre-y}
	& \ \sqrt{\lambda_{\min}(AA^*)} \left( 1 - \left\lvert 1 - \rho \right\rvert \right) \left\lVert y^{k+1} - y^{k} \right\rVert \leq \left( 1 - \left\lvert 1 - \rho \right\rvert \right) \left\lVert A^{*} \left( y^{k+1} - y^{k} \right) \right\rVert \nonumber\\
	\leq & \ \rho \left\lVert u^{k+1} - u^{k} \right\rVert + \left\lvert 1 - \rho \right\rvert \left\lVert A^{*} \left( y^{k} - y^{k-1} \right) \right\rVert - \left\lvert 1 - \rho \right\rvert \left\lVert A^{*} \left( y^{k+1} - y^{k} \right) \right\rVert .
	\end{align}
We plug into \eqref{est:pre-y}  the estimates for $\left\lVert u^{k+1} - u^{k} \right\rVert$ derived in \eqref{dec:al.1:u-pre-bound} and, respectively, \eqref{dec:al.2:u-pre-bound} and  divide the resulting inequality  by $\sqrt{\lambda_{\min}(AA^*)} \left( 1 - \left\lvert 1 - \rho \right\rvert \right) > 0$. 
This  furnishes the desired statement.
\end{proof}

The following regularization of the augmented Lagrangian will play an important role in the convergence analysis of the nonconvex proximal ADMM algorithms
\begin{align*}
& \Fr \colon \sR^{n} \times \sR^{m} \times \sR^{m} \times \sR^{n} \times \sR^{m} \to \sR \cup \left\lbrace + \infty \right\rbrace, \nonumber\\ 
& \Fr( x , z , y , x' , y') = \Lr \left( x , z , y \right) + \Tdecy \left\lVert A^{*} \left( y - y' \right) \right\rVert ^{2} + \dfrac{\Cdecxx}{2} \left\lVert x - x' \right\rVert^2,
\end{align*}
where $\Tdecy $ and $\Cdecxx$ are defined in Assumption \ref{ass}. For every $k \geq 1$ we denote
\begin{align}
\label{defi:Hk}
\Fk := \Fr\left( x^{k} , z^{k} , y^{k} , x^{k-1} , y^{k-1} \right) = \Lr \left( x^{k} , z^{k} , y^{k} \right) + \Tdecy  \left\lVert A^{*} \left( y^{k} - y^{k-1} \right) \right\rVert ^{2} + \dfrac{\Cdecxx}{2} \left\lVert x^{k} - x^{k-1} \right\rVert^2.
\end{align}

Since the convergence analysis will rely on the fact that the set of cluster points of the sequence $\left\lbrace \left( x^{k} , z^{k} , y^{k} \right) \right\rbrace _{k \geq 0}$ is nonempty, we will present first two situations which guarantee that this sequence is bounded. They make use of standard
coercivity assumptions for the functions $g$ and $h$, respectively. Recall that  a function $\Psi : \sR^N \to \sR \cup \left\lbrace + \infty \right\rbrace$ is called \emph{coercive}, if $\lim\limits_{\left\lVert x \right\rVert \to +\infty} \Psi \left( x \right) = + \infty$.

\begin{thm}
	\label{thm:bound}
Suppose that Assumption \ref{ass} holds true  and let $\left\lbrace \left( x^{k} , z^{k} , y^{k} \right) \right\rbrace _{k \geq 0}$ be a sequence generated by Algorithm \ref{algo:al.1} or Algorithm \ref{algo:al.2}. Suppose that one of the following conditions holds:
	\begin{enumerate}[label=(B-\Roman*)]
		\item
		\label{assume:B-I}
		$A$ is invertible and $g$ is coercive;
		\item
		\label{assume:B-II}
		$h$ is coercive.
	\end{enumerate}
Then the sequence $\left\lbrace \left( x^{k} , z^{k} , y^{k} \right) \right\rbrace _{k \geq 0}$ is bounded.
\end{thm}
\begin{proof}
From Lemma \ref{lem:decrease} we have that for every $k \geq 1$
	\begin{equation}
	\label{bound:dec:Hk}
	\F_{k+1} + \dfrac{1}{2} \left\lVert x^{k+1} - x^{k} \right\rVert _{\M_{3}^{k} - \Cdecxx\Id}^{2} + \dfrac{1}{2} \left\lVert z^{k+1} - z^{k} \right\rVert _{\M_{2}^{k}}^{2} \leq \F_{k}
	\end{equation}
	which shows, according to \eqref{assume:vm:positive}, that $\left\lbrace \F_{k} \right\rbrace _{k \geq 1}$ is monotonically decreasing. Consequently, for every $k \geq 1$ we have
	\begin{align*}
	\begin{split}
	\F_{1} \geq & \ \F_{k+1} + \dfrac{1}{2} \left\lVert x^{k+1} - x^{k} \right\rVert _{\M_{3}^{k} - \Cdecxx\Id}^{2} + \dfrac{1}{2} \left\lVert z^{k+1} - z^{k} \right\rVert _{\M_{2}^{k}}^{2} \\
	= & \ h \left( x^{k+1} \right) + g \left( z^{k+1} \right) - \dfrac{1}{2r} \left\lVert y^{k+1} \right\rVert ^{2} + \dfrac{r}{2} \left\lVert A x^{k+1} - z^{k+1} + \dfrac{1}{r} y^{k+1} \right\rVert ^{2} \\
	& \ + \Tdecy \left\lVert A^{*} \left( y^{k+1} - y^{k} \right) \right\rVert ^{2} + \dfrac{1}{2} \left\lVert x^{k+1} - x^{k} \right\rVert _{\M_{3}^{k} - C_0\Id}^{2} + \dfrac{1}{2} \left\lVert z^{k+1} - z^{k} \right\rVert _{\M_{2}^{k}}^{2} + \dfrac{\Cdecxx}{2}\|x^{k+1} - x^k\|,
	\end{split}
	\end{align*}
	which, thanks to \eqref{est:mult}, leads to
	\begin{align}
	\label{bound:upper}
	\F_{1} \geq & \ h \left( x^{k+1} \right) + g \left( z^{k+1} \right) - \dfrac{\Tdecx}{r} \left\lVert \nabla h \left( x^{k+1} \right) \right\rVert ^{2} + \dfrac{r}{2} \left\lVert A x^{k+1} - z^{k+1} + \dfrac{1}{r} y^{k+1} \right\rVert ^{2} \nonumber \\
	& \ + \frac{\Tdecy}{2} \left\lVert A^{*} \left( y^{k+1} - y^{k} \right) \right\rVert ^{2} + \dfrac{1}{2} \left\lVert x^{k+1} - x^{k} \right\rVert _{\M_{3}^{k} - \Cdecxx \Id}^{2} + \dfrac{1}{2} \left\lVert z^{k+1} - z^{k} \right\rVert _{\M_{2}^{k}}^{2} + \dfrac{\Cdecxx}{4}\|x^{k+1} - x^k\|^2.
	\end{align}
	Next we will prove the boundedness of $\left\lbrace \left( x^{k} , z^{k} , y^{k} \right) \right\rbrace _{k \geq 0}$ under each of the two scenarios.
	
	\item[(B-I)] Since $r \geq 4 \Tdecx L$, there exists $\sigma > 0$ such that
	\begin{equation*}
	\dfrac{1}{\sigma} - \dfrac{L}{2 \sigma^{2}} = \dfrac{\Tdecx}{r} .
	\end{equation*}
	From Proposition \ref{prop:lips-semiconvex} and the relation \eqref{bound:upper} we see that for every $k \geq 1$
	\begin{align*}
	& \ g \left( z^{k+1} \right) + \dfrac{r}{2} \left\lVert A x^{k+1} - z^{k+1} + \dfrac{1}{r} y^{k+1} \right\rVert ^{2} + \dfrac{\Cdecxx}{4}\|x^{k+1} - x^k\|^2 \\
	\leq & \ \F_{1} - \inf\limits_{x \in \sR^{n}} \left\lbrace h \left( x \right) - \dfrac{\Tdecx}{r} \left\lVert \nabla h \left( x \right) \right\rVert ^{2} \right\rbrace < + \infty .
	\end{align*}
 Since $g$ is coercive, it follows that the sequence $\left\lbrace z^{k} \right\rbrace _{k \geq 0}$ is bounded. On the other hand, since $g$ is bounded from below, it follows that the sequences $\left\lbrace A x^{k} - z^{k} + r^{-1} y^{k} \right\rbrace _{k \geq 0}$ and $\left\lbrace x^{k+1}-x^{k} \right\rbrace _{k \geq 0}$ are bounded as well. In addition, since for every $k \geq 0$ it holds
 	\begin{equation*}
 	\left\lVert A \left( x^{k+1}-x^{k} \right) - \left( z^{k+1}-z^{k} \right) \right\rVert \leq \left\lVert A \right\rVert \cdot \left\lVert x^{k+1}-x^{k} \right\rVert + \left\lVert z^{k+1} \right\rVert + \left\lVert z^{k} \right\rVert
 	\end{equation*}
it follows that $\left\lbrace A \left( x^{k+1}-x^{k} \right) - \left( z^{k+1}-z^{k} \right) \right\rbrace_{k \geq 0}$ is bounded, thus  $\left\lbrace r^{-1} \left( y^{k+1}-y^{k} \right) \right\rbrace _{k \geq 0}$ is bounded. According to the third update in the iterative scheme we obtain that $\left\lbrace A x^{k} - z^{k} \right\rbrace _{k \geq 0}$ is bounded and from here that $\left\lbrace y^{k} \right\rbrace _{k \geq 0}$ is also bounded. This implies the boundedness of $\left\lbrace A x^{k} \right\rbrace _{k \geq 0}$ and, finally, since $A$ is invertible, the boundedness of  $\left\lbrace x^{k} \right\rbrace _{k \geq 0}$.
	
	\item[(B-II)] Again thanks to \eqref{assume:r:positive} there exists $\sigma > 0$ such that
	\begin{equation*}
	\dfrac{1}{\sigma} - \dfrac{L}{2 \sigma^{2}} = \dfrac{3 \Tdecx}{2 r}.
	\end{equation*}
	We assume first that $\rho \neq 1$ or, equivalently, $\Tdecy  \neq 0$. From Proposition \ref{prop:lips-semiconvex} and \eqref{bound:upper} we see that for every $k \geq 1$
	\begin{align*}
	& \dfrac{1}{2} h \left( x^{k+1} \right) + \dfrac{\Tdecx}{4 r} \left\lVert \nabla h \left( x^{k+1} \right) \right\rVert ^{2} + \dfrac{r}{2} \left\lVert A x^{k+1} - z^{k+1} + \dfrac{1}{r} y^{k+1} \right\rVert ^{2} + \dfrac{\Tdecy }{2} \left\lVert A^{*} \left( y^{k+1} - y^{k} \right) \right\rVert \\
	\leq & \ \F_{1} - \inf\limits_{z \in \sR^{m}} g \left( z \right) - \dfrac{1}{2} \inf\limits_{x \in \sR^{n}} \left\lbrace h \left( x \right) - \dfrac{3 \Tdecx}{2r} \left\lVert \nabla h \left( x \right) \right\rVert ^{2} \right\rbrace < + \infty .
	\end{align*}
	Since $h$ is coercive and bounded from below, we obtain that $\left\lbrace x^{k} \right\rbrace _{k \geq 0}$, $\left\lbrace A x^{k} - z^{k} + r^{-1} y^{k} \right\rbrace _{k \geq 0}$ and $\left\lbrace  A^{*} \left( y^{k+1} - y^{k} \right) \right\rbrace _{k \geq 0}$ are bounded. For every $k \geq 0$ we have that 
	\begin{equation*}
	\lambda_{\min} \left( A^{*}A \right) \rho^{2} r^{2} \left\lVert Ax^{k+1}-z^{k+1} \right\rVert ^{2} = \lambda_{\min} \left( A^{*}A \right) \left\lVert y^{k+1}-y^{k} \right\rVert ^{2} \leq \left\lVert A^{*} \left( y^{k+1}-y^{k} \right) \right\rVert ^{2},
	\end{equation*}
        thus $\left\lbrace A x^{k} - z^{k} \right\rbrace _{k \geq 0}$ is bounded. Consequently, $\left\lbrace y^{k} \right\rbrace _{k \geq 0}$ and $\left\lbrace z^{k} \right\rbrace _{k \geq 0}$ are bounded.
        
        In case $\rho = 1$ or, equivalently, $\Tdecy  = 0$, we have that for every $k \geq 1$
	\begin{align*}
	& \dfrac{1}{2} h \left( x^{k+1} \right) + \dfrac{\Tdecx}{4 r} \left\lVert \nabla h \left( x^{k+1} \right) \right\rVert ^{2} + \dfrac{r}{2} \left\lVert A x^{k+1} - z^{k+1} + \dfrac{1}{r} y^{k+1} \right\rVert ^{2} \\
	\leq & \ \F_{1} - \inf\limits_{z \in \sR^{m}} g \left( z \right) - \dfrac{1}{2} \inf\limits_{x \in \sR^{n}} \left\lbrace h \left( x \right) - \dfrac{3 \Tdecx}{2r} \left\lVert \nabla h \left( x \right) \right\rVert ^{2} \right\rbrace < + \infty
	\end{align*}
        from which we deduce that $\left\lbrace x^{k} \right\rbrace _{k \geq 0}$ and $\left\lbrace A x^{k} - z^{k} + r^{-1} y^{k} \right\rbrace _{k \geq 0}$ are bounded. From Lemma \ref{lem:estimate} (iii), which now reads
        \begin{equation*}
        \left\lVert y^{k+1} - y^{k} \right\rVert \leq \Cestx \left\lVert x^{k+1} - x^{k} \right\rVert + \Cestxx \left\lVert x^{k} - x^{k-1} \right\rVert \ \forall k \geq 1,
        \end{equation*}
        it yields that $\left\lbrace y^{k+1} - y^k \right\rbrace _{k \geq 0}$ is bounded, thus,
        $\left\lbrace A x^{k} - z^{k}\right\rbrace _{k \geq 0}$ is bounded. Consequently, $\left\lbrace y^{k} \right\rbrace _{k \geq 0}$ and $\left\lbrace z^{k} \right\rbrace _{k \geq 0}$ are bounded.
        
Both considered scenarios lead to the  conclusion that the sequence $\left\lbrace \left( x^{k} , z^{k} , y^{k} \right) \right\rbrace _{k \geq 0}$ is bounded.
\end{proof}

We state now the first convergence result of this paper.
\begin{thm}
	\label{thm:limit}
Suppose that Assumption \ref{ass} holds true and let $\left\lbrace \left( x^{k} , z^{k} , y^{k} \right) \right\rbrace _{k \geq 0}$ be a sequence generated by Algorithm \ref{algo:al.1} or Algorithm \ref{algo:al.2}, which is assumed to be bounded.  The following statements are true:
	\begin{enumerate}
\item[(i)] 
		\label{thm:limit:i}
		for every $k \geq 1$ it holds
		\begin{equation}
		\label{limit:dec:Hk}
		\F_{k+1} + \dfrac{\Cdecxx}{4} \left\lVert x^{k+1} - x^{k} \right\rVert ^{2} + \dfrac{1}{2} \left\lVert z^{k+1} - z^{k} \right\rVert _{\M_{2}^{k}}^{2} \leq \F_{k};
		\end{equation}
		
\item[(ii)] 
		\label{thm:limit:ii}
		the sequence $\left\lbrace \F_{k} \right\rbrace _{k \geq 0}$ is bounded from below and convergent. In addition,
		\begin{equation}
		\label{limit:vanish}
		x^{k+1} - x^{k} \to 0, \ z^{k+1} - z^{k} \to 0 \textrm{ and } y^{k+1} - y^{k} \to 0 \textrm{ as } k \to + \infty;
		\end{equation}
		
\item[(iii)] 
		\label{thm:limit:iii}
		the sequences $\left\lbrace \F_{k} \right\rbrace _{k \geq 0}$, $\left\lbrace \Lr \left( x^{k} , z^{k} , y^{k} \right) \right\rbrace _{k \geq 0}$ and $\left\lbrace h \left( x^{k} \right) + g \left( z^{k} \right) \right\rbrace _{k \geq 0}$ have the same limit, which we denote by $\F_{*} \in \sR$.
	\end{enumerate}
\end{thm}
\begin{proof} 
\begin{enumerate}
		\item[(i)]  According to \eqref{assume:vm:positive} we have that $\M_{3}^{k} - \Cdecxx \Id \in \sP_{\frac{\Cdecxx}{2}}^{n}$ and thus  \eqref{bound:dec:Hk} implies \eqref{limit:dec:Hk}.
		
		\item[(ii)]  We will show that $\left\lbrace \Lr \left( x^{k} , z^{k} , y^{k} \right) \right\rbrace _{k \geq 0}$ is bounded from below, which will imply that $\left\lbrace \F_{k} \right\rbrace _{k \geq 0}$ is bounded from below as well. Assuming the contrary, as $\left\lbrace \left( x^{k} , z^{k} , y^{k} \right) \right\rbrace _{k \geq 0}$ is bounded, there exists a subsequence $\left\lbrace \left( x^{k_{q}} , z^{k_{q}} , y^{k_{q}} \right) \right\rbrace _{q \geq 0}$ converging to an element $\left( \widehat{x} , \widehat{z} , \widehat{y} \right) \in \sR^{n} \times \sR^{m} \times \sR^{m}$ such that $\left\lbrace \Lr \left( x^{k_{q}} , z^{k_{q}} , y^{k_{q}} \right) \right\rbrace _{q \geq 0}$ converges to $-\infty$ as $q \to +\infty$. However, using the lower semicontinuity of $g$ and the continuity of $h$, we obtain
		\begin{equation*}
		\liminf\limits_{q \to +\infty} \Lr \left( x^{k_{q}} , z^{k_{q}} , y^{k_{q}} \right) \geq h \left( \widehat x \right) + g \left( \widehat z \right) + \left\langle \widehat y , A \widehat x - \widehat z \right\rangle + \dfrac{r}{2} \left\lVert A \widehat x - \widehat z \right\rVert ^{2} ,
		\end{equation*}
		which leads to a contradiction. From Lemma \ref{lem:conv-pre} we conclude that $\left\lbrace \F_{k} \right\rbrace _{k \geq 1}$ is convergent and
		\begin{equation*}
		\mysum_{k \geq 0} \left\lVert x^{k+1} - x^{k} \right\rVert ^{2} < +\infty,
		\end{equation*}
thus $x^{k+1} - x^{k} \to 0$ as $k \to + \infty$. 

We proved in \eqref{dec:al.1:p1}, \eqref{dec:al.1:p2}, \eqref{dec:al.2:p1} and \eqref{dec:al.2:p2} that for every $k \geq 1$
		\begin{align*}
		\dfrac{1}{\rho r} \left\lVert y^{k+1} - y^{k} \right\rVert ^{2} & \leq \dfrac{\Cdecx - L}{2} \left\lVert x^{k+1} - x^{k} \right\rVert ^{2} + \dfrac{\Cdecxx}{2} \left\lVert x^{k} - x^{k-1} \right\rVert ^{2} \\
		& \quad + \Tdecy \left\lVert A^{*} \left( y^{k} - y^{k-1} \right) \right\rVert ^{2} - \Tdecy \left\lVert A^{*} \left( y^{k+1} - y^{k} \right) \right\rVert ^{2} .
		\end{align*}
Summing up the above inequality for $k = 1 , \ldots , K$, for $K > 1$, we get
		\begin{align*}
		\dfrac{1}{\rho r} \mysum_{k = 1}^{K} \left\lVert y^{k+1} - y^{k} \right\rVert ^{2} & \leq \dfrac{\Cdecx - L}{2} \mysum_{k = 1}^{K} \left\lVert x^{k+1} - x^{k} \right\rVert ^{2} + \dfrac{\Cdecxx}{2} \mysum_{k = 1}^{K} \left\lVert x^{k} - x^{k-1} \right\rVert ^{2} \\
		& \quad + \Tdecy \left\lVert A^{*} \left( y^{1} - y^{0} \right) \right\rVert ^{2} - \Tdecy \left\lVert A^{*} \left( y^{K+1} - y^{K} \right) \right\rVert ^{2} \\
& \leq \dfrac{\Cdecx - L}{2} \mysum_{k = 1}^{K} \left\lVert x^{k+1} - x^{k} \right\rVert ^{2} + \dfrac{\Cdecxx}{2} \mysum_{k = 1}^{K} \left\lVert x^{k} - x^{k-1} \right\rVert ^{2}  + \Tdecy \left\lVert A^{*} \left( y^{1} - y^{0} \right) \right\rVert ^{2}.
		\end{align*}
We let $K$ converge to $+ \infty$ and conclude
		\begin{equation*}
		\rho r \mysum_{k \geq 0} \left\lVert A x^{k+1} - z^{k+1} \right\rVert ^{2} = \dfrac{1}{\rho r} \mysum_{k \geq 0} \left\lVert y^{k+1} - y^{k} \right\rVert ^{2} < +\infty ,
		\end{equation*}
		thus $A x^{k+1} - z^{k+1} \to 0$ and $y^{k+1} - y^{k} \to 0$ as $k \to + \infty$. Since $x^{k+1} - x^{k} \to 0$ as $k \to + \infty$, it follows that $z^{k+1} - z^{k} \to 0$ as $k \to + \infty$.
		
		\item[(iii)]  By using \eqref{limit:vanish} and the fact that $\left\lbrace y^{k} \right\rbrace _{k \geq 0}$ is bounded, it follows
		\begin{equation*}
		\F_{*} = \lim\limits_{k \to +\infty} \F_{k} = \lim\limits_{k \to +\infty} \Lr \left( x^{k} , z^{k} , y^{k} \right) = \lim\limits_{k \to +\infty} \left\lbrace h \left( x^{k} \right) + g \left( z^{k} \right) \right\rbrace ,
		\end{equation*}
		which is the desired statement.
		\qedhere
	\end{enumerate}
\end{proof}

The following lemmas provides upper estimates in terms of the iterates for limiting subgradients of the augmented Lagrangian and the regularized augmented Lagrangian $\Fr$, respectively.

\begin{lem}
	\label{lem:subd-Lr-bound}
Suppose that Assumption \ref{ass} holds true and let $\left\lbrace \left( x^{k} , z^{k} , y^{k} \right) \right\rbrace _{k \geq 0}$ be a sequence generated by Algorithm \ref{algo:al.1} or Algorithm \ref{algo:al.2}.  For every $k \geq 0$ we have
	\begin{equation}
	\label{aLr:subd:all}
	d^{k+1} := \left( d_{x}^{k+1} , d_{z}^{k+1} , d_{y}^{k+1} \right) \in  \partial \Lr \left( x^{k+1} , z^{k+1} , y^{k+1} \right),
	\end{equation}
	where
\begin{subequations}
		\begin{align}		
		\label{aLr:subd:x}
		d_{x}^{k+1} & := \Cact \left( \nabla h \left( x^{k+1} \right) - \nabla h \left( x^{k} \right) \right) + A^{*} \left( y^{k+1} - y^{k} \right) + \M_{1}^{k} \left( x^{k} - x^{k+1} \right),\\
		\label{aLr:subd:z}
		d_{z}^{k+1} & := y^{k} - y^{k+1} + r A \left( x^{k} - x^{k+1} \right) + \M_{2}^{k} \left( z^{k} - z^{k+1} \right) ,\\
		\label{aLr:subd:y}
		d_{y}^{k+1} & := \dfrac{1}{\rho r} \left( y^{k+1} - y^{k} \right).
		\end{align}
\end{subequations}
and

\begin{align*}
    \Cact := \begin{cases}
	0, & \mbox{for Algorithm} \ \ref{algo:al.1},\\
    1, & \mbox{for Algorithm} \ \ref{algo:al.2}.
\end{cases} 
\end{align*}

 Moreover, for every $k \geq 0$ it holds	
	\begin{equation}
	\label{aLr:inq}
	\opnorm{d^{k+1}} \leq \CaLrx \left\lVert x^{k+1} - x^{k} \right\rVert + \CaLrz \left\lVert z^{k+1} - z^{k} \right\rVert + \CaLry \left\lVert y^{k+1} - y^{k} \right\rVert ,
	\end{equation}
	where
	\begin{equation*}
	\label{aLr:const}
	\CaLrx := \Cact L + \mu_{1} + r \left\lVert A \right\rVert , \quad \CaLrz := \mu_{2} , \quad \CaLry := 1 + \left\lVert A \right\rVert + \dfrac{1}{\rho r} .
	\end{equation*}
\end{lem}
\begin{proof}
	Let $k \geq 0$ be fixed. Applying the calculus rules of the limiting subdifferential, we obtain
	\begin{subequations}
		\label{aLr:subd:form}
		\begin{align}
		\begin{split}
		\label{aLr:subd:form:x}
		\nabla_{x} \Lr \left( x^{k+1} , z^{k+1} , y^{k+1} \right)		& = \nabla h \left( x^{k+1} \right) + A^{*} y^{k+1} + r A^{*} \left( A x^{k+1} - z^{k+1} \right) ,
		\end{split}
		\\
		\begin{split}
		\label{aLr:subd:form:z}
		\partial_{z} \Lr \left( x^{k+1} , z^{k+1} , y^{k+1} \right) 	& = \partial g \left( z^{k+1} \right) - y^{k+1}  - r \left( A x^{k+1} - z^{k+1} \right) ,
		\end{split}
		\\
		\begin{split}
		\label{aLr:subd:form:y}
		\nabla_{y} \Lr \left( x^{k+1} , z^{k+1} , y^{k+1} \right)	& = A x^{k+1} - z^{k+1} .
		\end{split}
		\end{align}
	\end{subequations}	
	Then \eqref{aLr:subd:y} follows directly from \eqref{aLr:subd:form:y} and \eqref{al.1:algo:y}, respectively, \eqref{al.2:algo:y}, while \eqref{aLr:subd:z} follows from
	\begin{equation*}
	y^{k} + r (A x^{k} - z^{k+1}) + \M_{2}^{k} \left( z^{k} - z^{k+1} \right) \in \partial g \left( z^{k+1} \right),
	\end{equation*}
	which is a consequence of the optimality criterion of  \eqref{al.1:algo:z} and \eqref{al.2:algo:z}, respectively. In order to derive \eqref{aLr:subd:x}, let us notice that for Algorithm \ref{algo:al.1} we have (see \eqref{dec:al.1:opt.con-f}) 
\begin{align}
		\label{aLr:al.1:opt.con-f}
		- A^{*} y^{k} + \M_{1}^{k} \left( x^{k} - x^{k+1} \right) & = \nabla h \left( x^{k+1} \right) + r A^{*} \left( A x^{k+1} - z^{k+1} \right),
		\end{align}
while for  Algorithm \ref{algo:al.2} we have  (see \eqref{dec:al.2:opt.con-f})
		\begin{align}
		\label{aLr:al.2:opt.con-f}
		- \nabla h \left( x^{k} \right) - A^{*} y^{k} + \M_{1}^{k} \left( x^{k} - x^{k+1} \right) & = r A^{*} \left( A x^{k+1} - z^{k+1} \right) .
		\end{align}
By using  \eqref{aLr:subd:form:x} we get the desired statement. 

Relation \eqref{aLr:inq} follows by combining the inequalities
	\begin{align*}
	\begin{split}
	\left\lVert d_{x}^{k+1} \right\rVert  
	\leq & \ \left( \Cact L + \mu_{1} \right) \left\lVert x^{k+1} - x^{k} \right\rVert + \left\lVert A \right\rVert \cdot \left\lVert y^{k+1} - y^{k} \right\rVert ,
	\end{split}
	\\
	\begin{split}
	\left\lVert d_{z}^{k+1} \right\rVert 
	\leq & \ \left\lVert y^{k} - y^{k+1} \right\rVert + r \left\lVert A \right\rVert \cdot \left\lVert x^{k+1} - x^{k} \right\rVert + \mu_{2} \left\lVert z^{k+1} - z^{k} \right\rVert
	\end{split}
	\end{align*}
and \eqref{intro:norm-inequality}.
\end{proof}

\begin{lem}
	\label{lem:subd-Hk-bound}
Suppose that Assumption \ref{ass} holds true and let $\left\lbrace \left( x^{k} , z^{k} , y^{k} \right) \right\rbrace _{k \geq 0}$ be a sequence generated by Algorithm \ref{algo:al.1} or Algorithm \ref{algo:al.2}.  For every $k \geq 0$ we have
	\begin{align}
	\label{aHk:subd:all}	
	D^{k+1} := \left( D_{x}^{k+1} , D_{z}^{k+1} , D_{y}^{k+1} , D_{x'}^{k+1} , D_{y'}^{k+1} \right) 
	\in  \partial \Fr \left( x^{k+1} , z^{k+1} , y^{k+1} , x^{k} , y^{k} \right)
	\end{align}
	where
	\begin{align}
	\label{aHk:subd}
	D_{x}^{k+1} := d_{x}^{k+1} + \Cdecxx \left( x^{k+1} - x^{k} \right) , & \quad D_{z}^{k+1} := d_{z}^{k+1} , \quad D_{y}^{k+1} := d_{y}^{k+1} + 2\Tdecy  A A^{*} \left( y^{k+1} - y^{k} \right) , \nonumber \\
	D_{x'}^{k+1} := - \Cdecxx \left( x^{k+1} - x^{k} \right) , & \quad D_{y'}^{k+1} := - 2\Tdecy  A A^{*} \left( y^{k+1} - y^{k} \right) .
	\end{align}
	 Moreover, for every $k \geq 0$ it holds	
	\begin{equation}
	\label{aHk:inq}
	\opnorm{D^{k+1}} \leq \CaHkx \left\lVert x^{k+1} - x^{k} \right\rVert + \CaHkz \left\lVert z^{k+1} - z^{k} \right\rVert + \CaHky \left\lVert y^{k+1} - y^{k} \right\rVert ,
	\end{equation}
	where 
	\begin{equation*}
	\label{aHk:const}
	\CaHkx := 2 \Cdecxx + \CaLrx, \quad \CaHkz := \CaLrz , \quad \CaHky := \CaLry + 4\Tdecy \left\lVert A \right\rVert ^{2} .
	\end{equation*}
\end{lem}
\begin{proof}
	Let $k \geq 0$ be fixed. Applying the calculus rules of the limiting subdifferential it follows
	\begin{subequations}
		\label{aHk:subd:form}	
		\begin{align}		
		\begin{split}
		\label{aHk:subd:form:x}
\nabla_{x} \Fr \left( x^{k+1} , z^{k+1} , y^{k+1} , x^{k} , y^{k} \right)	& := \nabla_{x} \Lr \left( x^{k+1} , z^{k+1} , y^{k+1} \right) + \Cdecxx \left( x^{k+1} - x^{k} \right) ,
		\end{split}
		\\
		\begin{split}
		\label{aHk:subd:form:z}
		\partial_{z} \Fr \left( x^{k+1} , z^{k+1} , y^{k+1} , x^{k} , y^{k} \right)	& := \partial_{z} \Lr \left( x^{k+1} , z^{k+1} , y^{k+1} \right)
		\end{split}
		\\
		\begin{split}
		\label{aHk:subd:form:y}
		\nabla_{y} \Fr \left( x^{k+1} , z^{k+1} , y^{k+1} , x^{k} , y^{k} \right)	& := \nabla_{y} \Lr \left( x^{k+1} , z^{k+1} , y^{k+1} \right) + 2\Tdecy  A A^{*} \left( y^{k+1} - y^{k} \right) ,
		\end{split}
		\\
		\begin{split}
		\label{aHk:subd:form:xx}
		\nabla_{x'}  \Fr \left( x^{k+1} , z^{k+1} , y^{k+1} , x^{k} , y^{k} \right)	& :=  - \Cdecxx \left( x^{k+1} - x^{k} \right) ,
		\end{split}
		\\
		\begin{split}
		\label{aHk:subd:form:yy}
		\nabla_{y'}  \Fr \left( x^{k+1} , z^{k+1} , y^{k+1} , x^{k} , y^{k} \right)	& := - 2\Tdecy  A A^{*} \left( y^{k+1} - y^{k} \right) .
		\end{split}
		\end{align}
Then \eqref{aHk:subd:all}  follows directly from the above relations and \eqref{aLr:subd:all}. Inequality \eqref{aHk:inq} follows by combining
		\begin{align*}
		\begin{split}
		\left\lVert D_{x}^{k+1} \right\rVert  
		\leq & \ \left\lVert d_{x}^{k+1} \right\rVert + \Cdecxx \left\lVert x^{k+1} - x^{k} \right\rVert ,
		\end{split}
		\\
		\begin{split}
		\left\lVert D_{y}^{k+1} \right\rVert 
		\leq & \ \left\lVert d_{y}^{k+1} \right\rVert + 2\Tdecy \left\lVert A \right\rVert ^{2} \cdot \left\lVert y^{k+1} - y^{k} \right\rVert .
		\end{split}
		\end{align*}
and \eqref{intro:norm-inequality}.
	\end{subequations}
\end{proof}

The following result is a straightforward consequence of Lemma \ref{lem:estimate} and Lemma \ref{lem:subd-Hk-bound}.

\begin{coro}
	\label{coro:subd-Hk-bound}
	Suppose that Assumption \ref{ass} holds true and let $\left\lbrace \left( x^{k} , z^{k} , y^{k} \right) \right\rbrace _{k \geq 0}$ be a sequence generated by Algorithm \ref{algo:al.1} or Algorithm \ref{algo:al.2}.  Then the norm of the element $D^{k+1} \in  \partial \Fr \left( x^{k+1} , z^{k+1} , y^{k+1} , x^{k} , y^{k} \right)$ defined in the previous lemma verifies for every $k \geq 2$ the following estimate
\begin{align}
	\label{subd:inq}
	\opnorm{D^{k+1}} \leq & \ \Csubx \left( \left\lVert x^{k+1} - x^{k} \right\rVert + \left\lVert x^{k} - x^{k-1} \right\rVert + \left\lVert x^{k-1} - x^{k-2} \right\rVert \right) \nonumber \\
	& \ + \Csuby \left( \left\lVert A^{*} \left( y^{k} - y^{k-1} \right) \right\rVert - \left\lVert A^{*} \left( y^{k+1} - y^{k} \right) \right\rVert \right) \nonumber \\
	& \ + \Csubyy \left( \left\lVert A^{*} \left( y^{k-1} - y^{k-2} \right) \right\rVert - \left\lVert A^{*} \left( y^{k} - y^{k-1} \right) \right\rVert \right) ,
\end{align}
where
	\begin{align*}		
	\label{subd:const}
	& \Csubx := \max \left\lbrace \CaHkx + \CaHkz \left\lVert A \right\rVert + \Cestx \CaHky + \dfrac{\Cestx \CaHkz}{\rho r} , \Cestxx \CaHky+ \dfrac{\Cestx \CaHkz}{\rho r}, \dfrac{\Cestxx \CaHkz}{\rho r} \right\rbrace, \nonumber \\
& \Csuby := \left( \CaHky + \dfrac{\CaHkz}{\rho r} \right) \Cesty, \qquad \Csubyy := \dfrac{\CaHkz \Cesty}{\rho r}.
	\end{align*}
\end{coro}

In the following, we denote by $\omega \left( \left\lbrace  u^{k}  \right\rbrace _{k \geq 0} \right)$ the set of \emph{cluster points} of the sequence $\left\lbrace u^{k} \right\rbrace _{k \geq 0 } \subseteq \sR^{N}$.

\begin{lem}
	\label{lem:cluster-aLr}
	Suppose that Assumption \ref{ass} holds true and let $\left\lbrace \left( x^{k} , z^{k} , y^{k} \right) \right\rbrace _{k \geq 0}$ be a sequence generated by Algorithm \ref{algo:al.1} or Algorithm \ref{algo:al.2}, which is assumed to be bounded.  The following statements are true:
	\begin{enumerate}
\item[(i)] 
		\label{lem:cluster-aLr:i}
		if $\left\lbrace (x^{k_q}, z^{k_q}, y^{k_q}) \right\rbrace _{q \geq 0}$ is a subsequence of $\left\lbrace \left( x^{k}, z^{k} , y^{k} \right) \right\rbrace _{k \geq 0}$ which converges to $(\widehat{x}, \widehat{z}, \widehat{y})$  as $q \to +\infty$, then
		\begin{equation*}
%		\label{cluster-aLr:lim}
		\lim\limits_{q \to \infty} \Lr \left( x^{k_q}, z^{k_q}, y^{k_q} \right) = \Lr \left( \widehat{x}, \widehat{z}, \widehat{y} \right);
		\end{equation*}	
\item[(ii)] 
		\label{lem:cluster-aLr:ii}
		it holds
		\begin{align*}
%		\label{cluster-aLr:KKT-set}
		\omega \left( \left\lbrace \left( x^{k}, z^{k} , y^{k} \right) \right\rbrace _{k \geq 0} \right) \subseteq & \ \crit \left(\Lr \right) \nonumber \\ 
		\subseteq & \ \left\lbrace \left( \widehat{x} , \widehat{z} , \widehat{y} \right) \in \sR^{n} \times \sR^{m} \times \sR^{m} \colon - A^{*} \widehat{y} = \nabla h \left( \widehat{x} \right), \widehat{y} \in \partial g \left( \widehat{z} \right), \widehat{z} = A \widehat{x} \right\rbrace;
		\end{align*}		
\item[(iii)] 
		\label{lem:cluster-aLr:iii}
		we have $\lim\limits_{k \to +\infty} \dist \left[\left( x^{k}, z^{k} , y^{k} \right), \omega \left( \left\lbrace \left( x^{k}, z^{k} , y^{k} \right) \right\rbrace _{k \geq 0} \right)\right] = 0$;
\item[(iv)] 
		\label{lem:cluster-aLr:iv}
		the set $\omega \left( \left\lbrace \left( x^{k}, z^{k} , y^{k} \right) \right\rbrace _{k \geq 0} \right)$ is nonempty, connected and compact;		
\item[(v)] 
		\label{lem:cluster-aLr:v}
		the function $\Lr$  takes on $\omega \left( \left\lbrace \left( x^{k}, z^{k} , y^{k} \right) \right\rbrace _{k \geq 0} \right)$ the value $\F_{*} = \lim_{k \to +\infty} \Lr \left( x^{k} , z^{k} , y^{k} \right)$, as the objective function $g \circ A + h$ does on the projection of the set $\omega \left( \left\lbrace \left( x^{k}, z^{k} , y^{k} \right) \right\rbrace _{k \geq 0} \right)$ onto the space $\sR^n$ corresponding to the first component.
	\end{enumerate}
\end{lem}
\begin{proof}
	Let $\left( \widehat{x}, \widehat{z}, \widehat{y} \right) \in \omega \left( \left\lbrace \left( x^{k}, z^{k} , y^{k} \right) \right\rbrace _{k \geq 0} \right)$ and $\left\lbrace \left( x^{k_q}, z^{k_q}, y^{k_q} \right) \right\rbrace _{q \geq 0}$ be a subsequence of $\left\lbrace x^{k}, z^{k} , y^{k} \right\rbrace _{k \geq 0}$ converging to $\left( \widehat{x}, \widehat{z}, \widehat{y} \right)$ as $q \to +\infty$.
	
\item[(i)] 
	From either \eqref{al.1:algo:z} or \eqref{al.2:algo:z} we obtain for every $q \geq 1$
	\begin{align*}
	& g \left( z^{k_{q}} \right) + \left\langle y^{k_{q}-1} , A x^{k_{q}-1} - z^{k_{q}} \right\rangle + \dfrac{r}{2} \left\lVert A x^{k_{q}-1} - z^{k_{q}} \right\rVert ^{2} + \dfrac{1}{2} \left\lVert z^{k_{q}} - z^{k_{q}-1} \right\rVert _{\M_{2}^{k_{q}-1}} ^{2} \\
	\leq & \ g \left( \widehat{z} \right) + \left\langle y^{k_{q}-1} , A x^{k_{q}-1} - \widehat{z} \right\rangle + \dfrac{r}{2} \left\lVert A x^{k_{q}-1} - \widehat{z} \right\rVert ^{2} + \dfrac{1}{2} \left\lVert \widehat{z} - z^{k_{q}-1} \right\rVert _{\M_{2}^{k_{q}-1}} ^{2}.
	\end{align*}
	Taking the limit superior on both sides of the above inequalities we get
	\begin{equation*}
	\qquad \limsup\limits_{q \to +\infty} g \left( z^{k_{q}} \right) \leq g \left( \widehat{z} \right),
	\end{equation*}
	which, combined with the lower semicontinuity of $g$, leads to
	\begin{equation*}
	\lim\limits_{q \to +\infty} g \left( z^{k_{q}} \right) = g \left( \widehat{z} \right) .
	\end{equation*}
	Since $h$ is continuous, we further obtain
	\begin{align*}
	\label{PADMM:lim-sub}
	\begin{split}
	\lim\limits_{q \to +\infty} \Lr \left( x^{k_q}, z^{k_q}, y^{k_q} \right) & = \lim\limits_{q \to +\infty} \left[ g \left( z^{k_{q}} \right) + h \left( x^{k_{q}} \right) + \left\langle y^{k_{q}}, A x^{k_{q}} - z^{k_{q}}  \right\rangle  + \dfrac{r}{2} \left\lVert A x^{k_{q}} - z^{k_{q}} \right\rVert ^{2} \right] \\
	& = g \left( \widehat{z} \right) + h \left( \widehat{x} \right) + \left\langle \widehat{y} , A \widehat{x} - \widehat{z} \right\rangle + \dfrac{r}{2} \left\lVert A \widehat{x} - \widehat{z} \right\rVert ^{2} = \Lr \left( \widehat{x} , \widehat{z} , \widehat{y} \right) .
	\end{split}
	\end{align*}
	
\item[(ii)] 
	For the sequence $\left\lbrace d^{k} \right\rbrace _{k \geq 0}$ defined in \eqref{aLr:subd:x}-\eqref{aLr:subd:y} we have that $d^{k_q} \in \partial \Lr(x^{k_q}, z^{k_q}, y^{k_q} )$ for every $q \geq 1$ and $d^{k_q} \to 0$ as $q \to +\infty$, while $\left( x^{k_q}, z^{k_q}, y^{k_q} \right) \to \left( \widehat{x}, \widehat{z}, \widehat{y} \right)$ and $\Lr \left( x^{k_q}, z^{k_q}, y^{k_q} \right) \to \Lr \left( \widehat{x}, \widehat{z}, \widehat{y} \right)$ as $q \to +\infty$. The closedness criterion of the limiting subdifferential guarantees that $0 \in \partial \Lr \left( \widehat{x}, \widehat{z}, \widehat{y} \right)$ or, in other words, $\left( \widehat{x}, \widehat{z}, \widehat{y} \right) \in \crit \left( \Lr \right)$. Choosing now an element $\left( \widehat{x}, \widehat{z}, \widehat{y} \right) \in \crit \left( \Lr \right)$ it holds
	\begin{align*}
	0 & \ = \nabla h \left(\widehat{x} \right) + A^{*} \widehat{y} + r A^{*} \left( A \widehat{x} - \widehat{z} \right) \\
	0 & \ \in \partial g \left(\widehat{z} \right) - \widehat{y} - r \left( A \widehat{x} - \widehat{z} \right) \\
	0 & \ = A \widehat{x} - \widehat{z},
	\end{align*}
	which is further equivalent to
	\begin{equation*}
	- A^{*} \widehat{y} = \nabla h \left( \widehat{x} \right) , \ \widehat{y} \in \partial g \left( \widehat{z} \right) , \ \widehat{z} = A \widehat{x} .
	\end{equation*}
	
\item[(iii)-(iv)]  The proof follows in the lines of the proof of Theorem 5 (ii)-(iii) in \cite{Bolte-Sabach-Teboulle}, also by taking into consideration \cite[Remark 5]{Bolte-Sabach-Teboulle}, according to which the properties in (iii) and (iv) are generic for sequences satisfying $\left( x^{k+1} , z^{k+1} , y^{k+1} \right) - \left( x^{k} , z^{k} , y^{k} \right) \to 0$ as $k \to +\infty$, which is indeed the case due to \eqref{limit:vanish}.
	
	\item[(v)] The conclusion follows according to the first two statements of this theorem and of the third statement of  Theorem \ref{thm:limit}.
\end{proof}

\begin{rmk}
	An element $\left( \widehat{x} , \widehat{z} , \widehat{y} \right) \in \sR^{n} \times \sR^{m} \times \sR^{m}$ fulfilling 
	\begin{equation*}
	-A^{*} \widehat{y} = \nabla h \left( \widehat{x} \right), \ \widehat{y} \in \partial g \left( \widehat{z} \right), \ \widehat{z} = A \widehat{x}
	\end{equation*} 	
	is a so-called \emph{KKT point} of the optimization problem \eqref{intro:problem}. For such a KKT point we have
	\begin{equation}\label{convexsub}
	0 = A^{*} \partial g \left( A \widehat{x} \right) + \nabla h \left( \widehat{x} \right).
	\end{equation}
	
	When $A$ is injective this is further equivalent to
	\begin{equation}\label{convexsub2}
	0 \in \partial (g \circ A)(\widehat{x}) + \nabla h \left( \widehat{x} \right) = \partial \left( g \circ A + h \right) \left( \widehat{x} \right),
	\end{equation}
	in other words, $\widehat x$ is a \emph{critical point} of the optimization problem \eqref{intro:problem}.
	
If the functions $g$ and $h$ are convex, then \eqref{convexsub} and \eqref{convexsub2} are equivalent, which means that $\widehat{x}$ is a \emph{global optimal solution} of the optimization problem \eqref{intro:problem}. In this case, $\widehat{y}$ is a \emph{global optimal solution} of the Fenchel dual problem of \eqref{intro:problem}.
\end{rmk}

By combining  Lemma \ref{lem:subd-Hk-bound}, Theorem \ref{thm:limit} and Lemma \ref{lem:cluster-aLr}, one obtains the following result.

\begin{lem}
	\label{lem:cluster-aHk}
	Suppose that Assumption \ref{ass} holds true and let $\left\lbrace \left( x^{k} , z^{k} , y^{k} \right) \right\rbrace _{k \geq 0}$ be a sequence generated by Algorithm \ref{algo:al.1} or Algorithm \ref{algo:al.2}, which is assumed to be bounded.  Let $\Omega := \omega \left( \left\lbrace \left( x^{k} , z^{k} , y^{k} , x^{k-1} , y^{k-1} \right) \right\rbrace _{k \geq 1} \right)$. The following statements are true:
	\begin{enumerate}
\item[(i)] 
		\label{lem:cluster-aHk:i}
		it holds
		\begin{equation*}
		\Omega \subseteq \left\lbrace \left( \widehat{x} , \widehat{z} , \widehat{y} , \widehat{x} , \widehat{y} \right) \in \sR^{n} \times \sR^{m} \times \sR^{m} \times \sR^{n} \times \sR^{m} \colon \left( \widehat{x} , \widehat{z} , \widehat{y} \right) \in \crit \left( \Lr \right) \right\rbrace ;
		\end{equation*}
		
\item[(ii)] 
		\label{lem:cluster-aHk:ii}
		we have
		\begin{equation*}
		\lim\limits_{k \to +\infty} \dist \left[ \left( x^{k} , z^{k} , y^{k} , x^{k-1} , y^{k-1} \right), \Omega \right] = 0 ;
		\end{equation*}
		
\item[(iii)] 
		\label{lem:cluster-aHk:iii}
		the set $\Omega$ is nonempty, connected and compact;
		
\item[(iv)] 
		\label{lem:cluster-aHk:iv}
		the regularized augmented Lagrangian $\Fr$ takes on $\Omega$ the value $\F_{*} = \lim_{k \to +\infty} \F_{k}$, as the objective function $g \circ A + h$ does on the projection of the set $\Omega$ onto the space $\sR^n$ corresponding to the first component.
	\end{enumerate}
\end{lem}

\subsection{Convergence analysis under Kurdyka-\Loja assumptions}
\label{subsec:conv-KL}

In this subsection we will prove global convergence for the sequence $\left\lbrace \left( x^{k} , z^{k} , y^{k} \right) \right\rbrace _{k \geq 0}$ generated by the two nonconvex proximal ADMM algorithms in the context of \emph{\KL \ property}. The origins of this notion go back to the pioneering work of Kurdyka who introduced in \cite{Kurdyka} a general form of the \Loja inequality (see \cite{Lojasiewicz}). A further extension to the nonsmooth setting has been proposed and studied in \cite{Bolte-Daniilidis-Lewis, Bolte-Daniilidis-Lewis-Shiota, Bolte-Daniilidis-Ley-Mazet}.

We recall that the \emph{distance function} of a given set $\Omega \subseteq \sR^{N}$ is defined for every $x$ by $\dist \left( x , \Omega \right) := \inf \left\lbrace \left\lVert x - y \right\rVert \colon y \in \Omega \right\rbrace$. 
%If $\Omega = \emptyset$, then $\dist \left( x , \Omega \right) = + \infty$.
\begin{defi}
	\label{defi:Class-Phi}
	Let $\eta \in \left( 0 , + \infty \right]$. We denote by $\Phi_{\eta}$ the set of all concave and continuous functions $\varphi \colon \left[ 0 , \eta \right) \to [0,+\infty)$ which satisfy the following conditions:
	\begin{enumerate}
		\item $\varphi \left( 0 \right) = 0$;
		\item $\varphi$ is $\sC^{1}$ on $\left( 0 , \eta \right)$ and continuous at $0$;
		\item for all $s \in \left( 0 , \eta \right): \varphi ' \left( s \right) > 0$.
	\end{enumerate}
\end{defi}
\begin{defi}	
	\label{defi:KL}
	Let $\Psi \colon \sR^{N} \to \sR \cup \left\lbrace + \infty \right\rbrace$ be proper and lower 
	semicontinuous.
	\begin{enumerate}
		\item The function $\Psi$ is said to have the Kurdyka-\Loja (\KL) property at a point $\widehat{u} \in \dom \partial \Psi := \left\lbrace u \in \sR^{N} \colon \partial \Psi \left( u \right) \neq \emptyset \right\rbrace$, if there exists $\eta \in \left( 0 , + \infty \right]$, a neighborhood $U$ of $\widehat{u}$ and a function $\varphi \in \Phi_{\eta}$ such that for every
		\begin{equation*}
		\label{intro:intersection}
		u \in U \cap \left[ \Psi \left( \widehat{u} \right) < \Psi \left( u \right) < \Psi \left( \widehat{u} \right) + \eta \right] 
		\end{equation*}
		it holds
		\begin{equation*}
		\label{intro:KL-inequality-0}
		\varphi ' \left( \Psi \left( u \right) - \Psi \left( \widehat{u} \right) \right) \cdot \dist \left( \0 , \partial \Psi \left( u \right) \right) \geq 1 .
		\end{equation*}		
		\item If $\Psi$ satisfies the \KL \ property at each point of $\dom \partial \Psi$, then $\Psi$ is called  \KL \ function.
	\end{enumerate}
\end{defi}
The functions $\varphi$ belonging to the set $\Phi_{\eta}$ for $\eta \in \left(0 , + \infty \right]$ are called \emph{desingularization functions}. The \KL \ property reveals the possibility to reparameterize the values of $\Psi$ in order to avoid flatness around the critical points. 
To the class of \KL \ functions belong semialgebraic, real subanalytic, uniformly convex functions and convex functions satisfying a growth condition. We refer the reader to \cite{Attouch-Bolte,Attouch-Bolte-Redont-Soubeyran, Attouch-Bolte-Svaiter, Bolte-Daniilidis-Lewis, Bolte-Daniilidis-Lewis-Shiota, Bolte-Daniilidis-Ley-Mazet, Bolte-Sabach-Teboulle} and to the references therein for more properties of \KL \ functions and illustrating examples.

The following result, taken from \cite[Lemma 6]{Bolte-Sabach-Teboulle}, will be crucial in our convergence analysis.
\begin{lem} \textbf{(Uniformized \KL \ property)}
	\label{lem:uniformized}
	Let $\Omega$ be a compact set and $\Psi \colon \sR^{N} \to \sR \cup \left\lbrace + \infty \right\rbrace$ be a proper and lower semicontinuous function. Assume that $\Psi$ is constant on $\Omega$ and satisfies the \KL \ property at each point of $\Omega$. Then there exist $\varepsilon > 0, \eta > 0$ 
	and $\varphi \in \Phi_{\eta}$ such that for every $\widehat{u} \in \Omega$ and  every element $u$ in the intersection
	\begin{equation*}
	\label{intro:intersection-all}
	\left\lbrace u \in \sR^{N} \colon \dist \left( u , \Omega \right) < \varepsilon \right\rbrace \cap \left[ \Psi \left( \widehat{u} \right) < \Psi \left( u \right) < \Psi \left( \widehat{u} \right) + \eta \right]
	\end{equation*}
	it holds
	\begin{equation*}
	\label{intro:KL-inequality}
	\varphi ' \left( \Psi \left( u \right) - \Psi \left( \widehat{u} \right) \right) \cdot \dist \left( \0 , \partial \Psi \left( u \right) \right) \geq 1 .
	\end{equation*}
\end{lem}

Working in the hypotheses of Lemma \ref{lem:cluster-aHk} we define for every $k \geq 1$
\begin{equation*}
\label{defi:err}
\E_{k} := \F \left( x^{k} , z^{k} , y^{k} , x^{k-1} , y^{k-1} \right) - \F_{*} = \F_k -  \F_{*} \geq 0,
\end{equation*}
where $\F_{*}$ is the limit of $\left\lbrace \F_{k} \right\rbrace _{k \geq 1}$ as $k \to +\infty$.
The sequence $\left\lbrace \E_{k} \right\rbrace _{k \geq 1}$ is monotonically decreasing and it converges to $0$ as $k \to +\infty$. 

The next result shows that, if the regularization of the augmented Lagrangian $\Fr$ is a \KL \ function, then the sequence $\left\lbrace \left(x^{k}, z^k, y^k  \right) \right\rbrace _{k \geq 0}$ converges to a KKT point of the optimization problem \eqref{intro:problem}.
\begin{thm}
	\label{thm:conv}
Suppose that Assumption \ref{ass} holds true and let $\left\lbrace \left( x^{k} , z^{k} , y^{k} \right) \right\rbrace _{k \geq 0}$ be a sequence generated by Algorithm \ref{algo:al.1} or Algorithm \ref{algo:al.2}, which is assumed to be bounded. If $\Fr$ is a \KL \ function, then the following statements are true: 
	\begin{enumerate}
\item[(i)] the sequence $\left\lbrace \left( x^{k}, z^{k}, y^{k} \right)\right\rbrace _{k \geq 0}$ has finite length, namely,
		\begin{equation}
		\label{conv:Cauchy}
		\mysum_{k \geq 0} \left\lVert x^{k+1} - x^{k} \right\rVert < + \infty , \qquad \mysum_{k \geq 0} \left\lVert z^{k+1} - z^{k} \right\rVert < + \infty , \qquad \mysum_{k \geq 0} \left\lVert y^{k+1} - y^{k} \right\rVert < + \infty;
		\end{equation}
\item[(ii)]  the sequence $\left\lbrace\left( x^{k}, z^{k}, y^{k} \right)\right\rbrace _{k \geq 0}$ converges to a KKT point of the optimization problem \eqref{intro:problem}.
	\end{enumerate}
\end{thm}
\begin{proof}
As in Lemma \ref{lem:cluster-aHk}, we denote by $\Omega := \omega \left( \left\lbrace \left( x^{k} , z^{k} , y^{k} , x^{k-1} , y^{k-1} \right) \right\rbrace _{k \geq 1} \right)$, which is a nonempty set. Let be $\left( \widehat{x} , \widehat{z} , \widehat{y} , \widehat{x} , \widehat{y} \right) \in \Omega$, thus $\Fr \left( \widehat{x} , \widehat{z} , \widehat{y} , \widehat{x} , \widehat{y} \right) = \F_{*}$. We have seen that $\left\lbrace \E_{k} = \F_k - \F^* \right\rbrace _{k \geq 1}$ converges to $0$ as $k \to +\infty$ and will consider, consequently, two cases. 

We assume first that there exists an integer $k' \geq 0$ such that $\E_{k'} = 0$ or, equivalently, $\F_{k'} = \F_{*}$. Due to the monotonicity of $\left\lbrace \E_{k} \right\rbrace _{k \geq 1}$ it follows that  $\E_{k} = 0$ or, equivalently, $\F_{k} = \F_{*}$ for all $k \geq k'$. Combining the inequality in \eqref{limit:dec:Hk} with Lemma \ref{lem:estimate}, it yields that $x^{k+1}-x^k=0$ for all $k \geq k'+1$. Using Lemma \ref{lem:estimate}  (iii) and telescoping sum arguments, it yields $\mysum\nolimits_{k \geq 0} \left\lVert y^{k+1} - y^{k} \right\rVert < + \infty$. Finally, by using Lemma \ref{lem:estimate}  (i), we obtain that
$\mysum\nolimits_{k \geq 0} \left\lVert z^{k+1} - z^{k} \right\rVert < + \infty$.
	
	Consider now the case when $\E_{k} > 0$ or, equivalently, $\F_{k} > \F_{*}$ for every $k \geq 1$. According to Lemma \ref{lem:uniformized}, there exist $\varepsilon > 0$, $\eta > 0$ and a desingularization function $\varphi$ such that for every element $u$ in the intersection
	\begin{align}
	\label{conv:intersection}
	\left\lbrace  u \in  \sR^{n} \times \sR^{m} \times \sR^{m} \times \sR^{n} \times \sR^{m}  \colon \dist \left( u, \Omega \right) < \varepsilon \right\rbrace \ & \cap \nonumber \\
\left\lbrace u \in  \sR^{n} \times \sR^{m} \times \sR^{m} \times \sR^{n} \times \sR^{m} \colon \F_{*} < \Fr \left( u \right) < \F_{*} + \eta \right\rbrace  \ &
	\end{align}
	it holds
	\begin{equation*}
	\varphi' \left( \Fr \left( u \right) - \F_{*} \right) \cdot \dist \left( \0 , \partial \Fr \left( u \right) \right) \geq 1.
	\end{equation*}
	Let be $k_{1} \geq 1$ such that for every $k \geq k_{1}$
	\begin{equation*}
	\F_{*} < \F_{k} < \F_{*} + \eta .
	\end{equation*}	
	Since $\lim\limits_{k \to +\infty} \dist \left[ \left( x^{k} , z^{k} , y^{k} , x^{k-1} , y^{k-1} \right), \Omega \right] = 0$, see Lemma \ref{lem:cluster-aHk} (ii), there exists $k_{2} \geq 1$ such that for every $k \geq k_{2}$
	\begin{equation*}
	\dist \left[ \left( x^{k} , z^{k} , y^{k} , x^{k-1} , y^{k-1} \right), \Omega \right] < \varepsilon .
	\end{equation*}	
	Thus, $\left( x^{k} , z^{k} , y^{k} , x^{k-1} , y^{k-1} \right)$ belongs to the intersection in \eqref{conv:intersection} for every $k \geq k_{0} := \max \left\lbrace k_{1} , k_{2} , 3 \right\rbrace$, which further implies
	\begin{equation}
	\label{conv:KL-property}
	\varphi ' \left( \F_{k} - \F_{*} \right) \cdot \dist \left( \0 ,  \partial \Fr \left( x^{k} , z^{k} , y^{k} , x^{k-1} , y^{k-1} \right) \right) = \varphi ' \left( \E_{k} \right) \cdot \dist \left( \0 ,  \partial \Fr \left( x^{k} , z^{k} , y^{k} , x^{k-1} , y^{k-1} \right) \right) \geq 1 .
	\end{equation}
	
	Define for two arbitrary nonnegative integers $p$ and $q$
	\begin{equation*}
	\Delta_{p,q} := \varphi \left(\F_p - \F_* \right) - \varphi \left(\F_q - \F_* \right) = \varphi \left( \E_{p} \right) - \varphi \left( \E_{q} \right) .
	\end{equation*}	
For every $K \geq k_{0} \geq 1$ it holds
	\begin{equation*}
	\mysum_{k = k_{0}}^{K} \Delta_{k,k+1} = \Delta_{k_{0},K+1} = \varphi \left(\E_{k_{0}} \right) - \varphi \left(\E_{K+1} \right) \leq  \varphi \left(\E_{k_{0}} \right),
	\end{equation*}
	from which we get $\mysum_{k \geq 1} \Delta_{k,k+1} < + \infty$.
	
By combining Theorem \ref{thm:limit} (i) with the concavity of $\varphi$ we obtain for every $k \geq 1$
	\begin{equation}
	\label{conv:concave}
	\Delta_{k,k+1} = \varphi \left(\E_{k} \right) - \varphi \left(\E_{k+1} \right) \geq \varphi ' \left(\E_{k} \right) \left[ \E_{k} - \E_{k+1} \right]  = \varphi ' \left(\E_{k} \right) \left[\F_{k} - \F_{k+1} \right]  \geq  \varphi ' \left( \E_{k} \right)  \dfrac{\Cdecxx}{4} \left\lVert x^{k+1} - x^{k} \right\rVert ^{2}.
	\end{equation}
	The last relation combined with \eqref{conv:KL-property} imply 
	\begin{align*}
	\left\lVert x^{k+1} - x^{k} \right\rVert ^{2} & \leq \varphi ' \left( \E_{k} \right) \cdot \dist \left( \0 ,  \partial \Fr \left( x^{k} , z^{k} , y^{k} , x^{k-1} , y^{k-1} \right) \right) \left\lVert x^{k+1} - x^{k} \right\rVert ^{2}\\
 & \leq \dfrac{4}{\Cdecxx} \Delta_{k,k+1} \cdot \dist \left( \0 , \partial \Fr \left( x^{k} , z^{k} , y^{k} , x^{k-1} , y^{k-1} \right) \right)  \ \forall{k \geq k_0}.
	\end{align*}
	
By the arithmetic mean-geometric mean inequality and Corollary \ref{coro:subd-Hk-bound} we have that for every $k \geq k_{0}$ and every $\beta >0$
	\begin{align}
	\label{conv:inq}
	\left\lVert x^{k+1} - x^{k} \right\rVert & \leq \sqrt{\dfrac{4}{\Cdecxx} \Delta_{k,k+1} \cdot \dist \left( \0 ,  \partial \Fr \left( x^{k} , z^{k} , y^{k} , x^{k-1} , y^{k-1} \right)\right)} \nonumber \\
	& \leq \dfrac{\beta}{\Cdecxx} \Delta_{k,k+1} + \dfrac{1}{\beta} \dist \left( \0 ,  \partial \Fr \left( x^{k} , z^{k} , y^{k} , x^{k-1} , y^{k-1} \right)\right) \nonumber \\
	& \leq \dfrac{\beta}{\Cdecxx} \Delta_{k,k+1} + \dfrac{\Csubx}{\beta} \left( \left\lVert x^{k} - x^{k-1} \right\rVert + \left\lVert x^{k-1} - x^{k-2} \right\rVert + \left\lVert x^{k-2} - x^{k-3} \right\rVert \right) \nonumber \\
	& \quad + \dfrac{\Csuby}{\beta} \left( \left\lVert A^{*} \left( y^{k-1} - y^{k-2} \right) \right\rVert - \left\lVert A^{*} \left( y^{k} - y^{k-1} \right) \right\rVert \right) \nonumber \\
	& \quad + \dfrac{\Csubyy}{\beta} \left( \left\lVert A^{*} \left( y^{k-2} - y^{k-3} \right) \right\rVert - \left\lVert A^{*} \left( y^{k-1} - y^{k-2} \right) \right\rVert \right).
	\end{align}
 We denote for every $k \geq 3$
	\begin{align*}
	\begin{split}
	a^{k} 		& := \left\lVert x^{k} - x^{k-1} \right\rVert \geq 0 , \\
	\delta_{k} 	& := \dfrac{\beta}{\Cdecxx} \Delta_{k,k+1} + \dfrac{\Csuby}{\beta} \left( \left\lVert A^{*} \left( y^{k-1} - y^{k-2} \right) \right\rVert - \left\lVert A^{*} \left( y^{k} - y^{k-1} \right) \right\rVert \right) \\
	& \quad + \dfrac{\Csubyy}{\beta} \left( \left\lVert A^{*} \left( y^{k-2} - y^{k-3} \right) \right\rVert - \left\lVert A^{*} \left( y^{k-1} - y^{k-2} \right) \right\rVert \right).
	\end{split}
	\end{align*}
	The inequality \eqref{conv:inq} is nothing than \eqref{sum:hypo} with $c_{0} = c_{1} = c_{2} := \dfrac{\Csubx}{\beta}$.	Observe that for every $K \geq k_{0}$ we have
	\begin{equation*}
	\mysum_{k = k_{0}}^{K} \delta_{k} \leq \dfrac{\beta}{\Cdecxx} \varphi \left(\E_{k_{0}} \right) + \dfrac{\Csuby}{\beta} \left\lVert A^{*} \left( y^{k_{0}-1} - y^{k_{0}-2} \right) \right\rVert + \dfrac{\Csubyy}{\beta} \left\lVert A^{*} \left( y^{k_{0}-2} - y^{k_{0}-3} \right) \right\rVert
	\end{equation*}
	and thus, by choosing $\beta > 3 \Csubx$, we can use Lemma \ref{lem:conv-ext} to conclude that
	\begin{equation*}
	\mysum_{k \geq 0} \left\lVert x^{k+1} - x^{k} \right\rVert < + \infty .
	\end{equation*}
	The other two statements in \eqref{conv:Cauchy} follow from Lemma \ref{lem:estimate}. This means that the sequence $\left\lbrace \left( x^{k}, z^{k}, y^{k} \right) \right\rbrace _{k \geq 0}$ is Cauchy, thus it converges to an element $\left( \widehat{x} , \widehat{z} , \widehat{y} \right)$ which is, according to Lemmas \ref{lem:cluster-aLr}, a KKT point of the optimization problem \eqref{intro:problem}.
\end{proof}

\begin{rmk}
	The function $\Fr$ is a \KL \ function if, for instance, the objective function of \eqref{intro:problem} is semi-algebraic, which is the case when the functions $g$ and $h$ are semi-algebraic.
\end{rmk}

\section{Convergence rates under \Loja assumptions}
\label{sec:rates}

In this section we derive convergence rates for the sequence $\left\lbrace \left( x^{k}, z^{k}, y^{k} \right) \right\rbrace _{k \geq 0}$ generated by Algorithm \ref{algo:al.1} or Algorithm \ref{algo:al.2} as well as for the regularized augmented Lagrangian function $\Fr$ along this sequence, provided that the latter satisfies the \Loja property.

\subsection{\Loja property and a technical lemma}

We recall the following definition from \cite{Attouch-Bolte} (see, also, \cite{Lojasiewicz}).
\begin{defi}
	\label{defi:Lojasiewicz}
	Let $\Psi \colon \sR^{N} \to \sR \cup \left\lbrace + \infty \right\rbrace$ be proper and lower semicontinuous. Then $\Psi$ satisfies the \Loja property if for any critical point $\widehat{u}$ of $\Psi$, there exists $C_{L}>0$, $\theta \in \left[ 0 , 1 \right)$ and $\varepsilon > 0$ such that
	\begin{equation*}
	\left\lvert \Psi \left( u \right) - \Psi \left( \widehat{u} \right) \right\rvert ^{\theta} \leq C_{L} \cdot \dist(0,\partial \Psi(u)) \ \forall u \in \Ball \left( \widehat{u} , \varepsilon \right),
	\end{equation*}
	where $\Ball \left( \widehat{u}, \varepsilon \right)$ denotes the open ball with centre $\widehat{u}$ and radius $\varepsilon$.
\end{defi}

If Assumption \ref{ass} is fulfilled and $\left\lbrace \left( x^{k}, z^{k} , y^{k} \right) \right\rbrace _{k \geq 0}$ is the sequence generated by Algorithm \ref{algo:al.1} or Algorithm \ref{algo:al.2},  assumed to be bounded, then, as seen in Lemma \ref{lem:cluster-aHk}, the set of cluster points $\Omega = \omega \left( \left\lbrace \left( x^{k} , z^{k} , y^{k} , x^{k-1} , y^{k-1} \right) \right\rbrace _{k \geq 0} \right)$ is nonempty, compact and connected and $\Fr$ takes on $\Omega$ the value $\F_{*}$; in addition, for every $\left( \widehat{x} , \widehat{z} , \widehat{y} , \widehat{x} , \widehat{y} \right) \in \Omega$, $\left( \widehat{x} , \widehat{z} , \widehat{y} \right)$ belongs to $\crit \left( \Lr \right)$. According to \cite[Lemma 1]{Attouch-Bolte}, if $\Fr$ has the \Loja property, then there exist $C_{L} > 0$, $\theta \in \left[ 0 , 1 \right)$ and  $\varepsilon > 0$ such that for every
\begin{equation*}
\left( x , z , y , x' , y' \right) \in \left\lbrace u \in \sR^{n} \times \sR^m \times \sR^m  \! \colon \! \dist \left( u , \Omega \right) < \varepsilon \right\rbrace
\end{equation*}
it holds
\begin{equation*}
\left\lvert \Fr \left( x , z , y , x' , y' \right) - \F_{*} \right\rvert ^{\theta} \leq C_{L} \cdot \dist \left( 0 , \partial \Fr \left( x , z , y , x' , y' \right) \right).
\end{equation*}
Obviously, $\Fr$ is a \KL \ function with desingularization function $\varphi : [0,+\infty) \to [0,+\infty)$, $\varphi \left( s \right) := \dfrac{1}{1 - \theta} C_{L} s^{1 - \theta}$, which, according to Theorem \ref{thm:conv}, means that $\Omega$ contains a single element $\left( \widehat{x} , \widehat{z} , \widehat{y} , \widehat{x} , \widehat{y} \right)$, namely, the limit of $\left\lbrace \left( x^{k} , z^{k} , y^{k} , x^{k-1} , y^{k-1} \right) \right\rbrace _{k \geq 0}$ as $k \to +\infty$. In other words, if $\Fr$ has the \Loja property, then there exist $C_{L} > 0$, $\theta \in \left[ 0 , 1 \right)$ and  $\varepsilon > 0$ such that
\begin{equation}
\label{Loja:uniform}
\left\lvert \Fr \left( x , z , y , x' , y' \right) - \F_{*} \right\rvert ^{\theta} \leq C_{L} \cdot \dist \left( 0 , \partial \Fr \left( x , z , y , x' , y' \right) \right) \ \forall \left( x , z , y , x' , y' \right) \in  \Ball \left( \left( \widehat{x} , \widehat{z} , \widehat{y} , \widehat{x} , \widehat{y} \right) , \varepsilon \right).
\end{equation}
In this case, $\Fr$ is said to satisfy the \Loja property with \emph{\Loja constant} $C_{L} > 0$ and \emph{\Loja exponent} $\theta \in \left[ 0 , 1 \right)$.

The following lemma will convergence rates for a particular class of monotonically decreasing sequences converging to $0$.
\begin{lem}
	\label{lem:rates}
	Let $\left\lbrace e_{k} \right\rbrace _{k \geq 0}$ be a monotonically decreasing sequence in $\sR_{+}$ converging $0$. Assume further that there exists natural numbers $k_{0} \geq l_0 \geq 1$ such that for every $k \geq k_{0}$
	\begin{equation}
	\label{rates:rec}
	e_{k - l_0} - e_{k} \geq C_{e} e_{k}^{2 \theta} ,
	\end{equation}
	where $C_{e} > 0$ is some constant and $\theta \in \left[ 0 , 1 \right)$. Then following statements are true:
	\begin{enumerate}
		\item[(i)] 
		\label{lem:rates:i}
		if $\theta = 0$, then $\left\lbrace e_{k} \right\rbrace _{k \geq 0}$ converges in finite time;
		
	\item[(ii)] 
		\label{lem:rates:ii}
		if $\theta \in \left( 0 , 1/2 \right]$, then there exists $C_{e,0} > 0$ and $Q \in \left[ 0 , 1 \right)$ such that for every $k \geq k_{0}$
		\begin{equation*}
		0 \leq e_{k} \leq C_{e,0} Q^{k};
		\end{equation*}
		
\item[(iii)] 
		\label{lem:rates:iii}
		if $\theta \in \left( 1/2 , 1 \right)$, then there exists $C_{e,1} > 0$ such that for every $k \geq k_{0} + l_0$
		\begin{equation*}
		0 \leq e_{k} \leq C_{e,1} \left( k - l_0 + 1 \right) ^{- \frac{1}{2 \theta - 1}}  .
		\end{equation*}
	\end{enumerate}
\end{lem}
\begin{proof}
	Fix an integer $k \geq k_{0}$. Since $k_{0} \geq l_0 \geq 0$, the recurrence inequality \eqref{rates:rec} is well defined for every $k \geq k_{0}$.
	\item[(i)] The case when $\theta = 0$. We assume that $e_{k} > 0$ for every $k \geq 0$. From \eqref{rates:rec} we get
	\begin{equation*}
	e_{k-l_0} - e_{k} \geq C_{e} > 0
	\end{equation*}
	for every $k \geq k_{0}$, which actually contradicts the fact that $\left\lbrace e_{k} \right\rbrace _{k \geq 0}$  converges to $0$ as $k \to + \infty$. Consequently, there exists $k' \geq 0$ such that $e_{k'} = 0$ for every $k \geq k'$ and thus the conclusion follows.
	
	For the proof of (ii) and (iii) we can assume that $e_{k} > 0$ for every $k \geq 0$. Otherwise, as $\left\lbrace e_{k} \right\rbrace _{k \geq 0}$ is monotonically decreasing and convergent to $0$, the sequence is constant beginning with a given index, which means that both statements are true.
	
	\item[(ii)] The case when $\theta \in \left( 0 , 1/2 \right]$. We have $e_{k} \leq e_{0}$, which leads to $e_{k-l_0} - e_{k} \geq C_{e} e_{k}^{2 \theta} \geq C_{e} e_{0}^{2 \theta - 1} e_{k}$ for every $k \geq k_0$. Therefore,
	
	\begin{align*}
	e_{k} 
	\leq \left( \dfrac{1}{C_{e} e_{0}^{2 \theta - 1} + 1} \right) ^{\frac{k}{l_0} - \frac{k_{0}}{l_0} - 1} e_{0} = e_{0} \left( C_{e} e_{0}^{2 \theta - 1} + 1 \right) ^{\frac{k_{0}}{l_0} + 1} \left( \dfrac{1}{\sqrt[l_0]{C_{e} e_{0}^{2 \theta - 1} + 1}} \right) ^{k}.
	\end{align*}
	
	\item[(iii)] The case when $\theta \in \left( 1/2 , 1 \right)$. From \eqref{rates:rec} we get
	\begin{equation}
	\label{rates:rec-new}
	C_{e} \leq \left(e_{k-l_0} - e_{k} \right) e_{k}^{-2 \theta} .
	\end{equation}
	Define $\zeta \colon \left( 0 , + \infty \right) \to \sR , \zeta(s) = s^{-2 \theta}$. We have that
	\begin{equation*}
	\dfrac{d}{ds} \left( \dfrac{1}{1 - 2 \theta} s^{1 - 2 \theta} \right) = s^{-2 \theta} = \zeta \left( s \right) \textrm{ and } \zeta ' \left( s \right) = -2 \theta s^{-2 \theta - 1} < 0  \ \forall s \in (0,+\infty).
	\end{equation*}
	Consequently, $\zeta \left(e_{k-l_0} \right) \leq \zeta \left( s \right)$ for all $s \in \left[ e_{k} , e_{k-l_0} \right]$.		
	
	\begin{itemize}[leftmargin=\parindent]
		\item 
		\emph{Assume that $\zeta \left( e_{k} \right) \leq 2 \zeta \left( e_{k-l_0} \right)$}. Then \eqref{rates:rec-new} gives
		\begin{align*}
		C_{e} \leq  2 \zeta \left(e_{k-l_0} \right) \int_{e_{k}}^{e_{k-l_0}} 1 ds \leq 2 \int_{e_{k}}^{e_{k-l_0}} \zeta \left( s \right) ds = \dfrac{2}{2 \theta - 1} \left(e_{k}^{1 - 2 \theta} - e_{k-l_0}^{1 - 2 \theta} \right)
		\end{align*}
		or, equivalently,
		\begin{equation}
		\label{rates:const:lowerbound:1}
		e_{k}^{1 - 2 \theta} - e_{k-l_0}^{1 - 2 \theta} \geq C_{1}', \ \textrm{ where } C_{1}' := \dfrac{\left( 2 \theta - 1 \right) C_{e}}{2} > 0 .
		\end{equation}
		
		\item
		\emph{Assume that $\zeta \left( e_{k} \right) > 2 \zeta \left( e_{k-l_0} \right)$}. For $\nu := 2 ^{-\frac{1}{2 \theta}} \in \left( 0 , 1 \right)$ this is equivalent to $\left( \nu^{1 - 2 \theta} - 1 \right) e_{k-l_0}^{1 - 2 \theta} \leq e_{k}^{1 - 2 \theta} - e_{k-l_0}^{1 - 2 \theta}$, thus,
\begin{equation}
		\label{rates:const:lowerbound:2}
		e_{k}^{1 - 2 \theta} - e_{k-l_0}^{1 - 2 \theta} \geq \left( \nu^{1 - 2 \theta} - 1 \right) e_{k-l_0}^{1 - 2 \theta} \geq C_{2}', \ \textrm{ where } C_{2}' := \left( \nu^{1 - 2 \theta} - 1 \right) e_{0}^{2 \theta - 1} > 0 .
		\end{equation}
	\end{itemize}
	In both situations we get for every $i \geq k_{0}$
	\begin{equation}
	\label{rates:const:lowerbound}
	e_{i}^{1 - 2 \theta} - e_{i-l_0}^{1 - 2 \theta} \geq C' := \min \left\lbrace C_{1}' , C_{2}' \right\rbrace > 0 ,
	\end{equation}
	where $C_{1}'$ and $C_{2}'$ are defined as in \eqref{rates:const:lowerbound:1} and \eqref{rates:const:lowerbound:2}, respectively. For every $k \geq k_{0} + 2l_0$, by summing up the inequalities \eqref{rates:const:lowerbound} for $i = k_{0} + l_0 , \cdots , k$, we get
	\begin{equation*}
	\mysum_{j=0}^{l_0-1} \left( e_{k-j}^{1 - 2 \theta} - e_{k_{0} + j}^{1 - 2 \theta} \right) \geq \left( k - k_{0} - l_0 + 1 \right) C'> 0 .
	\end{equation*}
Since
	\begin{equation*}
	l_0 \left(e_{k}^{1 - 2 \theta} - e_{k_{0}}^{1 - 2 \theta} \right) \geq \mysum_{j=0}^{l_0-1} \left(e_{k-j}^{1 - 2 \theta} - e_{k_{0} + j}^{1 - 2 \theta} \right) \geq \left( k - k_{0} - l_0+ 1 \right)C',
	\end{equation*}
we have
	\begin{equation}
	\label{rates:pre}
	e_{k}^{1 - 2 \theta} \geq e_{k_{0}}^{1 - 2 \theta} + \dfrac{k - k_{0} - l_0 + 1}{l_0} C'. 
	\end{equation}
We obtain from \eqref{rates:const:lowerbound} that
	\begin{equation}
	\label{rates:indu}
	e_{k_{0}}^{1 - 2 \theta} \geq \left\lfloor \dfrac{k_{0}+l_0}{l_0} \right\rfloor C' \geq \left( \dfrac{k_{0}+l_0}{l_0} - 1 \right) C' = \dfrac{k_{0}}{l_0} C',
	\end{equation}
where $\left\lfloor p \right\rfloor$ denotes the greatest integer that is less than or equal to the real number $p$.
	By plugging \eqref{rates:indu} into \eqref{rates:pre} we obtain
	\begin{equation*}
	e_{k}^{1 - 2 \theta} \geq \dfrac{k - l_0 + 1}{l_0} C',
	\end{equation*}
	which implies
	\begin{equation}
	\label{rates:near}
	e_{k} \leq \left( \dfrac{C'}{l_0} \right) ^{- \frac{1}{2 \theta - 1}} \left( k - l_0+ 1 \right) ^{- \frac{1}{2 \theta - 1}} .
	\end{equation}
	This concludes the proof.
\end{proof}

\begin{rmk}
	The inequality in Lemma \ref{lem:rates} (iii) can be writen for $k$ large enough in terms of $k$ instead of $k-l_0+1$. If, for instance, $k \geq 2 \left( l_0 + 1 \right)$, then  $k-l_0+1 \geq \dfrac{1}{2} k$ and thus from \eqref{rates:near} we get
	\begin{equation*}
	e_{k} \leq \left( \dfrac{C'}{l_0} \right) ^{- \frac{1}{2 \theta - 1}} \left( k - l_0+ 1 \right) ^{- \frac{1}{2 \theta - 1}} \leq \left( \dfrac{C'}{2 l_0} \right) ^{- \frac{1}{2 \theta - 1}} k^{- \frac{1}{2 \theta - 1}} .
	\end{equation*}
\end{rmk}

\subsection{Convergence rates} \label{subsec:non-relaxed}
In this subsection we will address convergence rates for Algorithm \ref{algo:al.1} and Algorithm \ref{algo:al.2} in the context of an assumption which is slightly more restricitve than Assumption \ref{ass}.

\begin{assume}
	\label{assume-0:decrease}	
	We work in the hypotheses of Assumption \ref{ass} except for \eqref{assume:vm:positive} which is replaced by
	\begin{equation}
	\label{assume-0:vm:positive}
	2 \M_{1}^{k} + r A^{*} A \succcurlyeq \left( L + \dfrac{C_{\M}'}{r} \right) \Id  \quad \forall k \geq 0,
	\end{equation}
\end{assume}
where
\begin{equation*}
C_{\M}' := \begin{cases}
\left( 10 \mu_{1} ^{2} + 8 \left( L + \mu_{1} \right) ^{2} \right) \Tdecx, & \mbox{for Algorithm} \ \ref{algo:al.1}, \\
\left( 8 \mu_{1} ^{2} + 10 \left( L + \mu_{1} \right) ^{2} \right) \Tdecx, & \mbox{for Algorithm} \ \ref{algo:al.2}.
\end{cases} 
\end{equation*}
The condition \eqref{assume-0:vm:positive} is nothing else than \eqref{assume:vm:positive} after replacing $C_{\M}$ by the bigger constant $C_{\M}'$. 

The examples in Remark \ref{rmk:para-choice} can be all adapted to the new setting and one can provide different settings which guarantee Assumption \ref{assume-0:decrease}. The scenarios which ensure Assumption \ref{assume-0:decrease} evidently satisfy Assumption \ref{ass}, too, therefore the results investigated in Section \ref{sec:main} remain valid in this setting. As follows we will provide improvements of the statements used in the convergence analysis which follow thanks to Assumption \ref{assume-0:decrease}.
\begin{lem}
	\label{lem:decrease-0}
Suppose that Assumption \ref{assume-0:decrease} holds true and let $\left\lbrace \left( x^{k} , z^{k} , y^{k} \right) \right\rbrace _{k \geq 0}$ be a sequence generated by Algorithm \ref{algo:al.1} or Algorithm \ref{algo:al.2}. Then for every $k \geq 1$ it holds
	\begin{align}
	\label{dec-0:inq}
	& \ \Lr \left( x^{k+1} , z^{k+1} , y^{k+1} \right) + 2 \Tdecy \left\lVert A^{*} \left( y^{k+1} - y^{k} \right) \right\rVert ^{2} + \dfrac{1}{2} \left\lVert x^{k+1} - x^{k} \right\rVert _{\M_{3}^{k}}^{2} + \dfrac{1}{2} \left\lVert z^{k+1} - z^{k} \right\rVert _{\M_{2}^{k}}^{2} \nonumber \\
	& \ + \dfrac{1}{\rho r} \left\lVert y^{k+1} - y^{k} \right\rVert ^{2} \nonumber \\
	\leq & \ \Lr \left( x^{k} , z^{k} , y^{k} \right) + 2 \Tdecy \left\lVert A^{*} \left( y^{k} - y^{k-1} \right) \right\rVert ^{2} +\Cdecxx \left\lVert x^{k} - x^{k-1} \right\rVert^{2}.
	\end{align}
\end{lem}
\begin{proof}
Let $k \geq 1$ be fixed. By the same arguments as in Lemma \ref{lem:decrease}, we have that (see \eqref{dec:pre})
\begin{align}
\label{dec-0:pre}
& \Lr \left( x^{k+1} , z^{k+1} , y^{k+1} \right) + \dfrac{1}{2} \left\lVert x^{k+1} - x^{k} \right\rVert _{2 \M_{1}^{k} + r A^{*} A}^{2} - \dfrac{L}{2} \left\lVert x^{k+1} - x^{k} \right\rVert ^{2} + \dfrac{1}{2} \left\lVert z^{k+1} - z^{k} \right\rVert _{\M_{2}^{k}}^{2} \nonumber \\
\leq & \ \Lr \left( x^{k} , z^{k} , y^{k} \right) + \dfrac{1}{\rho r} \left\lVert y^{k+1} - y^{k} \right\rVert ^{2}.
\end{align}
From \eqref{dec:al.1:p1}, \eqref{dec:al.1:p2}, \eqref{dec:al.2:p1} and \eqref{dec:al.2:p2} it follows that
\begin{align}
\label{dec-0:mult-bound}
\dfrac{1}{\rho r} \left\lVert y^{k+1} - y^{k} \right\rVert ^{2} \leq & \dfrac{\Cdecx - L}{2} \left\lVert x^{k+1} - x^{k} \right\rVert ^{2} + \dfrac{\Cdecxx}{2} \left\lVert x^{k} - x^{k-1} \right\rVert ^{2} + \nonumber \\ 
& \Tdecy \left\lVert A^{*} \left( y^{k} - y^{k-1} \right) \right\rVert ^{2} - \Tdecy \left\lVert A^{*} \left( y^{k+1} - y^{k} \right) \right\rVert ^{2} .
\end{align}
By multiplying  \eqref{dec-0:mult-bound} by $2$ and by adding the resulting inequality to \eqref{dec-0:pre} we obtain \eqref{dec-0:inq}.
\end{proof}

We replace  $\Tdecy$ with $2 \Tdecy$ in the definition of the regularized augmented Lagrangian $\Fr$, 
thus, the sequence $\left\lbrace \F_{k} \right\rbrace _{k \geq 1}$ in \eqref{defi:Hk} becomes
\begin{align*}
\F_{k} :=  \Lr \left( x^{k} , z^{k} , y^{k} \right) + 2 \Tdecy \left\lVert A^{*} \left( y^{k} - y^{k-1} \right) \right\rVert ^{2} + \Cdecxx\left\lVert x^{k} - x^{k-1} \right\rVert^{2} \ \forall k \geq 1.
\end{align*}
In this new context the inequality \eqref{dec-0:inq} reads for every $k \geq 1$
\begin{equation}
\label{limit-0:dec:Hk}
\F_{k+1} + \dfrac{\Cdecxx}{4} \left\lVert x^{k+1} - x^{k} \right\rVert ^{2} + \dfrac{1}{2} \left\lVert z^{k+1} - z^{k} \right\rVert _{\M_{2}^{k}}^{2} + \dfrac{1}{\rho r} \left\lVert y^{k+1} - y^{k} \right\rVert ^{2} \leq \F_{k}
\end{equation}
and provides an inequality which is tighter than  relation \eqref{limit:dec:Hk} in Theorem \ref{thm:limit}. Furthermore, for a subgradient $D^{k+1}$ of $\F_r$ at $(x^{k+1}, z^{k+1}, y^{k+1}, x^{k}, z^{k})$ defined as in \eqref{aHk:subd} (again by replacing $\Tdecy$ by $2 \Tdecy$)  we obtain for every $k \geq 2$ the following estimate, which is simpler than \eqref{subd:inq} in Corollary \ref{coro:subd-Hk-bound}
\begin{equation*}
\label{aHk-0:inq}
\opnorm{D^{k+1}} \leq \CaHkbx \left\lVert x^{k+1} - x^{k} \right\rVert + \CaHkby \left\lVert y^{k+1} - y^{k} \right\rVert + \CaHkbyy \left\lVert y^{k} - y^{k-1} \right\rVert,
\end{equation*}
where 
\begin{equation*}
\label{aHk-0:const}
\CaHkbx := \CaHkx + \CaHkz \left\lVert A \right\rVert , \quad \CaHkby := \CaHky + \dfrac{\CaHkz}{\rho r} , \quad \CaHkbyy := \dfrac{\CaHkz}{\rho r} .
\end{equation*}

This improvement provides, instead of inequality \eqref{conv:concave}  in the proof of Theorem \ref{thm:conv}, the following very useful estimate
\begin{align*}
\Delta_{k,k+1} & = \varphi \left( \E_{k} \right) - \varphi \left( \E_{k+1} \right) \geq \varphi ' \left( \E_{k} \right) \min \left\lbrace \dfrac{\Cdecxx}{4} , \dfrac{1}{\rho r} \right\rbrace \left( \left\lVert x^{k+1} - x^{k} \right\rVert ^{2} + \left\lVert y^{k+1} - y^{k} \right\rVert ^{2} \right) \\
& \geq \Clower \varphi ' \left( \E_{k} \right) \left( \left\lVert x^{k+1} - x^{k} \right\rVert + \left\lVert y^{k+1} - y^{k} \right\rVert \right) ^{2},
\end{align*}
where
\begin{equation*}
\label{const:min}
\Clower := \dfrac{1}{2} \min \left\lbrace \dfrac{\Cdecxx}{4} , \dfrac{1}{\rho r} \right\rbrace .
\end{equation*}
The last relation together with \eqref{conv:KL-property} imply that for every $k \geq k_0$
\begin{align*}
\left( \left\lVert x^{k+1} - x^{k} \right\rVert + \left\lVert y^{k+1} - y^{k} \right\rVert \right) ^{2} 
\leq \dfrac{\Delta_{k,k+1}}{\Clower}  \cdot \dist \left( \0 ,  \partial \Fr \left( x^{k} , z^{k} , y^{k} , x^{k-1} , y^{k-1} \right) \right)
\end{align*}
and from here, for arbitrary $\beta >0$,
\begin{align}
\label{conv-0:inq}
& \left\lVert x^{k+1} - x^{k} \right\rVert + \left\lVert y^{k+1} - y^{k} \right\rVert \nonumber \\
\leq & \ \dfrac{\beta \Delta_{k,k+1}}{4 \Clower} + \dfrac{\max \left\lbrace \CaHkbx , \CaHkby \right\rbrace}{\beta} \left( \left\lVert x^{k} - x^{k-1} \right\rVert +  \left\lVert y^{k} - y^{k-1} \right\rVert + \left\lVert y^{k-1} - y^{k-2} \right\rVert \right).
\end{align}
By denoting
\begin{equation*}
\label{conv-0:defi}
a^{k}  := \left( \left\lVert x^{k} - x^{k-1} \right\rVert , \left\lVert y^{k} - y^{k-1} \right\rVert \right) \in \sR_{+}^{2} \quad \textrm{ and } \quad \delta_{k} := \dfrac{\beta \Delta_{k,k+1}}{4 \Clower},
\end{equation*}
inequality \eqref{conv-0:inq} can be rewritten for every $k \geq k_{0}$ as
\begin{equation}\label{conv-0:lem}
\left\langle \mathbbm{1} , a^{k+1} \right\rangle \leq \left\langle c_{0} , a^{k} \right\rangle + \left\langle c_{1} , a^{k-1} \right\rangle + \delta_{k} ,
\end{equation}
where
\begin{equation*}
c_{0} := \dfrac{\max \left\lbrace \CaHkbx , \CaHkby \right\rbrace}{\beta} \left( 1 , 1 \right)  \quad \textrm{ and } \quad c_{1} := \dfrac{\max \left\lbrace \CaHkbx , \CaHkby \right\rbrace}{\beta} \left( 0 , 1 \right) .
\end{equation*}
Choosing $\beta > 2 \max \left\lbrace \CaHkbx , \CaHkby \right\rbrace$, Lemma \ref{lem:conv-ext} and Lemma \ref{lem:estimate} imply that $\left\lbrace \left( x^{k} , z^{k} , y^{k} \right) \right\rbrace _{k \geq 0}$ has finite length (see \eqref{conv:Cauchy}).

Next we prove a recurrence inequality for the  sequence $\left\lbrace \E_{k} \right\rbrace _{k \geq 0}$.
\begin{lem}
	\label{lem:err.est-0}
Suppose that Assumption \ref{assume-0:decrease} holds true and let $\left\lbrace \left( x^{k} , z^{k} , y^{k} \right) \right\rbrace _{k \geq 0}$ be a sequence generated by Algorithm \ref{algo:al.1} or Algorithm \ref{algo:al.2}, which is assumed to be bounded. If $\Fr$ satisfies the \Loja property with \Loja constant $C_{L} > 0$ and \Loja exponent $\theta \in \left[ 0 , 1 \right)$, then there exists $k_{0} \geq 1$ such that the following estimate holds for every $k \geq k_{0}$ 
	\begin{equation}
	\label{err.est-0:rec}
	\E_{k-1} - \E_{k+1} \geq \Crecb \E_{k+1}^{2 \theta} , \quad \textrm{ where } \quad \Crecb := \dfrac{\min \left\lbrace \dfrac{\Cdecxx}{4} , \dfrac{1}{\rho r} \right\rbrace}{3 C_{L}^{2} \max \left\lbrace \CaHkbx , \CaHkby \right\rbrace ^{2}}.
	\end{equation}
\end{lem}
\begin{proof}
For every $k \geq 2$ we obtain from \eqref{limit-0:dec:Hk}
	\begin{align*}
	\E_{k-1} - \E_{k+1} & = \F_{k-1} - \F_{k} + \F_{k} - \F_{k+1} \\
	& \geq \min \left\lbrace \dfrac{\Cdecxx}{4} , \dfrac{1}{\rho r} \right\rbrace \left( \left\lVert x^{k+1} - x^{k} \right\rVert ^{2} + \left\lVert y^{k+1} - y^{k} \right\rVert ^{2} + \left\lVert y^{k} - y^{k-1} \right\rVert ^{2} \right) \\
	& \geq \dfrac{1}{3} \min \left\lbrace \dfrac{\Cdecxx}{4} , \dfrac{1}{\rho r} \right\rbrace \left( \left\lVert x^{k+1} - x^{k} \right\rVert + \left\lVert y^{k+1} - y^{k} \right\rVert + \left\lVert y^{k} - y^{k-1} \right\rVert \right) ^{2} \\
	& \geq \Crecb C_{L}^{2} \opnorm{D^{k+1}}^{2}. 
	\end{align*}
	Let $\varepsilon > 0$ be such that \eqref{Loja:uniform} is fulfilled and choose $k_{0} \geq 1$ such that  $\left( x^{k+1}, z^{k+1}, y^{k+1} \right)$ belongs to $\Ball \left( \left( \widehat{x} , \widehat{z} , \widehat{y} \right) , \varepsilon \right)$ for every $k \geq k_{0}$. Then \eqref{Loja:uniform} implies \eqref{err.est-0:rec} for every $k \geq k_{0}$.
\end{proof}

The following convergence rates follow by combining  Lemma \ref{lem:rates} with Lemma \ref{lem:err.est-0}.
\begin{thm}
	\label{thm:rates-obj-0}
Suppose that Assumption \ref{assume-0:decrease} holds true and let $\left\lbrace \left( x^{k} , z^{k} , y^{k} \right) \right\rbrace _{k \geq 0}$ be a sequence generated by Algorithm \ref{algo:al.1} or Algorithm \ref{algo:al.2}, which is assumed to be bounded. If $\Fr$ satisfies the \Loja property with \Loja constant $C_{L} > 0$ and \Loja exponent $\theta \in \left[ 0 , 1 \right)$, then the following statements are true:
	\begin{enumerate}
		\item[(i)] if $\theta = 0$, then $\left\lbrace \F_{k} \right\rbrace _{k \geq 1}$ converges in finite time;
		
		\item[(ii)] if $\theta \in \left( 0 , 1/2 \right]$, then there exist $k_0 \geq 1$, $\widehat{C}_{0} > 0$ and $Q \in \left[ 0 , 1 \right)$ such that for every $k \geq k_{0}$
		\begin{equation*}
		0 \leq \F_{k} - \F_{*}  \leq \widehat{C}_{0} Q^{k};
		\end{equation*}
		
		\item[(iii)] if $\theta \in \left( 1/2 , 1 \right)$, then there exist $k_0 \geq 3$ and $\widehat{C}_{1} > 0$ such that for every $k \geq k_{0}$
		\begin{equation*}
		0 \leq \F_{k} - \F_{*}  \leq \widehat{C}_{1} \left( k - 1 \right) ^ {- \frac{1}{2 \theta - 1}}  .
		\end{equation*}
	\end{enumerate}
\end{thm}

The next lemma will play an importat role when transferring the convergence rates for $\left\lbrace \F_{k} \right\rbrace _{k \geq 0}$ to the sequence of iterates $\left\lbrace \left( x^{k} , z^{k} , y^{k} \right) \right\rbrace _{k \geq 0}$ (see \cite{Frankel-Garrigos-Peypouquet} for a similar statement).
\begin{lem}
	\label{lem:err.ite-0}
Suppose that Assumption \ref{assume-0:decrease} holds true and let $\left\lbrace \left( x^{k} , z^{k} , y^{k} \right) \right\rbrace _{k \geq 0}$ be a sequence generated by Algorithm \ref{algo:al.1} or Algorithm \ref{algo:al.2}, which is assumed to be bounded. Suppose further that $\Fr$ satisfies the \Loja property with \Loja constant $C_{L} > 0$, \Loja exponent $\theta \in \left[ 0 , 1 \right)$ and desingularization function $\varphi : [0,+\infty) \to [0,+\infty), \varphi \left( s \right) := \dfrac{1}{1 - \theta} C_{L} s^{1 - \theta}$. Let $\left( \widehat{x} , \widehat{z} , \widehat{y} \right)$ be the KKT point of the optimization problem \eqref{intro:problem} to which $\left\lbrace \left( x^{k}, z^{k}, y^{k} \right) \right\rbrace _{k \geq 0}$ converges as $k \to +\infty$. Then there exists $k_{0} \geq 2$ such that the following estimates hold for every $k \geq k_{0}$
	\begin{subequations}
		\label{err.ite-0:inq}
		\begin{align}
		\begin{split}
		\label{err.ite-0:inq:x}
		\left\lVert x^{k} - \widehat{x} \right\rVert \leq \Citebx \max \left\lbrace \sqrt{\E_{k}} , \varphi \left( \E_{k} \right) \right\rbrace , & \quad \textrm{ where } \quad \Citebx := \dfrac{7}{\sqrt{\Clower}} + \dfrac{1}{\Clower} ,
		\end{split}
		\\
		\begin{split}
		\label{err.ite-0:inq:y}
		\left\lVert y^{k} - \widehat{y} \right\rVert \leq \Citeby \max \left\lbrace \sqrt{\E_{k}} , \varphi \left( \E_{k} \right) \right\rbrace , & \quad \textrm{ where } \quad \Citeby := \dfrac{7}{2 \sqrt{\Clower}} + \dfrac{1}{2 \Clower} ,
		\end{split}
		\\
		\begin{split}
		\label{err.ite-0:inq:z}
		\left\lVert z^{k} - \widehat{z} \right\rVert \leq \Citebz \max \left\lbrace \sqrt{\E_{k-1}} , \varphi \left( \E_{k-1} \right) \right\rbrace , & \quad \textrm{ where } \quad \Citebz := \Citebx \left\lVert A \right\rVert + \dfrac{2 \Citeby}{\rho r}.
		\end{split}
		\end{align}
	\end{subequations}
\end{lem}
\begin{proof}
We assume that $\E_{k} > 0$ for every $k \geq 0$. Otherwise, beginning with a given index, the sequence $\left\lbrace \left( x^{k} , z^{k} , y^{k} \right) \right\rbrace _{k \geq 0}$ becomes identical to $\left( \widehat{x} , \widehat{z} , \widehat{y} \right)$  and the conclusion follows as in the proof of Theorem \ref{thm:conv}. Let $\varepsilon > 0$ be such that \eqref{Loja:uniform} is fulfilled and $k_{0} \geq 2$ such that  $\left( x^{k+1}, z^{k+1}, y^{k+1} \right)$ belongs to $\Ball \left( \left( \widehat{x} , \widehat{z} , \widehat{y} \right) , \varepsilon \right)$ for every $k \geq k_{0}$. We fix $k \geq k_{0}$. One can easily notice that
	\begin{subequations}
		\label{err.ite-0:crit.bounded}
		\begin{equation}
		\label{err.ite-0:crit.bounded:x}
		\left\lVert x^{k} - \widehat{x} \right\rVert \leq \left\lVert x^{k+1} - x^{k} \right\rVert + \left\lVert x^{k+1} - \widehat{x} \right\rVert \leq \cdots \leq \mysum_{l \geq k} \left\lVert x^{l+1} - x^{l} \right\rVert
		\end{equation}
and, similarly,
		\begin{equation}
		\label{err.ite-0:crit.bounded:zy}
		\left\lVert z^{k} - \widehat{z} \right\rVert \leq \mysum_{l \geq k} \left\lVert z^{l+1} - z^{l} \right\rVert \ \mbox{and} \ \left\lVert y^{k} - \widehat{y} \right\rVert \leq \mysum_{l \geq k} \left\lVert y^{l+1} - y^{l} \right\rVert .
		\end{equation}
	\end{subequations}
Recall that  the inequality \eqref{conv-0:inq} can be rewritten as \eqref{conv-0:lem}. For $\beta := 3 \max \left\lbrace \CaHkbx , \CaHkby \right\rbrace > 2 \max \left\lbrace \CaHkbx , \CaHkby \right\rbrace$, thanks to Lemma \ref{lem:conv-ext} and the estimate \eqref{limit-0:dec:Hk}, we have that
	\begin{align*}
	\begin{split}
	& \mysum_{l \geq k} \left\lVert x^{l+1} - x^{l} \right\rVert = \mysum_{l \geq k} a_{1}^{l+1} = \mysum_{l \geq k+1} a_{1}^{l} \\
	\leq & \ \left\lVert x^{k+1} - x^{k} \right\rVert + 2 \left\lVert x^{k+2} - x^{k+1} \right\rVert + 3 \left\lVert x^{k+3} - x^{k+2} \right\rVert + 2 \left\lVert y^{k+1} - y^{k} \right\rVert \\
	& + 2 \left\lVert y^{k+2} - y^{k+1} \right\rVert + 3 \left\lVert y^{k+3} - y^{k+2} \right\rVert + \dfrac{\varphi \left( \E_{k} \right)}{\Clower} \\
	\leq & \ \dfrac{2}{\sqrt{\Clower}} \sqrt{\F_{k} - \F_{k+1}} + \dfrac{2}{\sqrt{\Clower}} \sqrt{\F_{k+1} - \F_{k+2}} + \dfrac{3}{\sqrt{\Clower}} \sqrt{\F_{k+2} - \F_{k+3}} + \dfrac{\varphi \left( \E_{k} \right)}{\Clower} \\
	\leq & \ \dfrac{2}{\sqrt{\Clower}} \sqrt{\E_{k}} + \dfrac{2}{\sqrt{\Clower}} \sqrt{\E_{k+1}} + \dfrac{3}{\sqrt{\Clower}} \sqrt{\E_{k+2}} + \dfrac{\varphi \left( \E_{k} \right)}{\Clower}
	\end{split}	
	\end{align*}
	and, similarly,
	\begin{equation*}
	\mysum_{l \geq k} \left\lVert y^{l+1} - y^{l} \right\rVert \leq \dfrac{1}{\sqrt{\Clower}} \sqrt{\E_{k}} + \dfrac{1}{\sqrt{\Clower}} \sqrt{\E_{k+1}} + \dfrac{3}{2 \sqrt{\Clower}} \sqrt{\E_{k+2}} + \dfrac{\varphi \left( \E_{k} \right)}{2 \Clower} .
	\end{equation*}
	By taking into account the relations above, \eqref{err.ite-0:crit.bounded:x}-\eqref{err.ite-0:crit.bounded:zy} as well as
	\begin{equation*}
	\sqrt{\E_{k+2}} \leq \sqrt{\E_{k+1}} \leq \sqrt{\E_{k}} \qquad \textrm{ and } \qquad \varphi \left( \E_{k+1} \right) \leq \varphi \left( \E_{k} \right) \ \forall k \geq 1,
	\end{equation*}
the estimates \eqref{err.ite-0:inq:x} and \eqref{err.ite-0:inq:y} follow. Statement \eqref{err.ite-0:inq:z} follows from Lemma \ref{lem:estimate} and by considering \eqref{err.ite-0:crit.bounded:zy}.
\end{proof}

We provide now convergence rates for the sequence $\left\lbrace \left( x^{k} , z^{k} , y^{k} \right) \right\rbrace _{k \geq 0}$.
\begin{thm}
	\label{thm:rates-ite-0}
Suppose that Assumption \ref{assume-0:decrease} holds true and let $\left\lbrace \left( x^{k} , z^{k} , y^{k} \right) \right\rbrace _{k \geq 0}$ be a sequence generated by Algorithm \ref{algo:al.1} or Algorithm \ref{algo:al.2}, which is assumed to be bounded. Suppose further that $\Fr$ satisfies the \Loja property with \Loja constant $C_{L} > 0$ and \Loja exponent $\theta \in \left[ 0 , 1 \right)$. Let $\left( \widehat{x} , \widehat{z} , \widehat{y} \right)$ be the KKT point of the optimization problem \eqref{intro:problem} to which $\left\lbrace \left( x^{k}, z^{k}, y^{k} \right) \right\rbrace _{k \geq 0}$ converges as $k \to +\infty$. Then the following statements are true:
	\begin{enumerate}
		\item[(i)]
		\label{thm:rates-ite-0:i}
		if $\theta = 0$, then the algorithms converge in finite time;
		
		\item[(ii)]
		\label{thm:rates-ite-0:ii}
		if $\theta \in \left( 0 , 1/2 \right]$, then there exist $k_0 \geq 1$, $\widehat{C}_{0,1} , \widehat{C}_{0,2} , \widehat{C}_{0,3} > 0$ and $\widehat Q \in \left[0 , 1 \right)$ such that for every $k \geq k_{0}$
		\begin{equation*}
		\left\lVert x^{k} - \widehat{x} \right\rVert \leq \widehat{C}_{0,1} \widehat Q^{k} , \quad
		\left\lVert y^{k} - \widehat{y} \right\rVert \leq \widehat{C}_{0,2} \widehat Q^{k} , \quad
		\left\lVert z^{k} - \widehat{z} \right\rVert \leq \widehat{C}_{0,3} \widehat Q^{k} ;
		\end{equation*}
		
		\item[(iii)]
		\label{thm:rates-ite-0:iii}
		if $\theta \in \left( 1/2 , 1 \right)$, then there exist $k_0 \geq 3$ and $\widehat{C}_{1,1} , \widehat{C}_{1,2} , \widehat{C}_{1,3} > 0$ such that for every $k \geq k_{0}$
		\begin{equation*}
		\left\lVert x^{k} - \widehat{x} \right\rVert \leq \widehat{C}_{1,1} \left( k-1 \right)^{- \frac{1-\theta}{2 \theta - 1}}, \quad
		\left\lVert y^{k} - \widehat{y} \right\rVert \leq \widehat{C}_{1,2} \left( k-1 \right)^{- \frac{1-\theta}{2 \theta - 1}}, \quad
		\left\lVert z^{k} - \widehat{z} \right\rVert \leq \widehat{C}_{1,3} \left( k-2 \right)^{- \frac{1-\theta}{2 \theta - 1}} .
		\end{equation*}
	\end{enumerate}
\end{thm}
\begin{proof}
	By denoting $\varphi : [0,+\infty) \to [0,+\infty), \varphi \left( s \right) := \dfrac{1}{1 - \theta} C_{L} s^{1 - \theta}$, the desingularization function, there exist $k_{0}' \geq 2$ such that for every $k \geq k_{0}'$ the inequalities \eqref{err.ite-0:inq:x}-\eqref{err.ite-0:inq:z} in Lemma \ref{lem:err.ite-0} and 
$\E_{k} \leq \left ( \dfrac{1}{1 - \theta} C_{L} \right)^{\frac{2}{2 \theta - 1}}$ hold.
	
	\item[(i)] If $\theta =0$, then $\{\F_k\}_{k \geq 1}$ converges in finite time. According to \eqref{limit-0:dec:Hk}, the sequences $\left\lbrace \left( x^{k}\right) \right\rbrace _{k \geq 0}$ and $\left\lbrace \left( y^{k} \right) \right\rbrace _{k \geq 0}$ converge also in finite time. Further, by Lemma \ref{lem:estimate}, it follows that $\left\lbrace \left(z^{k}\right) \right\rbrace _{k \geq 0}$ converges in finite time, too. In other words, starting from a given index, the sequence $\left\lbrace \left( x^{k} , z^{k} , y^{k} \right) \right\rbrace _{k \geq 0}$ becomes identical to $\left( \widehat{x} , \widehat{z} , \widehat{y} \right)$  and the conclusion follows.		
	
	\item[(ii)] If $\theta \in \left( 0 , 1/2 \right]$, then $\dfrac{1}{1 - \theta} C_{L} \E_{k}^{1 - \theta} \leq \sqrt{\E_{k}}$, for every $k \geq k_0'$, which implies that $\max \left\lbrace \sqrt{\E_{k}} , \varphi \left( \E_{k} \right) \right\rbrace = \sqrt{\E_{k}}$. By Theorem \ref{thm:rates-obj-0}, there exist $k_0'' \geq 1$, $\widehat{C}_{0} > 0$ and $Q \in [0,1)$ such that for $\widehat{Q} := {Q}^\frac{1}{2}$ and every $k \geq k_0''$ it holds
	\begin{equation*}
	\sqrt{\E_{k}} \leq \sqrt{\widehat{C}_{0}} Q^{\frac{k}{2}} = \sqrt{\widehat{C}_{0}} \widehat{Q}^{k}.
	\end{equation*}
	The conclusion follows from Lemma \ref{lem:err.ite-0} for $k_0:=\max\{k_0', k_0''\}$, by noticing that
	\begin{equation*}
	\sqrt{\E_{k-1}} \leq \sqrt{\widehat{C}_{0}} Q^{\frac{k-1}{2}} = \sqrt{\dfrac{\widehat{C}_{0}}{Q}} \widehat{Q}^{k} \quad \textrm{ and } \quad \sqrt{\E_{k-2}} \leq \sqrt{\widehat{C}_{0}} Q^{\frac{k-2}{2}} = \dfrac{\sqrt{\widehat{C}_{0}}}{Q} \widehat{Q}^{k}  \forall k \geq k_0.
	\end{equation*}
	
	\item[(iii)]
	If $\theta \in \left( 1/2 , 1 \right)$, then $\E_{k}^{\frac{1}{2}} \leq \dfrac{1}{1 - \theta} C_{L} \E_{k}^{1 - \theta}$, for every $k \geq k_{0}'$, which implies that $\max \left\lbrace \sqrt{\E_{k}} , \varphi \left( \E_{k} \right) \right\rbrace = \varphi(\E_{k}) = \dfrac{1}{1 - \theta} C_{L} \E_{k}^{1 - \theta}$. By Theorem \ref{thm:rates-obj-0}, there exist $k_0'' \geq 3$ and $\widehat{C}_{1} > 0$ such that for all $k \geq k_{0}''$
	\begin{equation*}
	\dfrac{1}{1 - \theta} C_{L} \E_{k}^{1 - \theta} \leq\dfrac{1}{1 - \theta} C_{L} \widehat{C}_{1}^{1-\theta} \left( k - 2 \right) ^{- \frac{1 - \theta}{2 \theta - 1}} .
	\end{equation*}
	The conclusion follows again for $k_0:=\max\{k_0', k_0''\}$ from Lemma \ref{lem:err.ite-0}.
\end{proof}

\begin{rmk}
For $\rho = 1$ the same convergence rates can be obtained under the original Assumption \ref{ass}. Indeed, when $\rho = 1$ we have that $\Tdecy = 0$ and, as a consequence, the sequence $\left\lbrace \F_{k} \right\rbrace _{k \geq 1}$ defined in \eqref{defi:Hk} becomes
\begin{equation*}
\F_{k} =  \Lr \left( x^{k} , z^{k} , y^{k} \right) + C_1\left\lVert x^{k} - x^{k-1} \right\rVert^{2} \  \forall k \geq 1 .
\end{equation*}
In addition, the inequality \eqref{est:y-bounded-by-x} simplifies to
\begin{equation*}
\left\lVert y^{k+1} - y^{k} \right\rVert \leq \Cestx \left\lVert x^{k+1} - x^{k} \right\rVert + \Cestxx \left\lVert x^{k} - x^{k-1} \right\rVert \  \forall k \geq 1,	
\end{equation*}
as $\Cesty$ is equal to $0$. Combining this inequality with \eqref{est:z-bounded-by-x} and, by taking into account Lemma \ref{lem:subd-Hk-bound}, we obtain (instead of \eqref{subd:inq})
\begin{equation*}
\opnorm{D^{k+1}} \leq \ \Csubx \left( \left\lVert x^{k+1} - x^{k} \right\rVert + \left\lVert x^{k} - x^{k-1} \right\rVert + \left\lVert x^{k-1} - x^{k-2} \right\rVert \right) \  \forall k \geq 2.
\end{equation*}
Consequently, for every $k \geq 3$ we have that
\begin{align*}
\E_{k-2} - \E_{k+1} & = \F_{k-2} - \F_{k-1} + \F_{k-1} - \F_{k} + \F_{k} - \F_{k+1} \\
& \geq \dfrac{\Cdecxx}{4} \left( \left\lVert x^{k-1} - x^{k-2} \right\rVert ^{2} + \left\lVert x^{k} - x^{k-1} \right\rVert ^{2} + \left\lVert x^{k+1} - x^{k} \right\rVert ^{2} \right) \\
& \geq \dfrac{\Cdecxx}{12} \left( \left\lVert x^{k-1} - x^{k-2} \right\rVert + \left\lVert x^{k} - x^{k-1} \right\rVert + \left\lVert x^{k+1} - x^{k} \right\rVert \right) ^{2} \\
& \geq \dfrac{\Cdecxx}{12 \Csubx^{2}} \opnorm{D^{k+1}}^{2} .
\end{align*}

Let $\varepsilon > 0$ be such that \eqref{Loja:uniform} is fulfilled and $k_{0} \geq 3$ such that  $\left( x^{k+1}, z^{k+1}, y^{k+1} \right)$ belongs to the open ball $\Ball \left( \left( \widehat{x} , \widehat{z} , \widehat{y} \right) , \varepsilon \right)$ for every $k \geq k_{0}$. 
Then \eqref{Loja:uniform} implies that for every $k \geq k_{0}$
\begin{equation*}
\E_{k-2} - \E_{k+1} \geq \Crec \E_{k+1} , \quad \textrm{ where } \quad \Crec := \dfrac{\Cdecxx}{12 C_{L}^{2} \Csubx^{2}},
\end{equation*}
which is the key inequality for deriving  convergence rates, as we have seen above.
\end{rmk}

{\bf Acknowledgements.} The authors are thankful to Ern\"o Robert Csetnek (University of Vienna) and to two anonymous reviewers valuable comments which improved the quality of the paper.

%%%%%%%%%%%%%%%%%%%%%%%%%%%%%%%%%%%%%%%%%%%%%%%%%%%%%%%%%%%%%%%%%%%%%%%%%%%%%%%%%%%%%%%%%%%%%%%%%%%%%%%
%
%
%		REFFERENCES
%
%
%%%%%%%%%%%%%%%%%%%%%%%%%%%%%%%%%%%%%%%%%%%%%%%%%%%%%%%%%%%%%%%%%%%%%%%%%%%%%%%%%%%%%%%%%%%%%%%%%%%%%%%


\begin{thebibliography}{9}


%%	AAAAAAAAAAAAAAAAAAAAAAAAAAAAAA
	
\bibitem{Ames-Hong}
\textbf{B. Ames, M. Hong}. 
\emph{Alternating direction method of multipliers for penalized zero-variance discriminant analysis}.
Computational Optimization and Applications  64(3),  725--754 (2016)

\bibitem{Attouch-Bolte}
\textbf{H. Attouch, J. Bolte}. 
\emph{On the convergence of the proximal algorithm for nonsmooth functions involving analytic features}. 
Mathematical Programming  116(1),  5--16 (2009)

\bibitem{Attouch-Bolte-Redont-Soubeyran}
\textbf{H. Attouch, J. Bolte, P. Redont, A. Soubeyran}. 
\emph{Proximal alternating minimization and projection methods for nonconvex problems: An approach based on the Kurdyka–Łojasiewicz inequality}. 
Mathematics of Operations Research  35(2),  438--457 (2010)

\bibitem{Attouch-Bolte-Svaiter}
\textbf{H. Attouch, J. Bolte, B. F. Svaiter}. 
\emph{Convergence of descent methods for semi-algebraic and tame problems: proximal algorithms, forward-backward splitting, and regularized Gauss-Seidel methods}. 
Mathematical Programming  137(1–2),   91--129 (2013)

%%	BBBBBBBBBBBBBBBBBBBBBBBBBBBBBB

\bibitem{Banert-Bot-Csetnek}
\textbf{S. Banert, R. I. Bo\c{t}, E. R. Csetnek}. 
\emph{Fixing and extending some recent results on the ADMM algorithm}. arXiv:1612.05057 (2016)

\bibitem{Beck}
\textbf{A. Beck}.
\emph{First-Order Methods in Optimization}.
MOS-SIAM Series on Optimization. SIAM, Philadelphia (2017)

\bibitem{Bolte-Daniilidis-Lewis}
\textbf{J. Bolte, A. Daniilidis, A. Lewis}. 
\emph{The Łojasiewicz inequality for nonsmooth subanalytic functions with applications to subgradient dynamical systems}. 
SIAM Journal on Optimization  17(4),  1205–1223 (2006)

\bibitem{Bolte-Daniilidis-Lewis-Shiota}
\textbf{J. Bolte, A. Daniilidis, A. Lewis, M. Shiota}. 
\emph{Clarke subgradients of stratifiable functions}. 
SIAM Journal on Optimization  18(2),  556–572 (2007)

\bibitem{Bolte-Daniilidis-Ley-Mazet}
\textbf{J. Bolte, A. Daniilidis, O. Ley, L. Mazet}. 
\emph{Characterizations of Łojasiewicz inequalities: subgradient flows, talweg, convexity}. 
Transactions of the American Mathematical Society 362(6),  3319--3363 (2010)

\bibitem{Bolte-Sabach-Teboulle}
\textbf{J. Bolte, S. Sabach, M. Teboulle}. 
\emph{Proximal alternating linearized minimization for nonconvex and nonsmooth problems}. 
Mathematical Programming  146(1),  459--494 (2014)

\bibitem{Bolte-Sabach-Teboulle:MOOR}
\textbf{J. Bolte, S. Sabach, M. Teboulle}. 
\emph{Nonconvex Lagrangian-based optimization: monitoring schemes and global convergence}.
Mathematics of Operations Research 43(4), 1210--1232 (2018)

\bibitem{Bot-Csetnek-iADMM}
\textbf{R. I. Bo\c{t}, E. R. Csetnek}. 
\emph{An inertial alternating direction method of multipliers}. 
Minimax Theory and its Applications 71(3),  29--49 (2016)

\bibitem{Bot-Csetnek}
\textbf{R. I. Bo\c{t}, E. R. Csetnek}. 
\emph{An inertial Tseng's type proximal algorithm for nonsmooth and nonconvex optimization problems}. Journal of Optimization Theory and Applications  171(2),  600--616 (2016)

\bibitem{Bot-Csetnek-ADMM}
\textbf{R. I. Bo\c{t}, E. R. Csetnek}. 
\emph{ADMM for monotone operators: convergence analysis and rates}. 
Advances in Computational Mathematics 45(1), 327--359 (2019)

\bibitem{Bot-Csetnek-Heinrich}
\textbf{R. I. Bo\c{t}, E. R. Csetnek, A. Heinrich}. 
\emph{A primal-dual splitting algorithm for finding zeros of sums of maximal monotone operators}. 
SIAM Journal on Optimization  23(4),  2011--2036 (2013)

\bibitem{Bot-Csetnek-Laszlo}
\textbf{R. I. Bo\c{t}, E. R. Csetnek, S. C. Laszlo}. 
\emph{An inertial forward-backward algorithm for the minimization of the sum of two nonconvex functions}. EURO Journal on Computational Optimization  4(1),  3--25 (2016)

\bibitem{Boyd-et.al}
\textbf{S. Boyd, N. Parikh, E. Chu, B. Peleato, J. Eckstein}. 
\emph{Distributed optimization and statistical learning via the alternating direction method of multipliers}. Foundations and Trends in Machine Learning  3(1), 1--122 (2010)

%%	CCCCCCCCCCCCCCCCCCCCCCCCCCCCCC

\bibitem{Chambolle-Pock}
\textbf{A. Chambolle, T. Pock}. 
\emph{A first-order primal-dual algorithm for convex problems with applications to imaging}. 
Journal of Mathematical Imaging and Vision  40(1),  120--145 (2011)

\bibitem{Combettes-Vu-quasi-Fejer}
\textbf{P. L. Combettes, B. C. V\~{u}}. 
\emph{Variable metric quasi-Fejér monotonicity}. 
Nonlinear Analysis: Theory, Methods and Applications  78,  17--31 (2014)

\bibitem{Combettes-Wajs}
\textbf{P. L. Combettes, V. R. Wajs}. 
\emph{Signal recovery by proximal forward-backward splitting}. 
Multiscale Modeling and Simulation  4(4),  1168--1200 (2005)

\bibitem{Condat}
\textbf{L. Condat}. 
\emph{A primal-dual splitting method for convex optimization involving Lipschitzian, proximable and linear composite terms}. 
Journal of Optimization Theory and Applications  158(2),  460--479 (2013)

\bibitem{Cui-Li-Sun-Toh}
\textbf{Y. Cui, X.D. Li, D.F. Sun, K.C. Toh}. 
\emph{On the convergence properties of a majorized ADMM for linearly constrained convex optimization problems with coupled objective functions}. 
Journal of Optimization Theory and Applications 169, 1013--1041 (2016)

%%	FFFFFFFFFFFFFFFFFFFFFFFFFFFFFF

\bibitem{Fazel-Pong-Sun-Tseng}
\textbf{M. Fazel, T.K. Pong, D.F. Sun, P. Tseng}. 
\emph{Hankel matrix rank minimization with applications to system identification and realization}. 
SIAM Journal on Matrix Analysis and Applications 34, 946--977 (2013)

\bibitem{Fortin-Glowinski}
\textbf{M. Fortin, R. Glowinski}. 
\emph{On decomposition-coordination methods using an augmented Lagrangian}.
in: \textbf{M. Fortin and R. Glowinski (eds.)}, Augmented Lagrangian Methods: Applications to the Solution of Boundary-Value Problems, North-Holland, Amsterdam (1983)

\bibitem{Frankel-Garrigos-Peypouquet}
\textbf{P. Frankel, G. Garrigos, J. Peypouquet}. 
\emph{Splitting methods with variable metric for Kurdyka-\Loja functions and general convergence rates}. Journal of Optimization Theory and Applications 165(3), 874--900 (2015)

%%	GGGGGGGGGGGGGGGGGGGGGGGGGGGGGG

\bibitem{Gabay}
\textbf{D. Gabay}. 
\emph{Applications of the method of multipliers to variational inequalities}.
in: \textbf{M. Fortin and R. Glowinski (eds.)} 
\emph{Augmented Lagrangian Methods: Applications to the Solution of Boundary-Value Problems}, North-Holland, Amsterdam (1983)
%
\bibitem{Gabay-Meicer}
\textbf{D. Gabay, B. Mercier}. 
\emph{A dual algorithm for the solution of nonlinear variational problems via finite element approximation}. 
Computers and Mathematics with Applications  2(1),  17--40 (1976)

\bibitem{Guo-Han-Wu} 
\textbf{K. Guo, D.R. Han, T.T. Wu}. 
\emph{Convergence of alternating direction method for minimizing sum of two nonconvex functions with linear constraints}. 
International Journal of Computer Mathematics 94(8), 1653-1669 (2017)


%%	HHHHHHHHHHHHHHHHHHHHHHHHHHHHHH

\bibitem{Hare-Sagastizabal}
\textbf{W. Hare, C. Sagastiz\'{a}bal}. 
\emph{Computing proximal points of nonconvex functions}. 
Mathematical Programming  116(1-2),  221--258 (2009)

\bibitem{Hong-Lou-Razaviyayn}
\textbf{M. Hong, Z.Q. Luo, M. Razaviyayn}. 
\emph{Convergence analysis of alternating direction method of multipliers for a family of nonconvex problems}. 
SIAM Journal on Optimization  26(1),  337--364 (2016)


\bibitem{Hong-Luo}
\textbf{M. Hong, Z.-Q. Luo}. 
\emph{On the linear convergence of the alternating direction method of multipliers}. 
Mathematical Programming  162,  165--199 (2017)

%%	KKKKKKKKKKKKKKKKKKKKKKKKKKKKKK

\bibitem{Kurdyka}
\textbf{K. Kurdyka}. 
\emph{On gradients of functions definable in o-minimal structures}. 
Annales de l’Institut Fourier  48,  769--783 (1998)

%%	LLLLLLLLLLLLLLLLLLLLLLLLLLLLLL

\bibitem{Lewis-Malick}
\textbf{A. Lewis, J. Malick.} 
\emph{Alternating projection on manifolds}. 
Mathematics of Operations Research 33(1),  216--234 (2008)

\bibitem{Li-Pong}
\textbf{G. Li, T. K. Pong}. 
\emph{Global convergence of splitting methods for nonconvex composite optimization}. 
SIAM Journal on Optimization  25(4),  2434--2460 (2015)

\bibitem{Lin-Liu-Li}
\textbf{Z. Lin, R. Liu, H. Li}. 
\emph{Linearized alternating direction method with parallel splitting and adaptive penalty for separable convex programs in machine learning}. 
Machine Learning  99(2), 287--325 (2015)

\bibitem{Liu-Shen-Gu}
\textbf{Q. Liu, X. Shen, Y. Gu}. 
\emph{Linearized ADMM for non-convex non-smooth optimization with convergence analysis}. arXiv:1705.02502 (2017)

\bibitem{Lojasiewicz}
\textbf{S. \L ojasiewicz}. 
\emph{Une propri\'et\'e topologique des sous-ensembles analytiques r\'eels}. 
in: Colloques internationaux du C.N.R.S.: Les \'equations aux d\'eriv\'ees partielles, Paris (1962),
\'Editions du Centre National de la Recherche Scientifique, Paris, 87--89 (1963)

%%	MMMMMMMMMMMMMMMMMMMMMMMMMMMMMM

\bibitem{Mordukhovich}
\textbf{B. Mordukhovich}. 
\emph{Variational Analysis and Generalized Differentiation, I: Basic Theory, II: Applications.}. Springer, Berlin (2006)

\bibitem{Moreau}
\textbf{J. Moreau}. 
\emph{Fonctions convexes duales et points proximaux dans un espace hilbertien}. 
Comptes Rendus de l’Acad\'emie des Sciences (Paris), S\'erie A  255,  2897--2899 (1962)


%%	OOOOOOOOOOOOOOOOOOOOOOOOOOOOOO

\bibitem{Ouyang-Chen-Lan-Pasiliao}
\textbf{Y. Ouyang, Y. Chen, G. Lan, E. Pasiliao, Jr.}. 
\emph{An accelerated linearized alternating direction method of multipliers}. 
SIAM Journal on Imaging Sciences 8(1), 644--681 (2015)

%%	RRRRRRRRRRRRRRRRRRRRRRRRRRRRRR

\bibitem{Ren-Lin}
\textbf{X. Ren, Z. Lin}. 
\emph{Linearized alternating Ddirection method with adaptive penalty and warm starts for fast solving transform invariant low-Rank textures}. International Journal of Computer Vision  104(1), 1--14  (2013)

\bibitem{Rockafellar-Wets}
\textbf{R. T. Rockafellar, R. J.-B. Wets}. 
\emph{Variational Analysis}, 
Fundamental Principles of Mathematical Sciences  317, Springer, Berlin (1998)

%%	SSSSSSSSSSSSSSSSSSSSSSSSSSSSSS

\bibitem{Shefi-Teboulle}
\textbf{R. Shefi, M. Teboulle}. 
\emph{Rate of convergence analysis of decomposition methods based on the proximal method of multipliers for convex minimization}. 
SIAM Journal on Optimization  24(1),  269--297 (2014)

\bibitem{Sun-Toh-Yang}
\textbf{D. Sun, K.-C. Toh, L. Yang}. 
\emph{A convergent 3-block semi-proximal alternating direction method of multipliers for conic programming with 4-type constraints}. SIAM Journal on Optimization 25(2), 882--915 (2015)

%%	VVVVVVVVVVVVVVVVVVVVVVVVVVVVVV

\bibitem{Vu}
\textbf{B. C. V\~{u}}. 
\emph{A splitting algorithm for dual monotone inclusions involving cocoercive operators}. 
Advances in Computational Mathematics  38(3),  667--681 (2013)

%%	WWWWWWWWWWWWWWWWWWWWWWWWWWWWWW

\bibitem{Wang-Xu-Xu}
\textbf{Y. Wang, Z. Xu, H.-K. Xu}. 
\emph{Convergence of Bregman alternating direction method with multipliers for nonconvex composite problems}. UCLA CAM Report 15-62, UCLA  (2015)

\bibitem{Wang-Yin-Zeng}
\textbf{Y. Wang, W. Yin, J. Zeng}. 
\emph{Global convergence of ADMM in nonconvex nonsmooth optimization}. 
Journal of Scientific Computing 78(1), 29--63 (2019)
%%	YYYYYYYYYYYYYYYYYYYYYYYYYYYYYY

\bibitem{Xu-Wu}
\textbf{M.H. Xu and T. Wu}. 
\emph{A class of linearized proximal alternating direction methods}. 
Journal of Optimization Theory and Applications 151(2), 321--337 (2011)
	
%% YYYYYYYYYYYYYYYYYYYYYYYYYYYYYY

\bibitem{Yang-Pong-Chen}
\textbf{L. Yang, T. K. Pong, X. Chen}. 
\emph{Alternating direction method of multipliers for a class of nonconvex and nonsmooth problems with applications to background/foreground extraction}. 
SIAM Journal on Imaging Sciences, 10(1), 74--110 (2017)

\bibitem{Yang-Yuan}
\textbf{J. Yang, X. Yuan}. 
\emph{Linearized augmented Lagrangian and alternating direction methods for nuclear norm minimization}. 
Mathematics of Computation 82, 301--329 (2013)
\end{thebibliography}
\end{document}